\documentclass[reqno]{amsart}
\usepackage{hyperref}
\usepackage{amssymb}
\usepackage{mathrsfs}
\usepackage{esint}
\usepackage{graphicx}
\usepackage{wrapfig}
\usepackage{float}
\usepackage{bm}
\usepackage{inputenc}
\usepackage{amsmath}
\usepackage{mathtools}
\numberwithin{equation}{section}
\newtheorem{theorem}{Theorem}[section]
\newtheorem{lemma}[theorem]{Lemma}
\newtheorem{proposition}[theorem]{Proposition}
\newtheorem{remark}[theorem]{Remark}

\numberwithin{equation}{section}
\newtheorem{thm}{Theorem}[section]
\newtheorem{lema}[thm]{Lemma}

\title[]%\hfilneg Hydrodynamic limit of Boltzmann equation]
{Acoustic limit of Boltzmann equations for  gas mixture}

\author[
]{Gaofeng Wang,  Tianfang Wu  and Linjie Xiong}

\address{G.F. Wang. School of Mathematics and Statistics, Nantong University,
	Nantong 226019, P. R. China.}
\email{gfwang@ntu.edu.cn}

\address{T.F. Wu. (corresponding author). Wenzhou Business College, Wenzhou 325000, P. R. China.}
\email{20249237@wzbc.edu.cn}

\address{L.J.Xiong. School of Mathematical Sciences, Hunan University, Changsha, 410082, P.R.China.}
\email{xlj@hnu.edu.cn}
\thanks{}

\begin{document}	
	
	\maketitle
	\begin{abstract}
		In this paper, we study the hydrodynamic and acoustic limit from Boltzmann equations for two species gas mixture
		with potential $\gamma \in \left(-3, 1\right]$. % in the whole space $(x \in \mathbb{R}^3)$.
		Here the particle masses are different which derives to the loss of symmetry to the linearized
		collision operator. %This paper resolves it precisely  by using a framework based on vector-valued functions.
		We construct the hydrodynamic limit for two species based on the Hilbert expansion method when the Knudsen number is small. The key observation is the precise properties of the linearized collision operators, including the extra operators due to the different particle masses $(m^A \neq m^B)$.   In additional, the acoustic limit of the Boltzmann equations for gas mixtures is rigorously justified by assuming the strength of the initial data depends on the Knudsen number. %Existing literature on fluid dynamic limits for the Boltzmann equation predominantly concerns single-species gases. For systems involving multiple species, the particles are typically assumed to have equal masses.
		\\
		
		 \vspace{5pt}
	\end{abstract}
	
	\allowdisplaybreaks

	\section{Introduction and main result}
	
	\subsection{Introduction}
	We consider the Boltzmann equations for two species particles $A, B$ as follows
	\begin{equation}\label{MAMAMAMAzhuyaotuimox}
		\left\{
		\begin{gathered}
			\partial_tF^A+v\cdot \nabla_xF^A=\frac{1}{\varepsilon}\Big[Q^{AA}(F^A,F^A)+Q^{AB}(F^A,F^B)\Big],
			\\
			\partial_tF^B+v\cdot \nabla_xF^B=\frac{1}{\varepsilon}\Big[Q^{BA}(F^B,F^A)+Q^{BB}(F^B,F^B)\Big].
		\end{gathered}
		\right.
	\end{equation}
	In the following, the particle species is denoted by the Greek letters $\alpha$ and $\beta$ ($\alpha,\beta\in\{A,B\}$). The density distribution function for species $\alpha$ is given by $F^\alpha(t,x,v)$, which is defined at time $t>0$, position $x=(x_1,x_2,x_3)$ with the velocity $v=(v_1,v_2,v_3)$. Particle mass of species $\alpha$ is denoted by $m^\alpha$.  The parameter $ \varepsilon$ is called Knudsen number, which is proportional to the mean free path.
	The collision operator  for $\alpha$-$\beta$ particle pairs takes the form
	$$ Q^{\alpha \beta}(F^{\alpha},F^{\beta})=\int_{\mathbb{R}^3}\int_{\mathbb{S}^2}B^{\alpha \beta}(|v-v_*|,\theta) \Big[F^{\alpha}(v') F^{\beta}(v'_*)-F^{\alpha}(v)F^{\beta}(v_*) \Big] d\omega dv_*. $$
	Here $v, v_*$
	and $v', v'_*$  are the velocities before and after the collision corresponding to the species $\alpha$ and $\beta$, which obey the laws of momentum conservation and energy conservation:
	\begin{equation*}
		m^{\alpha} v' +m^{\beta}v_*' =m^{\alpha} v +m^{\beta} v_*, \quad
		m^{\alpha}|v'|^{2}+m^{\beta}|v_*'|^2 =m^{\alpha}|v|^2+m^{\beta}|v_*|^2.
	\end{equation*}
	Let $\omega\in\mathbb{S}^2$ and $\cos\theta=\frac{\omega\cdot(v-v_*)}{|v-v_*|}$. The velocity $ v', v'_* $ after the collision are expressed as	
	\begin{equation*}
		v'=v-\frac{2m^{\beta}}{m^\alpha+m^\beta}[(v-v_*)\cdot \omega]\omega,
		\quad
		v'_{*}=v_{*}+\frac{2m^{\alpha}}{m^\alpha+m^\beta}[(v-v_*)\cdot \omega]\omega.
	\end{equation*}
	
	The collision kernels $ B^{\alpha \beta}(|v-v_*|,\theta)$ is assumed to satisfy the following assumptions:
	
	\noindent (1) The symmetry of collision kernels
	\begin{equation*}
		B^{\alpha \beta}(|v-v_*|,\theta) =B^{\beta \alpha}(|v_*-v|,\theta) \,\quad  \alpha,  \beta\in \{A,B\}.
	\end{equation*}
	
	\noindent (2) The collision kernels are decomposed into
	\begin{equation*}
		B^{\alpha \beta}(|v-v_*|,\theta) =\Phi^{\alpha \beta} (|v-v_*|) b^{\alpha \beta}(\cos \theta)\,\quad  \alpha,  \beta\in \{A,B\}.
	\end{equation*}
	Here the kinetic part takes
	\begin{equation*}
		\Phi^{\alpha \beta}(|v-v_*|) =C_{\alpha \beta}^{\Phi} |v-v_*|^{\gamma} , \,  C_{\alpha \beta}^{\Phi} >0 , \, \gamma \in
		( -3,1 ],
	\end{equation*}
	where $\gamma=$
	and there exist a constant $C_{b}  >0 $, for $\forall \alpha,  \beta\in \{A,B\}$ and $\theta \in [0,\pi]$, such that
	\begin{equation*}
		0\leq b^{\alpha \beta}(\cos \theta)\leq C_{b}  |\cos \theta |.
	\end{equation*}
	It means that the angular part satisfies Grad’s angular cutoff assumption \cite{[61]Grad1958TG}.
	
	The derivation of the fluid dynamical equations from kinetic theory can be traced back to the seminal contributions of Maxwell \cite{[37]Maxwell1867JCM} and Boltzmann \cite{[9]Boltzmann1872AWW}. Although the formal derivations \cite{[4]BGLJSP 1991} have been established at various conceptual levels, a comprehensive mathematical justification remains challengeable which aligns with the objective of Hilbert's sixth problem \cite{[23]Hilbert}: to construct a unified mathematical framework for gas dynamics across different descriptive scales.
	
	Over the past four decades, substantial research has addressed the well-posedness and fluid limits of the Boltzmann equation. There exist two main methodologies in the classical solutions framework: one involves the spectral analysis of the semigroup generated by the linearized Boltzmann
	equation, refer to \cite{[6]BSU1991}, \cite{[PG]Liu2004}, \cite{[PG]Liu2006}, \cite{[YU2010]}, while the other lies in developing a framework for $L^2$ energy estimates, see \cite{[DYZHD]VPB}, \cite{[ivR]Guo2002}, \cite{[iJMP]Guo}, \cite{[19]Guo2010ARMA}. %While semigroup methods struggle with complex Boltzmann equations, energy-based approaches often require small initial data constraints. 	
	Guo developed the $L^2-L^{\infty}$ method \cite{[21]Guo2009KRM}, \cite{[22]Guo2010CPAM}, which employs Hilbert expansion and reduce to the system of remainder terms. It is sufficient to establish the uniform of the remainders via $L^2$ energy and $L^{\infty}$ characteristic line procedures with weighting techniques. Within this framework, Guo, Huang, and Wang \cite{[20]Guo2021ARMA} studied the fluid limit of the single Boltzmann equation for the hard-sphere model in a half-space with specular reflection. Later, the hydrodynamic limit of the same equation with the Maxwell or complete diffuse reflection were examined in \cite{[29]Jiang2021}, \cite{[30]Jiang2021}. Both the kinetic boundary layer and the fluid boundary layer appeared to match the boundary conditions and the compressible Euler equation in the interior.
	
	Compared with the extensive study of single-gas Boltzmann equations, the study on the Boltzmann equation of gas mixtures has also attracted attention. These studies focus on the case of the same particle mass, that is $m^A=m^B$.  Wang \cite{YJWangSIMA} investigated diffusive limits in VPB systems with periodic boundaries, later quantifying decay rates in such systems \cite{[49]WangJDE2013}.  Aoki, Bardos, and Takata \cite{[1]Aoki2003JSP} studied Knudsen layer existence for Boltzmann equations with zero macroscopic flow, later generalized by Bardos and Yang \cite{[7]Bardos2012CPDE} for arbitrary flow velocities.
	Guo's foundational work \cite{[17]Guo2003Invention} established the Euler-Maxwell limit for hard-sphere VMB systems. Duan and Liu \cite{[DL]VPB} examined one-dimensional Euler-Poisson limits from Boltzmann system of mixtures, while Jiang, Lei, and Zhao \cite{[2025X]Jiang} explored Euler-Maxwell limits for $-3<\gamma\leq1$, $(m_{\pm}=e_{\pm}=1)$. It is worthwhile to mention Fang and Qi's recent study \cite{[FQ]ARC} on the hydrodynamic limit from Boltzmann equations for gas mixtures to a two-fluid macroscopic system. % According to the existing literature,
	There are relatively few results on the fluid limits of the Boltzmann equation system for gas mixtures of particles with unequal masses.
	
	\begin{figure}{}
\centering
	\includegraphics[width=11.5cm,height=6.5cm]{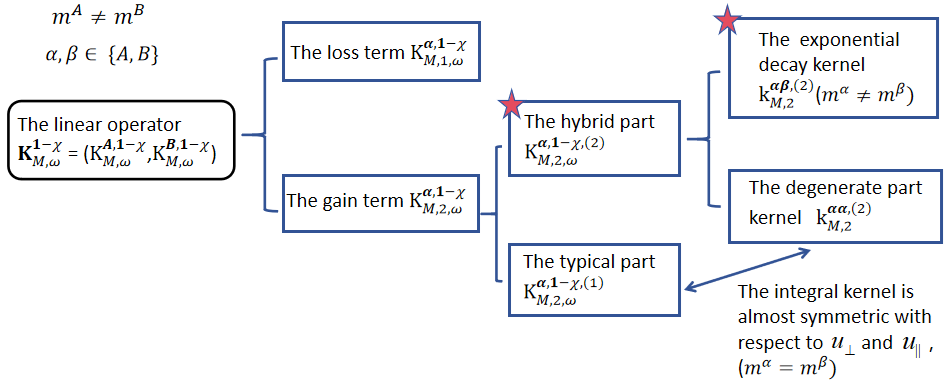}
	\caption{The collision between nitrogen and oxygen molecules }
	\end{figure}
	
	%Earth's atmosphere is composed of 78\% nitrogen $(N_2)$ and 21\% oxygen $(O_2)$, with a molecular mass ratio of approximately 7:8. For example, when investigating the greenhouse effect, the analysis of carbon dioxide $(CO_2)$ distribution and concentration variations requires quantifying its mass ratio relative to nitrogen—yielding a value of 7:11.
	
	However, it is nature to study the combined system with different mass particles $m^A \neq m^B$. In \cite{[60]Briant2016ARM}, Briant and Daus studied the stability and the time decay rate of the solution around the equilibrium of the multi-species Boltzmann system. Although the compressible Euler system for gas mixture has been established in reference \cite{[50]Wu2023JDE}, %
	the estimates of the linear operators were brought from those of single-species Boltzmann equation without proof. To study the hydrodynamic limit of the system \ref{MAMAMAMAzhuyaotuimox}, one should give some more precise analysis of the linear operators. In particular, the main observation here is the estimation of the operator $K_{M,2,w}^{\alpha,1-\chi}$ for the Boltzmann system of unequal-mass mixtures, (see Lemma \ref{LeMK2ker4444}), which reveals that it can be divided into two parts
	\begin{equation}
		K_{M,2,w}^{\alpha,1-\chi} = \underbrace{K_{M,2,w}^{\alpha,1-\chi,(1)}}_{\text{Typical part}} + \underbrace{K_{M,2,w}^{\alpha,1-\chi,(2)}}_{\text{Hybrid part}},
	\end{equation}
	where ``Typical part" shares the properties analogous to that of the Boltzmann equations in the case of identical gas molecules $(m_A=m_B)$.  ``Hybrid part" consists of a degenerate part $(m^{\alpha}=m^{\beta})$ and a part that exhibit exponential decay $(m^{\alpha} \neq m^{\beta})$. The exponential decay part is the main difference arising from the different mass particles $m^A \neq m^B$.  %Consequently,  \cite{[50]Wu2023JDE}'s approach of substituting $K_{M,2,w}^{\alpha,c}$ with $K_{M,2,w}^{\alpha,c,(1)}$ remains valid.
	The estimate of $K_{M,2,w}^{\alpha,1-\chi}$ provides critical theoretical results for spectral analysis of the linearized Boltzmann operator in multi-component gas systems with mass disparity. %which enriches the theoretical foundation for fluid limits of unequal-mass gas mixtures.
	This precise structural analysis offers the fundamental estimates of the linear operator in the unequal mass case, we believed that these estimates are helpful to get the pointwise estimates of classical solutions to the Boltzmann equation for the mixed gases. This will be studied in the forthcoming paper.
%	\begin{figure}
	%	\centering
	%	\includegraphics[width=15cm,height=6cm]{TPHB}
	%	\vspace{-0.5cm}
	%	\caption{The structural of operator $K_{M,2,w}^{\alpha,1-\chi}$.}
%	\end{figure}
	%has the following significant implications:
	%	\begin{enumerate}
		%		\item It provides technical support for pointwise estimates of classical solutions to the Boltzmann equation for mixed gases with unequal masses and their associated coupled models.
		
		%		\item It offers fundamental estimates for the linear collision operator of unequal-mass mixed gases within the framework of the $L^2$ energy method.
		
		%\item This work establishes a foundational lemma (Lemma \ref{LeMK2ker4444}), which
		%	\end{enumerate}

	\subsection{The properties of the collision operator.}
	
	For $\mathbf{F} = (F^A, F^B)^\top$,
	we introduce the bilinear Boltzmann collision operator $\mathscr{C}$ in vector form as follows:
	\begin{equation}\label{equation1.4}
		\mathscr{C} \mathbf{F} =
		\left(
		\begin{array}{cccc}
			Q^{AA}(F^A,F^A)+Q^{AB}(F^A,F^B)\vspace{3pt}\\
			Q^{BA}(F^B,F^A)+Q^{BB}(F^B,F^B)
		\end{array}
		\right).
	\end{equation}
	The operator $\mathscr{C}$ has the following properties \cite{[60]Briant2016ARM}:\\
	{\bf Collision invariant} A function $\mathbf{\Psi} = (\Psi^{A}, \Psi^{B})^\top$
	is called a collision invariant of the operator $\mathscr{C}$ with respect to the inner product in $(L^2_v(\mathbb{R}^3))^2$
	if it satisfies
	\begin{equation}\label{Collision invariants} \langle\mathscr{C}\mathbf{F}, \mathbf{\Psi} \rangle_{(L_v^2(\mathbb{R}^3))^2} = \sum_{\alpha,\beta \in {A,B}} \left\langle Q^{\alpha \beta}(F^{\alpha}, F^{\beta}), \Psi^{\alpha} \right\rangle_{L_v^2} = 0 \quad \text{for any vector } \mathbf{F}.
	\end{equation}
	It derives to
	\begin{equation}\label{mabasis}
		\mathbf{\Psi} \in \text{Span} \{\mathbf{e}_1, \mathbf{e}_2, v_1 \mathbf{m}, v_2 \mathbf{m}, v_3 \mathbf{m}, |v|^2 \mathbf{m}\},
	\end{equation}
	where $\mathbf{m} = (m^A, m^B)^{T}$
	and $\mathbf{e}_j$
	is the $j^{\text{th}}$
	unit vector in $\mathbb{R}^2$.\\
	{\bf H-theorem} The entropy function of the two-species Boltzmann equations  satisfies
	\begin{equation}\label{equation1.5}
		\left \langle
		\mathscr{C}\mathbf{F},  \log \mathbf{F}
		\right \rangle_{(L_v^2(\mathbb{R}^3))^2}
		= \sum_{\alpha,\beta\in \{A,B\}}\left\langle Q^{\alpha  \beta}(F^{\alpha},F^{\beta}) , ~~\log F^{\alpha} \right\rangle_{L_v^2(\mathbb{R}^3)}   \leq 0,
	\end{equation}
	and this equality holds if and only if
	\begin{equation}\label{dyjhlibsxF}
		\mathbf{F} 	=
		\left(
		\begin{array}{cccc}
			\frac{n^A(m^A)^{3/2} }{(2\pi \theta)^{3/2}}e^{-\frac{m^A |v-\mathbf{u}|^2}{2\theta}} \vspace{3pt}\\
			\frac{n^B(m^B)^{3/2} }{(2\pi \theta)^{3/2}}e^{-\frac{m^B |v-\mathbf{u}|^2}{2\theta}}
		\end{array}
		\right)=: \left(
		\begin{array}{cccc}  \mu^A \\ \mu^B\end{array}
		\right),  		
	\end{equation}
	where $n^\alpha (\alpha=A,B)$ are the number density of particle $\alpha$, $u,\theta$ are the bulk velocity and temperature of the system.   	
	
	For later use, we also denote that $F =(F^A,F^B)^\top $ with $F^{\alpha}= \mu^{\alpha}+ \sqrt{\mu^{\alpha}}f^\alpha$. Thus,  for $\mathbf{f}=(f^A,f^B)^\top$, we yield the following linear operator
	\begin{eqnarray}\label{1DefinitionL}
		\mathbf{L} \mathbf{f}=\left(\begin{array}{lll} \mathbf{L}^{A}\mathbf{f}\\ \mathbf{L}^{B}\mathbf{f}\end{array}\right)=: -\left(\begin{array}{lll}
			\frac{1}{\sqrt{\mu^{A}}} \displaystyle{\sum_{\beta=A,B}} \left[Q^{ A \beta}(\mu^A,\sqrt{\mu^{\beta}} f^{\beta})+Q^{A \beta}(\sqrt{\mu^A}f^{A},\mu^\beta)\right]\\
			\frac{1}{\sqrt{\mu^{B}}} \displaystyle{\sum_{\beta=A,B}} \left[Q^{B \beta}(\mu^B,\sqrt{\mu^{\beta}} f^{\beta})+Q^{B \beta}(\sqrt{\mu^B}f^{B},\mu^\beta)\right]  \end{array}\right).
	\end{eqnarray}

	From the H-theorem, it is well known that $\mathbf{L} $ is a self-adjoint operator in $(L^2_v(\mathbb{R})^3)^2$,  and the kernel of $ \mathbf{L}$ is spanned by the following six dimensional vector functions (see \cite{[60]Briant2016ARM})
	\begin{equation}
		{\rm Ker}\mathbf{L}={\rm span}\{ \mathbf{X}_0,\mathbf{X}_1,\cdots, \mathbf{X}_5\},
	\end{equation}
	with
	\begin{equation}\label{abasisMACRO}
		\begin{split}
			&\mathbf{X}_0= \sqrt{\frac{\mu^{A}}{n^A }}  \mathbf{e}_1, \,\,\, \mathbf{X}_1= \sqrt{\frac{\mu^{B}}{n^B }} \mathbf{e}_2,\,\,\, \mathbf{X}_j=\frac{v_{j-1}-u^{j-1}}{\sqrt{\theta n}}
			\left(
			\begin{array}{ccc}
				m^A\sqrt{\mu^{A}}\\[2mm]
				m^B \sqrt{\mu^{B}}
			\end{array}
			\right) ~(j=2,3,4),\\
			& \mathbf{X}_5=\frac{1}{\sqrt{6n}}
			\left(
			\begin{array}{ccc}
				(\frac{m^A |v-\mathbf{u}|^2}{\theta}-3)\sqrt{\mu^{A}}\\[2mm]
				(\frac{m^B |v-\mathbf{u}|^2}{\theta}-3)\sqrt{\mu^{B}}
			\end{array}
			\right)~~~~{\rm where}~ n=n^A+n^B.
		\end{split} 		
	\end{equation} 	
	One can check that
	$
	\langle \mathbf{X}_i, \mathbf{X}_j\rangle
	=\delta_{ij},\, 0\leq i,j \leq 5$, where $\delta_{ij}$ is the Kronecker symbol.
	%\item The decomposition of $(L^2_v(\mathbb{R}^3))^2 $. According to the operator $\mathbf{L}$,
	The space $(L^2_v(\mathbb{R}^3))^2$ can be decomposed as ${\rm Ker}\mathbf{L} \oplus \mathcal{N}(\mathbf{L})^{\bot}  $, where $ \mathcal{N}(\mathbf{L})^{\bot}$ is the orthogonal set of ${\rm Ker}\mathbf{L}$.
	Then,  for any $ \mathbf{f}=(f^A,f^B)^\top \in (L^2_v(\mathbb{R}^3))^2 $, it can be decomposed  as
	\begin{gather}\label{MCMADCP}
		\mathbf{f} = \mathcal{P} \mathbf{f} + (\mathbf{I}- \mathcal{P}) \mathbf{f}, \quad {\rm with}\quad \mathcal{P} \mathbf{f}= \sum_{i=0}^5 \langle \mathbf{f}, \mathbf{X}_i\rangle \mathbf{X}_i,
	\end{gather}
	where $\mathcal{P} \mathbf{f}$ and $(\mathbf{I}-\mathcal{P})\mathbf{f} = \mathbf{f}-\mathcal{P}\mathbf{f} $ are the macro-micro parts of $\mathbf{f}$ respectively. The main properties of the operator $\mathbf{L}$ is list in the following.
	
	\noindent (The solvability of $\mathbf{L}$). For any $\mathbf{R} =(R^A,R^B)^\top \in (L^2_v(\mathbb{R}^3)^2)^2 $, there exists  a unique solution for
	$\mathbf{L}\mathbf{f}=\mathbf{R}$,
	if and only if $\mathbf{R} \in \mathcal{N}(\mathbf{L})^{\bot}$, that is,
	\begin{gather}\label{1.22}
		\langle\mathbf{R}, \mathbf{X}_i\rangle_{(L^2_v(\mathbb{R})^3)^2} =0, \quad {\rm for~ all}\quad i=0,\cdots,5.
	\end{gather}
	Moreover, the solution $\mathbf{f} $ could be formulated as
	\begin{eqnarray}\label{1.23}
		\mathbf{f}= \mathbf{L}^{-1} \mathbf{R}\in \mathcal{N}(\mathbf{L})^{\bot}.
	\end{eqnarray}
	
	\noindent(The coercivity of $\mathbf{L}$ ). There exists a positive number $c_0 > 0$, for any $\mathbf{f}=(f^A,f^B)^\top \in (L^2_v(\mathbf{R}^3))^2 $,
	it holds
	\begin{equation}\label{Coercivity1.22}
		\left \langle \mathbf{L}\mathbf{f}, \mathbf{f}	\right \rangle_v \geq c_0 |
		(\mathbf{I}-\mathcal{P})\mathbf{f} |_{\nu}^2.
	\end{equation}
	The proof of \eqref{Coercivity1.22} for hard potential case is in \cite{[60]Briant2016ARM}, and \cite{[52]WuPERP}, and for soft potential case one can refer \cite{[51]WuPERP}.
	\subsection{The main results} Following the ideas in \cite{[ininp]Guo2010CMP}, the compressible Euler system for mixtures is derived by the Hilbert expansion of the Boltzmann equation for gas mixture. The life-span of the fluid limit system together with the truncation of the Hilbert expansion is obtained when the strength of the perturbation of the initial data is sufficiently small but independent of the Knudsen number. The hydrodynamic limit of the Boltzmann equation of gas mixture is established by study the uniform estimates of the remainder term. Later, the acoustic limit is yielded when the strength of the perturbation of the initial data is dependent on the Kundsen number. It is proved that the rigorous justification of the acoustic limit of \ref{MAMAMAMAzhuyaotuimox} is valid for any time interval when the Knudsen number is small enough. More precisely, we firstly assume the solutions of \eqref{MAMAMAMAzhuyaotuimox} have the following Hilbert expansion form
	\begin{equation}\label{MAINEXPFPHIKKK}
		F_{\varepsilon}^\alpha= F^\alpha_0 + \sum_{k=1}^{5}\varepsilon^k F^\alpha_k + \varepsilon^3F_R^{\alpha}, \quad {\rm for}\quad \alpha \in\{A,B\}.
	\end{equation}
	Substituting it into \eqref{MAMAMAMAzhuyaotuimox}, and compare the coefficients about order of $\varepsilon$, we can obtain that
	\begin{equation}\label{ORDERZONG}
		\begin{aligned}
			& \text{Order ~$O(\varepsilon^{-1})$:} \qquad\quad ~~ \qquad\qquad ~\,~~\, 0= \displaystyle{\sum_{\beta=A,B}}Q^{ \alpha \beta}(F_0^{\alpha}, F_0^{\beta}), \\
			& \text{Order ~$O(1)$:} \qquad\quad ~~~~( \partial_t +v \cdot \nabla_x)F_0^\alpha =\displaystyle{\sum_{i'+i''=1}\sum_{\beta=A,B}}Q^{ \alpha \beta}(F_{i'}^{\alpha},F_{i''}^{\beta}),\\
			& \text{Order ~$O(\varepsilon)$:} \quad \quad\qquad ~~~~	(\partial_t +v \cdot \nabla_x)F_1^\alpha =\displaystyle{\sum_{i'+i''=2} \sum_{\beta=A,B}} Q^{\alpha \beta}(F_{i'}^{\alpha},F_{i''}^{\beta}),\\
			&.\,.\,.\\
			& \text{Order ~$O(\varepsilon^{i})$:} \,\quad\quad ~~~~~~
			(\partial_t +v \cdot \nabla_{x})F_{i}^{\alpha}=\displaystyle{\mathop{\sum}_{i'+i''=i+1\atop i',i''\geq0} \sum_{\beta=A,B}} Q^{ \alpha\beta}(F_{i'}^{\alpha},F_{i''}^{\beta}),   \quad{i}\in \mathbf{Z}^+,\quad {i}\geq2.
		\end{aligned}
	\end{equation}
	The remainder term satisfies
	\begin{equation}\label{reeqmain}
		\begin{gathered}
			(\partial_t +v \cdot \nabla_{x})F_R^{\alpha}-\frac{1}{\varepsilon}\displaystyle{\sum_{\beta=A,B}}(Q^{ \alpha \beta}(F_0^{\alpha}, F_R^{\beta})+Q^{ \alpha \beta}(F_R^{\alpha}, F_0^{\beta}))
			=\varepsilon^{2}\displaystyle{\sum_{\beta=A,B}}Q^{ \alpha \beta}(F_R^{\alpha}, F_R^{\beta})\\
			\hspace{4cm}+\displaystyle{\sum^{5}_{i=1}}\displaystyle{\sum_{\beta=A,B}}\varepsilon^{i-1}(Q^{ \alpha \beta}(F_i^{\alpha}, F_R^{\beta})+Q^{ \alpha \beta}(F_R^{\alpha}, F_i^{\beta}))+\varepsilon^{2}E^{\alpha},
		\end{gathered}
	\end{equation}
	where
	\begin{equation}\label{dyjjieshuzkep}
		\begin{aligned}
			E^{\alpha}:=&-(\partial_t +v \cdot \nabla_{x})F_{5}^{\alpha}+\mathop{\sum}_{i'+i'' \geq 6\atop  i',i'' \geq1}\displaystyle{\sum_{\beta=A,B}} \varepsilon^{i'+i''-6} Q^{ \alpha \beta}(F_{i'}^{\alpha},F_{i''}^{\beta}).
		\end{aligned}
	\end{equation}
	
	\noindent{\it Step 1}. From $\eqref{ORDERZONG}_1$ and H-theorem, we know that $\mathbf{F}_0=(F_0^A,F_0^B)^\top$ is a local  bi-Maxwellian as
	\begin{equation}\label{1dyjhlibsxF0}
		\mathbf{F}_0 =	
		\left(
		\begin{array}{cccc}
			\frac{n^A_{\delta}(m^A)^{3/2} }{(2\pi \theta_{\delta})^{3/2}}e^{-\frac{m^A |v-\mathbf{u}_{\delta}|^2}{2\theta_{\delta}}} \vspace{3pt}\\
			\frac{n^B_{\delta}(m^B)^{3/2} }{(2\pi \theta_{\delta})^{3/2}}e^{-\frac{m^B |v-\mathbf{u}_{\delta}|^2}{2\theta_{\delta}}}
		\end{array}
		\right)=: \left(
		\begin{array}{cccc}  \mu^A_{\delta} \\ \mu^B_{\delta}\end{array}
		\right).  		
	\end{equation}
	By imposing the initial data \eqref{nonlindt}, one has $ \mathbf{F}_0(0)= ( \mu^{A,\rm in}_{\delta},\mu^{B,\rm in}_{\delta})^\top $.
	The densities $n^A_{\delta}, n^B_{\delta}$, the bulk velocity $\mathbf{u}_{\delta}$ and the bulk temperature $\theta_{\delta}$ will be determined later with this initial data.
	
	%We denote them as $n^A_{\delta}, n^B_{\delta},\mathbf{u}_{\delta}, \theta_\delta$ since they depend on the strength of the initial perturbation.
	
	\noindent{\it Step 2}. From $\eqref{ORDERZONG}_2$, one could
	derives that
	\begin{eqnarray}\label{1operator1}
		\mathbf{L}_{\delta}\mathbf{f}_1 = \mathbf{R}_0,
	\end{eqnarray}
	where the operator $\mathbf{L}_{\delta}$ is defined in \eqref{1DefinitionL} by replacing $\mu^{\alpha}$ by $\mu^{\alpha}_{\delta}$ and  $\mathbf{R}_0= (-\frac{1}{\sqrt{\mu^A_{\delta}}} ( \partial_t +v \cdot \nabla_x)\mu^A_{\delta}, -\frac{1}{\sqrt{\mu^B_{\delta}}} ( \partial_t +v \cdot \nabla_x)\mu^B_{\delta})^\top$.
	%We solve the $F_k (k=0,\cdots,5)$ as follows.
	From the solvability of $\mathbf{L}_{\delta}$, one could derive the compressible Euler equations \eqref{EQF0EPSION} and
	get $n^A_{\delta},n^B_{\delta}, \mathbf{u}_{\delta},\theta_{\delta}$ in the time interval $[0,\tau^{\delta}]$ for $\tau^{\delta}=O(\frac{1}{\delta})$ with the initial data \eqref{nonlindt}.
	
	\noindent{\it Step 2}. To solve $ \eqref{dyjhlibsxF}_k$, we set $F_{k}^{\alpha}=f_{k}^{\alpha}\sqrt{\mu^{\alpha}_{\delta}}$ and $\mathbf{f}_{k}=(f_{k}^{A},f_{k}^{B})^\top $ for $k=1,2,\cdots, 5$. Here $\mathbf{f}_{k}$ will be decomposed to the macroscopic and microscopic parts as
	$$ \mathbf{f}_{k}= \mathcal{P} \mathbf{f}_{k}+(\mathbf{I}-\mathcal{P}) \mathbf{f}_{k}, $$
	with
	\begin{eqnarray}\label{macroscopicpartk}\mathcal{P}\mathbf{f}_k=\frac{n^{A}_{k}}{\sqrt{n^{A}_{\delta}}} \mathbf{X}_0^\delta +\frac{n^{B}_{k}}{\sqrt{n^{B}_{\delta}}}\mathbf{X}_1^\delta
		+\sum_{j=2,3,4}\sqrt{\frac{n_{\delta}}{\theta_{\delta}}} u_k^{j-1} \mathbf{X}_j^\delta
		+\sqrt{\frac{n_{\delta}}{6}}\frac{\theta_k}{\theta_{\delta}}\mathbf{X}_5^\delta. \end{eqnarray}
	Here $ \mathbf{X}_k^\delta (k=0,\cdots,5)$ comes from \eqref{abasisMACRO} by replacing $\mu^{\alpha}$ with $\mu^{\alpha}_\delta$.
	
	On the one hand, from the solvability of $\mathbf{L}_{\delta}$, \eqref{1.23} and \eqref{1operator1} give
	\begin{eqnarray}
		(\mathbf{I}- \mathcal{P}){f}_1 = \mathbf{L}_{\delta}^{-1} \mathbf{R}_0  \in \mathcal{N}(\mathbf{L}_{\delta})^{\bot}.
	\end{eqnarray}
	%	which is the microscopic part $(\mathbf{I}- \mathcal{P}){f}_1$.
	On the other hand, the equation of the macroscopic part $(n^A_1, n^B_1, \mathbf{u}_1, \theta_1)$ in \eqref{macroscopicpartk} could be derived from the solvability of the next order $\eqref{ORDERZONG}_3$, which is the linear compressible Euler equation \eqref{linarFLsys}. The existence and the uniform estimates of this part is established when the initial data $(n^{A,\rm in}_1, n^{B,\rm in}_1, \mathbf{u}^{\rm in}_1, \theta^{\rm in}_1)$ is specified. Thus, one could get the vector $\mathbf{f}_1$ as well as $\mathbf{F}_1$.
	
	Following the same process, we can get the terms $\mathbf{F}_k (k=1,\cdots,5)$ together with the uniform estimates in the interval $[0,\tau^{\delta}]$  when the initial data $(n^{A,\rm in}_k, n^{B,\rm in}_k, \mathbf{u}^{\rm in}_k, \theta^{\rm in}_k)$ are given.   Thus, one could get the vector $\mathbf{f}_k$ and $\mathbf{F}_k$ for $k=2,\cdots,5$.  We omit the details here.
	
	With these in hand, we derive the uniform estimates of the remainder $\mathbf{F}_R$ from the equation \eqref{reeqmain}. The hydrodynamic limit of the \eqref{MAMAMAMAzhuyaotuimox} could be established as follows.
	
	\begin{thm}\label{mainthemCPE}
		Let $(n_{\delta}^{A, \rm in}, n_{\delta}^{B, \rm in}, \mathbf{u}^{\rm in}_{\delta}, \theta^{\rm in}_{\delta})$ is given as \eqref{nonlindt}.
		Suppose that there exists a series of integers  $ s_0 >  s_1  > \cdots > s_4 > s_5  > 3,$ such that the initial data satisfy
		\begin{equation}\label{1000conditifs}
			\mathcal{E}^{\rm in}_{s_0}=:\|(\sigma^{A,\rm in}, \sigma^{B,\rm in}, \mathbf{u}^{\rm in},  \theta^{\rm in})\|_{H^{s_0}} <+\infty
		\end{equation}
		and
		\begin{equation}\label{1nkukthekint}
			\mathcal{E}^{\rm in}_{s_k}=:\|(n^{A,\rm in}_{k},  n^{B,\rm in}_{k},  \mathbf{u}^{\rm in}_{k},  \theta_{k}^{\rm in})\|_{H^{s_k}}<+\infty, \,\, {\rm for} \,\, k=1,2\cdots,5.
		\end{equation}
		Then, there exists some constant $\delta_0>0$ independent of $\varepsilon$ such that, for any $0<\delta<\delta_0$,  \eqref{ORDERZONG} admit $ \mathbf{F}_k  $ for all $0\leq t\leq \tau^{\delta}=O(\frac{1}{\delta})$  with $n^{A}_{k},  n^{B}_{k},  \mathbf{u}_{k},  \theta_{k} \in C([0,\tau^{\delta}), H^{s_k})$ for $k=0,\cdots, 5$.
		Let the initial data $ \mathbf{F}_\varepsilon^{\rm in} = (F_\varepsilon^{A,{\rm in}},F_\varepsilon^{B,{\rm in}})^\top$   of \eqref{MAMAMAMAzhuyaotuimox} is imposed as
		\begin{equation}\label{1initalMain}
			\mathbf{F}_\varepsilon|_{t=0} = \mathbf{F}_\varepsilon^{\rm in} =: \mathbf{F}_0(0)+\sum_{k=1}^5 \varepsilon^k \mathbf{F}_k(0)+ \varepsilon^3 \mathbf{F}_R(0),
		\end{equation}
		with $ F_\varepsilon^{A,{\rm in}}, F_\varepsilon^{B,{\rm in}}\geq 0$, where $\mathbf{F}_k(0) (k=0,\cdots,5)$ are the initial value of $\mathbf{F}_k$ from Step 1-Step 3.
		Let the global bi-Maxwellian $\mu_M$ is defined in \eqref{def1.35} with \eqref{globalMaxwellian} and the initial value of the remainder term $\mathbf{F}_R(0)=(F_{R}^{A}(0),F_{R}^{B}(0))^\top$ is assumed to satisfy
		\begin{equation}\label{remainKZ}
			\mathcal{E}^{\rm in}_R=:	\sum_{\alpha=A,B} \left\{ \left\|\frac{F_{R}^{\rm \alpha}(0)}{\sqrt{\mu^{\alpha,\rm in}_{\delta}}} \right\|_{L^{2}_{x,v} }+ \left\|\left \langle v \right\rangle^{2l}\frac{F_{R}^{\rm \alpha}(0)}{\sqrt{\mu^{\alpha}_{M}}} \right\|_{L^{\infty}_{x,v} }\right\}<\infty, \quad {\rm for} \quad l\geq \frac{25}{4}.
		\end{equation}
		Then, there exists some constant $\varepsilon_0>0$ such that, for $0<\varepsilon\leq\varepsilon_0$, the initial value problem \eqref{MAMAMAMAzhuyaotuimox}  with \eqref{1initalMain} admits a unique solution
		\begin{equation}\label{1Hilbertexpansion}
			\mathbf{F}_\varepsilon(t) =  \mathbf{F}_0(t)+\sum_{k=1}^5 \varepsilon^k \mathbf{F}_k(t)+ \varepsilon^3 \mathbf{F}_R(t),
		\end{equation}
		for all $0\leq t\leq \tau^{\delta}$, where the remainder term satisfies
		\begin{equation}\label{1MEstimate}
			\sup_{0\leq t \leq \tau^{\delta}}\sum_{\alpha=A,B} \left\{ \sqrt{\varepsilon}\left\|\frac{F_{R}^{\rm \alpha}}{\sqrt{\mu^{\alpha}_{\delta}}}(t) \right\|_{L^{2}_{x,v} }+ \varepsilon^2\left\|\left \langle v \right\rangle^{2l}\frac{F_{R}^{\rm \alpha}(t)}{\sqrt{\mu^{\alpha}_{M}}} \right\|_{L^{\infty}_{x,v} }\right\} \leq C.
		\end{equation}
		Furthermore, it implies that
		\begin{equation}\label{zjkytld}
			\sup_{0\leq t \leq \tau^{\delta}}\sum_{\alpha=A,B} \left\{ \left\|\frac{F_{\varepsilon}^{\rm \alpha}-\mu^{\alpha}_{\delta}}{\sqrt{\mu^{\alpha}_{\delta}}} (t)\right\|_{L^{2}_{x,v} }+ \left\|\left \langle v \right\rangle^{2l}\frac{F_{\varepsilon}^{\rm \alpha}(t)-\mu^{\alpha}_{\delta}(t)}{\sqrt{\mu^{\alpha}_{M}}} \right\|_{L^{\infty}_{x,v} }\right\} \leq C \varepsilon \rightarrow 0, \,\, {\rm as}\,\, \varepsilon \rightarrow 0.
		\end{equation}
	\end{thm}
	
	%	\begin{remark}We
		%		assume the expansion to be
		%		\begin{equation*}
			%			F_{\varepsilon}^\alpha= \sum_{i=0}^{l}\varepsilon^i F^\alpha_i + \varepsilon^rF_R^{\alpha}.
			%		\end{equation*}
		%		By analyzing the equation  \eqref{reeqmain} governing the remainder, we equate the coefficients of $\varepsilon$
		%		for both $Q^{ \alpha \beta}(F_R^{\alpha}, F_R^{\beta})$
		%		and  $E^{\alpha}$. To maintain balance, we select $l=2r-1$. Finally, we demonstrate why expanding at least five terms is necessary.		
		%	\end{remark}
	%	\begin{remark}We need to particularly emphasize that the existence time t of the aforementioned theorem can be chosen arbitrarily, which differs from the short duration assumed in \cite{[50]Wu2023JDE}.
		%	\end{remark}

	It should be pointed out that the compressible Euler limit could be derived in the time interval $[0,\tau^{\delta}]$ with $\tau^{\delta}=O(\frac{1}{\delta})$ regardless of the strength $\delta$. It is interesting to derive the acoustic system when the fluctuation amplitude $\delta$ satisfies $\delta \to 0 $ as $ \varepsilon \to 0$.  Throughout the paper, we assume that
	\begin{eqnarray}\label{1relation}
		\frac{\varepsilon}{\delta} \rightarrow 0, \quad {\rm as}\quad \varepsilon \rightarrow 0.
	\end{eqnarray}
	For instance, one can take
	$$ \delta=\varepsilon^{\tilde{m}} \quad {\rm for\,\, some}\quad 0<\tilde{m}<1.$$
	Similar to \cite{[22]Guo2010CPAM}, we suppose that the Boltzmann equations for gas mixture \eqref{MAMAMAMAzhuyaotuimox} is defined as
	\begin{equation}\label{defGepa}
		\mathbf{F}_{\varepsilon}=\bm\mu_0 + \delta \mathbf{G}_{\varepsilon},
	\end{equation}
	where $\mathbf{G}_\varepsilon =(G^A_\varepsilon, G^B_\varepsilon)^{\top}$ and  the scaled standard global bi-Maxwellian is denoted as
	\begin{equation}\label{defuoaf}
		\bm{\mu}_0 = \left(\begin{array}{ccc} \mu^A_0 \\[2mm] \mu^B_0 \end{array}\right)=\left(\begin{array}{ccc}
			(\frac{m^{A}}{2\pi})^{\frac{3}{2}} \exp\left(-\frac{m^{A}|v|^2}{2}\right) \\[2mm]
			(\frac{m^{B}}{2\pi})^{\frac{3}{2}}  \exp\left(-\frac{m^{B}|v|^2}{2}\right) \end{array}\right).
	\end{equation}
	
	On one hand, combing with \eqref{1Hilbertexpansion}, it could be shown that
	\begin{eqnarray*}
		\mathbf{G}_\varepsilon =\frac{1}{\delta}\left(\mathbf{F}_0 (t)- \bm{\mu}_0(t)\right) + \frac{\varepsilon}{\delta} \Big\{\sum_{k=1}^5 \varepsilon^{k-1}\mathbf{F}_k(t) +\varepsilon^2 \mathbf{F}_R(t)\Big\}.
	\end{eqnarray*}
	Thus, from \eqref{1dyjhlibsxF0} and \eqref{1MEstimate}, one gets $\mathbf{G}_\varepsilon$
	in the interval $[0,\tau^{\delta})$ as in Theorem \ref{mainthemCPE}.
	
	On the other hand, one derives that $\mathbf{G}_\varepsilon =(G^A_\varepsilon, G^B_\varepsilon)^{\top}$ satisfies
	\begin{eqnarray}\label{Perturbationequation}
		\begin{aligned}
			\partial_t G^\alpha_\varepsilon +v\cdot \nabla_x G^\alpha_\varepsilon-\frac{1}{\varepsilon} \sum_{\beta=A,B} \Big[Q^{\alpha \beta}(\mu^\alpha_0, G^\beta_\varepsilon)+Q^{\alpha \beta}(G^\alpha_\varepsilon, \mu^\beta_0)\Big]=\frac{\delta}{\varepsilon}  \sum_{\beta=A,B}Q^{\alpha \beta }( G^\alpha_\varepsilon, G^\beta_\varepsilon).
		\end{aligned}	
	\end{eqnarray}
	Formally, one gets that
	$$ \left(
	\begin{array}{lll}
		Q^{A A}(\mu^A_0, G^A_\varepsilon)+Q^{A A}(G^A_\varepsilon, \mu^A_0)+Q^{A B}(\mu^A_0, G^B_\varepsilon)+Q^{A B}(G^A_\varepsilon, \mu^B_0)\vspace{2mm}\\
		Q^{B A}(\mu^B_0, G^A_\varepsilon)+Q^{B A}(G^B_\varepsilon, \mu^A_0)+Q^{B B}(\mu^B_0, G^B_\varepsilon)+Q^{B B}(G^B_\varepsilon, \mu^B_0)
	\end{array}\right)
	\rightarrow 0~~{\rm as}~~ \varepsilon \to 0.$$ It means that, for $\alpha =A,B$, one has
	\begin{eqnarray}\label{1dflimit}
		G^{\alpha}_{\varepsilon} \rightarrow G^{\alpha} =  \left\{\sigma^\alpha + m^\alpha \mathbf{u} \cdot \mathbf{v} + \Big(\frac{m^\alpha |v|^2-3}{2}\Big)\theta \right\} \mu^{\alpha}_0, ~~{\rm as}~~ \varepsilon \rightarrow 0.
	\end{eqnarray}
	From the solvability of \eqref{Perturbationequation}, one derives that $\sigma^A, \sigma^B, \mathbf{u}, \theta $ satisfy the acoustic system \eqref{EQF0EPSIONline}. Let the initial data of $\sigma^A, \sigma^B, \mathbf{u}, \theta $ is given as \eqref{nonlindt}, the acoustic limit of \eqref{MAMAMAMAzhuyaotuimox} is stated as follows.
	% Given $(\sigma^A_{\rm in},\sigma^B_{\rm in},\tilde{\mathbf{u}}_{\rm in},\tilde{\theta}_{\rm in})$ with irrotational velocity $\tilde{\mathbf{u}}_{\rm in}$, we denote
	%$\mathbf{\mu}^{\rm}_{\delta}$
	%where
	%\begin{gather}\label{mainthemSOFTCASD}
	%			\mu^{\alpha,\rm in}_{\delta}=	\mu^{\alpha}_{\delta}(0,x,v)=\frac{1+\delta\sigma^{\alpha}_{\rm in}(x)}{[2\pi(1+\delta \tilde{\theta}_{\rm in}(x))]^{\frac{3}{2}}}\exp\{-\frac{|v-\delta \tilde{\mathbf{u}}_{\rm in}(x)|^2}{2(1+\delta \tilde{\theta}_{\rm in}(x))}\}.
	%		\end{gather}	

%	Moreover, we define
%	\begin{equation}\label{defGa}
	%		G^{\alpha}=\left(\sigma^{\alpha}+v\cdot\tilde{\mathbf{u}}+(\frac{m^{\alpha}|v|^2-3}{2})\tilde{\theta}\right)\mu^{\alpha}_0.
	%	\end{equation}
%the above acoustic system \eqref{EQF0EPSIONline}
% where $\mathbf{\bar{R}}^{\rm in}_{k-1}(x,v)=\mathbf{\bar{R}}_{k-1}(0,x,v)$, see \eqref{RNDDYTH}.

\begin{thm}\label{The result for compressible euler limit}
	Let $\mathbf{F}_\varepsilon$ is the solution of \eqref{MAMAMAMAzhuyaotuimox} established in Theorem \ref{mainthemCPE}. Suppose that the relation \eqref{1relation} holds and $\mathbf{F}_\varepsilon$ is denoted as \eqref{defGepa}. Then, for any $\tau>0$, there exist $\varepsilon_0>0$ and $C >0$ such that, for each $0<\varepsilon\leq\varepsilon_0$, it holds
	\begin{equation*}
		\sup_{0\leq t \leq \tau} \sum_{\alpha=A,B}	\left( \|	G^{\alpha}_{\varepsilon} - G^{\alpha} \|_{L^{2}_{x,v} }+\|	G^{\alpha}_{\varepsilon} - G^{\alpha} \|_{L^{\infty}_{x,v} }\right)\leq C \left(\frac{\varepsilon}{\delta}+\delta\right),
	\end{equation*}
	where $G^{\alpha}_{\varepsilon}$ and $G^{\alpha}$ are defined in \eqref{defGepa} and \eqref{1dflimit}.  Particularly, it can be observed that the optimal result is achieved when $\delta=\sqrt{\varepsilon}$.
\end{thm}

%	\begin{remark}
	%		The above theorem constitutes the main conclusion of this paper, whose results characterize the acoustic limit of the Boltzmann equation for unequal mass gas mixtures.
	%		?The initial Hilbert expansion \eqref{initalMain} and condition \eqref{remainKZ} ensure that
	%		\begin{equation*}
		%			\sum_{\alpha=A,B} \{ \|	G^{\alpha}_{\varepsilon}(0) - G^{\alpha}(0) \|_{L^{2}_{x,v} }+\|	G^{\alpha}_{\varepsilon}(0) - G^{\alpha}(0) \|_{L^{\infty}_{x,v} }\}
		%		\end{equation*}
	%		is  initially bounded by $\frac{\varepsilon}{\delta}+\delta$. Particularly, it can be observed that when $\delta=\sqrt{\varepsilon}$, the optimal result is achieved.
	%	\end{remark}

\noindent{\bf Notations}: We use notation $\Vert \cdot\Vert_{L^{p}_{x,v} } $ to denote the norm of $L^p$ space with respect to the variables $(x,v) \in \mathbb{R}^3  \times \mathbb{R}^3$,  and $| \cdot|_{L^p_v} $, $| \cdot|_{L^p_x} $ are the $L^p$ norms   w.r.t. the variable $v$ or $x$ ( $1\leq p \leq \infty$). For vector functions $\mathbf{f}(t,x,v)=(f^A,f^B)^\top$, its $L^p$ norms are denoted by
\begin{equation*}
	\Vert \mathbf{f} \Vert_{L^{p}_{x,v}}=\Vert f^A\Vert_{L^{p}_{x,v}}+\Vert f^B \Vert_{L^{p}_{x,v}},\qquad | \mathbf{f} |_{L^{p}_{v}}=| f^A|_{L^{p}_{v}}+| f^B|_{L^{p}_{v}}.
\end{equation*}
The notations  $\left \langle \cdot, \cdot	\right \rangle_{x,v}$ and $\left\langle \cdot, \cdot\right \rangle_v$ are the inner product of the function space $L^2_{x,v}$ and $L^2_{v}$ respectively. For vector functions $\mathbf{f}$ and $\mathbf{g}$ ,  their inner product in $L^2_v$ and $L^2_{x,v}$ are denoted as
\begin{gather*}
	\left \langle \mathbf{f}, \mathbf{g}\right \rangle_v = \left \langle f^A, g^A	\right \rangle_v + \left \langle f^B, g^B	\right \rangle_v, \\
	\left \langle \mathbf{f}, \mathbf{g}\right \rangle_{x,v} = \left \langle f^A, g^A	\right \rangle_{x,v} + \left \langle f^B, g^B	\right \rangle_{x,v}.
\end{gather*}
Let $\ell=(\ell_1,\ell_2,\ell_3)$ be a multi-index, its $s-$th order derivatives are $\partial_{x}^{\ell} =\partial_{x_1}^{\ell_1} \partial_{x_2}^{\ell_2} \partial_{x_3}^{\ell_3}  $ with $|\ell|=s $.  $W_x^{s,p}$ denotes the standard Sobolev space with corresponding norm $\Vert \cdot\Vert_{W_x^{s,p}} $. Moreover,  we write $W_x^{s,2}$ as $H_x^s$.  We also define $\left \langle v \right\rangle= 1+|v|$, and the collision frequency function
\begin{equation}\label{ffqzdf}
	\upsilon^{\alpha}=   \sum_{\beta=A,B} \int_{\mathbb{R}^3\times \mathbb{S}^2}
	B^{ \alpha \beta}(|v-v_*|, \cos \theta) \mu^{\beta}(v_*) d\sigma dv_*.
\end{equation}
For $-3<\gamma\leq1$, it can be proved that $\upsilon^{\alpha} \thicksim \left \langle v \right\rangle^{\gamma}$, we introduce the following weighted $L^2$ norm
\begin{gather*}
	\Vert f \Vert_{\nu}^2 = \iint_{\mathbb{R}^3 \times \mathbb{R}^3} |f|^2 \left \langle v \right\rangle^{\gamma} dxdv, \qquad  | f |_{\nu}^2 = \int_{\mathbb{R}^3 } |f|^2 \left \langle v \right\rangle^{\gamma}dv.
\end{gather*}
The letter $C$ is a constant, which may vary from line to line. $J \lesssim K$ means $J \leq C K$. If both hold $J \lesssim K$ and $K \lesssim J$, then we use the notation $J\thicksim K$.

This paper is structured as follows: Section 2 establishes the properties of solutions to the compressible Euler system \eqref{EQF0EPSION} and linear hyperbolic systems \eqref{linarFLsys}. Combining this with residual estimates, Theorem \ref{mainthemCPE} is proven. We firstly derive equations the expansion, including the existence of the solution for the compressible Euler limit-Theorem \ref{mainthemCPE} as well as the existence and bound of the terms in the expansion.  Then, we prove Theorem \ref{mainthemCPE} by establishing the estimates of the remainder terms by  $L^2-L^{\infty}$ estimates  Notably, it analyzes the velocity decay of the $K_{M,2,w}^{\alpha,c}$ operator under mass inequality conditions and provides a rigorous justification. We establish the acoustic limit results under the framework of small perturbations in Section 3. It derives the acoustic system by linearizing the compressible Euler equations in the vicinity of small perturbations around the equilibrium state (1,1,0,1), and establishes the existence and boundedness estimates of the perturbations. We further expands the solution of the compressible Euler equations to the second-order in small perturbations, deriving the linear equation satisfied by the second-order perturbation coefficients, and establishes the existence and boundedness estimates of the solution. Then we provide an error estimate between $\bf{\mu}_{\delta}$ and its first-order expansion with respect to small perturbations using the Taylor expansion method. Finally, we yield the acoustic limit for a mixture of two particle species with unequal masses by the above discussion.

\section{The compressible Euler limit}

\subsection{The determination of $\mathbf{F}_k~~(k=0,1,\cdots,5)$.}
We firstly determine $\mathbf{F}_0$ as following
\begin{equation}\label{dyjhlibsxF0}
	\mathbf{F}_0 =	\left(
	\begin{array}{cccc}
		F_0^{A}\\
		F_0^{B}
	\end{array}
	\right)
	=
	\left(
	\begin{array}{cccc}
		\frac{n_{\delta}^A(m^A)^{3/2} }{(2\pi \theta_{\delta})^{3/2}}e^{-\frac{m^A |v-\mathbf{u}_{\delta}|^2}{2\theta_{\delta}}} \vspace{3pt}\\
		\frac{n^B_{\delta}(m^B)^{3/2} }{(2\pi \theta_{\delta})^{3/2}}e^{-\frac{m^B |v-\mathbf{u}_{\delta}|^2}{2\theta_{\delta}}}
	\end{array}
	\right)=: \left(
	\begin{array}{cccc}  \mu^A_{\delta} \\ \mu^B_{\delta}\end{array}
	\right)  		
\end{equation}
with the initial data  $ \mathbf{F}_0(0)= ( \mu^{A,\rm in}_{\delta},\mu^{B,\rm in}_{\delta})^\top $. Here the initial value is imposed as
\begin{equation}\label{nonlindt}
	\begin{split}
		%(n^{A,\rm{in}}_{\delta}, n^{B,\rm{in}}_{\delta}, \mathbf{u}^{\rm{in}}_{\delta}, \theta^{\rm{in}}_{\delta})=
		(n^A_{\delta}, n^B_{\delta}, \mathbf{u}_{\delta}, \theta_{\delta})|_{t=0} &= (n^{A,\rm{in}}_{\delta}, n^{B,\rm{in}}_{\delta}, \mathbf{u}^{\rm{in}}_{\delta}, \theta^{\rm{in}}_{\delta}).
	\end{split}	
\end{equation}
which is given as
\begin{equation*}
	(n^{A,\rm{in}}_{\delta}, n^{B,\rm{in}}_{\delta}, \mathbf{u}^{\rm{in}}_{\delta}, \theta^{\rm{in}}_{\delta})=	(1+\delta\sigma^{A,\rm in}, 1+\delta\sigma^{B, \rm in},\delta \mathbf{u}^{\rm in},1+\delta \theta^{\rm in}),
\end{equation*}
where the strength of the initial perturbation is denoted as $\delta$ and
\begin{equation}\label{000conditifs}
	\mathcal{E}^{\rm in}_{s_0}=:\|(\sigma^{A,\rm in}, \sigma^{B,\rm in}, \mathbf{u}^{\rm in},  \theta^{\rm in})\|_{H^{s_0}}<+\infty, \quad \quad s_0\geq9.
\end{equation}

It is necessary to get the equations of $n^A_{\delta}, n_\delta^B,\mathbf{u}_\delta,\theta_\delta$.
%By a direct calculation, $(n^A,n^B,\mathbf{u},\theta)$ satisfy the following relation:
%\begin{equation}
%	\int_{\mathbb{R}^3} F_0^{\alpha} dv = n^{\alpha}_{\delta}, \quad \int_{\mathbb{R}^3} v F_0^{\alpha} dv = n^{\alpha}_{\delta}\mathbf{u}_{\delta},\quad \int_{\mathbb{R}^3}|v|^2 F_0^{\alpha} dv = n^{\alpha}_{\delta} |\mathbf{u}_{\delta}|^2+3n^{\alpha}_{\delta}\theta_{\delta},
%\end{equation}
%and we define
%\begin{equation*}
%	\rho_{\delta}=m^An^A_{\delta}+m^Bn^B_{\delta}, \quad \qquad n_{\delta}=n^A_{\delta}+n^B_{\delta}.
%\end{equation*}
%The first line of equation $\eqref{ORDERZONG}_1$ can be written in a vector form,
From the solvability of $\mathbf{L}_\delta$, the equation \eqref{1operator1} gives $\mathbf{R}_0 \in \mathcal{N}^{\bot}$, that is $\langle\mathbf{R}_0, \mathbf{X}_i^\delta\rangle_v=0$. It derives to
%\begin{equation}
%	\left(
%	\begin{array}{cccc}
	%		( \partial_t +v \cdot \nabla_x)F_0^A\\
	%		( \partial_t +v \cdot \nabla_x)F_0^B
	%	\end{array}
%	\right)=\sum_{i+j=1}
%	\left(
%	\begin{array}{cccc}
	%		Q^{AA}(F_i^A,F_j^A)+Q^{AB}(F_i^A,F_j^B)\vspace{3pt}\\
	%		Q^{BA}(F_i^B,F_j^A)+Q^{BB}(F_i^B,F_j^B)
	%	\end{array}
%	\right).
%\end{equation}
%Taking inner product with six collision invariants$ \{\mathbf{e}_1, \mathbf{e}_2, v_1 \mathbf{m}, v_2 \mathbf{m}, v_3 \mathbf{m}, |v|^2 \mathbf{m}\}$ (see \eqref{mabasis}), then one can obtain the system of $(n^{A}_{\delta},n^{B}_{\delta},\mathbf{u}_{\delta},\theta_{\delta})$:
\begin{equation}\label{EQF0EPSION}
	\left \{
	\begin{array}{lll}
		\partial_t n^{A}_{\delta} +{\rm div}(n^{A}_{\delta}\mathbf{u}_{\delta})=0, \\[2mm]
		\partial_t n^{B}_{\delta} +{\rm div}(n^{B}_{\delta}\mathbf{u}_{\delta})=0, \\[2mm]
		\partial_t  \mathbf{u}_{\delta}
		+(\mathbf{u}_{\delta} \cdot \nabla \mathbf{u}_{\delta}) + \frac{n_{\delta}\nabla \theta_{\delta}}{\rho_{\delta}}+ \frac{\theta_{\delta}\nabla n_{\delta}}{\rho_{\delta}}
		=0, \\[2mm]
		\partial_t \theta_{\delta} + \mathbf{u}_{\delta} \cdot \nabla \theta_{\delta} + \frac{2}{3} \theta_{\delta} \nabla \cdot \mathbf{u}_{\delta} = 0,
	\end{array}
	\right.
\end{equation}
where
\begin{equation*}
	\rho_{\delta}=m^An^{A}_{\delta}+m^Bn^{B}_{\delta}, \quad \qquad n_{\delta}=n^{A}_{\delta}+n^{B}_{\delta}. \\
\end{equation*}

%Then we define
%\begin{equation}\label{ininninnn}
%	\rho^{\rm{in}}_{\delta}=m^An^{A,\rm{in}}_{\delta}+m^Bn^{B,\rm{in}}_{\delta}, \quad \qquad n^{\rm{in}}_{\delta}=n^{A,\rm{in}}_{\delta}+n^{B,\rm{in}}_{\delta}.
%\end{equation}
%W%e are interested in the smooth solution of \eqref{EQF0EPSION} when the initial data is a smooth perturbation from a constant state with strength $\delta$.
The classical theory of symmetric hyperbolic systems shows that \eqref{EQF0EPSION} with \eqref{nonlindt} admits smooth solution
at least in the time interval $O(\frac{1}{\delta})$. The result is summarized in the following lemma.
\begin{lema}\label{cpesystemlem}
	
	Let $\delta_1>0$ be sufficiently small such that $(n^{A,\rm in}_\delta, n^{B,\rm in}_\delta,  \theta^{\rm in}_\delta)>0$ for all $0<\delta<\delta_1$. Then, there are constants $C_0, C_1>0$ such that
	\eqref{EQF0EPSION}-\eqref{nonlindt} admits a family of classical solutions
	\begin{equation*}
		(n^A_{\delta}, n^B_{\delta}, \mathbf{u}_{\delta}, \theta_{\delta})= (1+\delta \sigma^A, 1+\delta \sigma^B, \delta \mathbf{u}, 1+\delta \theta)
	\end{equation*}
	with
	\begin{equation*}
		(\sigma^A, \sigma^B, \mathbf{u}, \theta)\in C([0, \tau^\delta]; H^s) \cap C^1([0, \tau^\delta]; H^{s-1}),
	\end{equation*}
	for some $\tau^\delta \geq \frac{C_0}{\delta}.$
	Moreover, it satisfy $n^A_{\delta}, n^B_{\delta},\theta_{\delta}>0$ and
	\begin{equation}\label{2boundzeroorder}
		\|(n^A_{\delta}, n^B_{\delta}, \mathbf{u}_{\delta}, \theta_{\delta})-(1,1,0,1)\|_{C([0, \tau^\delta]; H^s) \cap C^1([0, \tau^\delta]; H^{s-1})}\leq C_1 \delta.
	\end{equation}
\end{lema}
\noindent The proof of this lemma is in \cite{[50]Wu2023JDE} page 427. $\hfill{\square}$\\
%\begin{remark} For  simplicity, we use notations $\mu^{\alpha}_M, \theta_M$ here, though they both actually depend on $\delta$, and we have omitted the subscript $\delta$. Furthermore, for any time interval, \eqref{globalMaxwellian} can be satisfied by choosing a sufficiently small $\delta$.
%\end{remark}

Next, we determine the terms  $\mathbf{F}_k$, or $\mathbf{f}_k=(f^A_k,f^B_k)^{\top} (k=1,2,\cdots,5)$ in the following.

Suppose that one has $(f_l^A, f_l^B)^\top$ for $k\leq k$ in hand, from $\eqref{ORDERZONG}_3$, we obtain
\begin{eqnarray}\label{RNDDYTH}
	\begin{aligned}
		\mathbf{L}_{\delta}  \mathbf{f}_{k}%&=\frac{1}{\sqrt{\mu^\alpha_\delta}}\Big[-(\partial_t +v \cdot \nabla_{x}) ( \sqrt{\mu^\alpha_\delta} f_k^{\alpha})
		%		+\displaystyle{\sum_{i+j=k+1\atop i,j\geq1} \sum_{\beta=A,B}} Q^{ \alpha\beta}(\sqrt{\mu^\alpha_\delta}f_i^{\alpha},\sqrt{\mu^\alpha_\delta}f_j^{\beta})\Big]\\
		=\mathbf{R}_{k-1}, \quad \mathbf{L}_{\delta}  \mathbf{f}_{k+1}%&=\frac{1}{\sqrt{\mu^\alpha_\delta}}\Big[-(\partial_t +v \cdot \nabla_{x}) ( \sqrt{\mu^\alpha_\delta} f_k^{\alpha})
		%		+\displaystyle{\sum_{i+j=k+1\atop i,j\geq1} \sum_{\beta=A,B}} Q^{ \alpha\beta}(\sqrt{\mu^\alpha_\delta}f_i^{\alpha},\sqrt{\mu^\alpha_\delta}f_j^{\beta})\Big]\\
		=\mathbf{R}_{k}.
	\end{aligned}
\end{eqnarray}
On one hand, one could get the micro-part of $\mathbf{f}_{k}$ is expressed as
\begin{eqnarray}\label{micropart1.301}
	(\mathbf{I}-\mathcal{P})\mathbf{f}_{k} = \mathbf{L}_\delta^{-1} \mathbf{R}_{k-1}.
\end{eqnarray}
On the other hand, the solvability condition of \eqref{RNDDYTH}$_2$  means that $	\left\langle \mathbf{R}_k, \mathbf{X}_{j}^\delta \right\rangle_v =0$ for all $j=0,\cdots,5.$ It derives to the following equations
\begin{equation}\label{linarFLsys}
	\left\{
	\begin{array}{lllll}
		\partial_t n^{A}_{k} +{\rm div}_x(n^{A}_{k} \mathbf{u}_{\delta}+n^{A}_{\delta}\mathbf{u}_{k})=0,\vspace{3pt} \\
		\partial_t n^{B}_{k} +{\rm div}_x(n^{B}_{k}\mathbf{u}_{\delta}+n^{B}_{\delta}\mathbf{u}_{k})=0,%\vspace{3pt}
		\\
		%\hspace{1cm}
		\rho_{\delta}(\partial_{t}\mathbf{u}_{k}+\mathbf{u}_{k}\cdot\nabla_{x}\mathbf{u}_{\delta}+\mathbf{u}_{\delta}\cdot\nabla_{x}\mathbf{u}_{k})-\Big[\frac{n^{A}_{k}}{n^{A}_{\delta}}\nabla_{x}(n^{A}_{\delta}\theta_{\delta}) + \frac{n^{B}_{k}}{n^{B}_{\delta}}\nabla_{x}(n^{B}_{\delta}\theta_{\delta})\Big] \vspace{3pt}\\
		\hspace{3cm}+\nabla \Big(\frac{n^A_{\delta} \theta_k+3\theta_{\delta} n_k^A}{3}\Big)+\nabla \Big(\frac{n^B_{\delta} \theta_k+3\theta_{\delta} n_k^B}{3}\Big)=\mathbf{H}_{k-1},\vspace{3pt} \\
		n_{\delta} \Big(\partial_t \theta_{k} + \frac{2}{3}(\theta_{k} {\rm div} \mathbf{u}_{\delta} + 3\theta_{\delta} {\rm div} \mathbf{u}_{k})	+\mathbf{u}_{\delta}\cdot \nabla_{x} \theta_{k} + 3\mathbf{u}_{k} \cdot \nabla_{x} \theta_{\delta} \Big) = {\rm g}_{k-1},
	\end{array}
	\right.
\end{equation}
where
\begin{eqnarray}\label{espq1.30}
	\begin{aligned}
		(\mathbf{H}_{k-1})_i= &- \sum_{\alpha=A,B}m^{\alpha}\sum_{j=1}^{3} \partial_{x_j}\Big(\frac{\theta_{\delta}}{m^\alpha} \int_{\mathbb{R}^3} \mathbf{A}^{\alpha}_{i,  j} f_k^{\alpha} dv\Big),   \\
		{\rm g}_{k-1}=& - 2 \sum_{\alpha=A,B} m^{\alpha} \sum_{i=1}^{3} \partial_{x_i} \Big((\frac{\theta_{\delta}}{m^\alpha})^{\frac{3}{2}}\int_{\mathbb{R}^3} \mathbf{B}^{\alpha}_{i }f_k^{\alpha} dv + \sum_{j=1}^{3} u_{\delta,j} \frac{\theta_{\delta}}{m^\alpha} \int_{\mathbb{R}^3} \mathbf{A}^{\alpha}_{i,  j}f_k^{\alpha} dv\Big)\\
		&  -2\mathbf{u}_k \cdot \mathbf{H}_{k-1}.
	\end{aligned}
\end{eqnarray}
The Burnett functions  $\mathbf{A}^{\alpha}_{i,  j},  \mathbf{B}^{\alpha}_{i } (\alpha=A, B) $ are defined as follows
\begin{equation*}
	\begin{split}
		\mathbf{A}^{\alpha}_{i,  j}:= &\Big(\frac{m^{\alpha}(v_i-\mathbf{u}_{\delta,i})(v_j-\mathbf{u}_{\delta,j})}{\theta_{\delta}} -\delta_{ij} \frac{m^{\alpha}|v-\mathbf{u}_{\delta}|^2}{3\theta_{\delta}}\Big)\sqrt{\mu_{\delta}^\alpha}, \\
		\mathbf{B}^{\alpha}_{i } :=& \frac{v_i-\mathbf{u}_{\delta,i}}{2}\sqrt{\frac{m^\alpha}{\theta_{\delta}}}(\frac{m^{\alpha}|v-\mathbf{u}_{\delta}|^2}{\theta_{\delta}}-5)\sqrt{\mu_{\delta}^\alpha}.
	\end{split}
\end{equation*}
Since $\mathbf{A}^{\alpha}_{i,  j} \in \mathcal{N}^{\bot}$ and $\mathbf{B}^{\alpha}_{i }\in \mathcal{N}^{\bot}$, the source term $\mathbf{H}_{k-1}$ and $\rm{g}_{k-1}$ in \eqref{espq1.30} depend only on the microscopic part of $(f_k^A, f_k^B)^\top$.
With the imposed initial data
\begin{equation}\label{kinivaluefk}
	(n^{A}_{k}, n^{B}_{k}, \mathbf{u}_{k}, \theta_{k})(0,x,v)=(n^{A,\rm in}_{k}, n^{B,\rm in}_{k}, \mathbf{u}^{\rm in}_{k}, \theta_{k}^{\rm in}),
\end{equation}
one could yield the solution $(n_k^A,n_k^B,\mathbf{u}_k,\theta_k)$ of  \eqref{linarFLsys} in the time interval $[0,\tau^{\delta})$. Thus, one yield the macroscopic part of $\mathcal{P}\mathbf{f}_k $ as
\begin{eqnarray}\label{macropart1.23}
	\begin{split}
		\mathcal{P} \mathbf{f}_k
		= &\frac{n^{A}_{k}}{\sqrt{n^{A}_{\delta}}}\mathbf{X}_0^\delta+\frac{n^{B}_{k}}{\sqrt{n^{B}_{\delta}}}\mathbf{X}_1^\delta
		+\sum_{j=2,3,4}\sqrt{\frac{n_{\delta}}{\theta_{\delta}}} u_k^{j-1} \mathbf{X}_j^\delta
		+\sqrt{\frac{n_{\delta}}{6}}\frac{\theta_k}{\theta_{\delta}}\mathbf{X}_5^\delta \\
		=& \left(\begin{array}{lll}
			\Big(\frac{n^A_{k}}{n^A_{\delta}} + \mathbf{u}_{k} \cdot \frac{m^A(v-\mathbf{u}_{\delta})}{\theta_{\delta}} +
			\frac{\theta_k}{6\theta_{\delta}}(\frac{m^A |v-\mathbf{u}_{\delta}|^2}{\theta_{\delta}} -3 )\Big) \sqrt{\mu^A_{\delta}}\vspace{3pt}\\
			\Big( \frac{n^B_{k}}{n^B_{\delta}} + \mathbf{u}_{k} \cdot \frac{m^B(v-\mathbf{u}_{\delta})}{\theta_{\delta}} +
			\frac{\theta_k}{6\theta_{\delta}}(\frac{m^B|v-\mathbf{u}_{\delta}|^2}{\theta_{\delta}} -3 )\Big) \sqrt{\mu^B_{\delta}}
		\end{array}\right),
	\end{split}
\end{eqnarray}
where %$(n_k^A, n_k^B, \mathbf{u}_k, \theta_k)$ are density of the mass, velocity and temperature from $\mathbf{f}_k$. For simplicity, we denote
\begin{eqnarray}\label{fluid def}
	\begin{split}
		&n_k= n_k^A+n_k^B, \quad  n_k^{\alpha} =\int_{\mathbb{R}^3} F_k^{\alpha} dv = \int_{\mathbb{R}^3} f_k^{\alpha} \sqrt{\mu_{\delta}^\alpha} dv,\\
		&  \rho_k= m^A n_k^A+m^B n_k^B, \,\,n_k^{\alpha} u^j +n^{\alpha}_{\delta} u^j_k= \int_{\mathbb{R}^3} v_j F_k^\alpha  dv= \int_{\mathbb{R}^3} v_j f_k^\alpha \sqrt{\mu_{\delta}^\alpha} dv,\\
		& n_{\delta}^\alpha \theta_{k} + 3 \theta_{\delta} n^{\alpha}_k = \int_{\mathbb{R}^3}m^{\alpha} |v-\mathbf{u}_{\delta}|^2 F_k^\alpha dv= \int_{\mathbb{R}^3}m^{\alpha} |v-\mathbf{u}_{\delta}|^2 f_k^\alpha \sqrt{\mu_{\delta}^\alpha} dv.
	\end{split}
\end{eqnarray}

We suppose 
$m^s = \min\{m^A,m^B\}$, for $s \in \{A,B\}$. 
The result is concluded in the following lemma.

\begin{lemma}\label{FI2KM1EST}
	Let $(n^A_{\delta}, n^B_{\delta}, \mathbf{u}_{\delta}, \theta_{\delta} )$ be the smooth solution of  established in Lemma \ref{cpesystemlem}.
	Suppose that there exists a series of integers  $ s_0 >  s_1  > \cdots > s_4 > s_5  > 3,$ such that initial data satisfy
	\begin{equation}\label{nkukthekint}
		\mathcal{E}^{\rm in}_{s_k}=:\|(n^{A,\rm in}_{k}, n^{B,\rm in}_{k}, \mathbf{u}^{\rm in}_{k}, \theta_{k}^{\rm in})\|_{H^{s_k}}<+\infty, \quad for \quad  k=1,2\cdots5.
	\end{equation}
	Then, $ F_{k}^{\alpha} = \sqrt{\mu^\alpha_{\delta}}f_{k}^{\alpha} $ $(\alpha = A,B, k = 1, 2,\cdots , 5)$ of \eqref{MAINEXPFPHIKKK} can be constructed from \eqref{micropart1.301}-\eqref{fluid def} for $ t \in [0, \tau^{\delta}] $. Moreover, there are series $ p_0 > p_1 > \cdots > p_5 \geq 4 $ and $\frac{\max\{m^A,m^B\}}{2(m^A+m^B)}<b<\frac{1}{2}$  such that
	\begin{align}\label{assumptions2}
		\begin{split}
			{\displaystyle  \sup_{t\in [0,\tau^{\delta}]} \sum_{k=1}^5 \sum_{\alpha=A,B}\sum_{\delta_k+|\delta_k' |\leq s_k} \|(1+|v|)^{p_k} [\mu^s_{\delta}]^{-b} \partial^{\delta_k}_t \partial_x^{\delta'_k} f^{\alpha}_k  \|_{L^2_xL^{\infty}_v}} \leq C(\tau, m^A,m^B, \mathcal{E}^{\rm in} ), 	\end{split}
	\end{align}
	where $\mu^s_{\delta}=\frac{n_{\delta}^s(m^s)^{3/2} }{(2\pi \theta_{\delta})^{3/2}}e^{-\frac{m^s |v-\mathbf{u}_{\delta}|^2}{2\theta_{\delta}}}$ and $\mathcal{E}^{\rm in} = \sum_{k=0}^5\mathcal{E}_{s_k}^{\rm in}$.
\end{lemma}

\subsection{The estimates of the remainder}
To prove this theorem, we need to establish a uniform estimate of the remainder, which is divided into two parts ($L^2$ estimate and $L^\infty$ estimate). We first rewrite the remainder as
\begin{equation}\label{littlef3.1}
	F_R^{\alpha} = 	\sqrt{\mu^\alpha_{\delta}}f^\alpha_R, \quad{\rm for} ~\alpha=A,B.  		
\end{equation}
From \eqref{reeqmain}, we derive $\mathbf{f}_R=(f_R^A, f_R^B)^\top$ satisfies the following equation
\begin{equation}\label{equationoff}
	\begin{split}
		\partial_t \mathbf{f}_R & +v\cdot\nabla_{x}\mathbf{f}_R + \frac{1}{\varepsilon} \mathbf{L}_\delta \mathbf{f}_R
		=  \mathbf{E}_1+\mathbf{E}_2+ \mathbf{E}_3+ \mathbf{E}_4.
	\end{split}	
\end{equation}
Here $\mathbf{E}_i= (E_i^A,E_i^B)^\top (i=1,2,3,4)$, and
\begin{equation*}
	\begin{split}
		E_1^{\alpha} = & \frac{(\partial_t+v\cdot \nabla_{x})\sqrt{\mu^\alpha_{\delta}}}{\sqrt{\mu^\alpha_{\delta}}}f_R^\alpha,\quad E_2^{\alpha}= \varepsilon^2 \displaystyle{\sum_{\beta=A,B} \Gamma^{ \alpha \beta}( f_R^{\alpha},  f_R^\beta ) },\\
		E_3^{\alpha}=& \sum_{i=1}^{5}\sum_{\beta=A,B}\varepsilon^{i-1}\Big[\Gamma^{ \alpha \beta}(f^\alpha_i,f^\beta_R)+\Gamma^{\alpha\beta }(f^\alpha_R, f^\beta_i)\Big],\\
		E_4^{\alpha}= & -\frac{(\partial_t+v\cdot \nabla_{x})}{\sqrt{\mu^\alpha_{\delta}}}(\sqrt{\mu^\alpha_{\delta}}f_5^\alpha)+{\small \sum_{ i,j \leq 5, i+j\geq 6 } \sum_{\beta=A,B}\frac{\varepsilon^{i+j-6}}{\sqrt{\mu^\beta_{\delta}}}\Big[\Gamma^{\alpha\beta}(f^\alpha_i,f^\beta_j)
			+\Gamma^{ \alpha \beta}(f^\alpha_j,f^\beta_i)\Big]}.
	\end{split}
\end{equation*}
Here the (non-symmetric) bilinear form  $\Gamma^{\alpha \beta}$ is given by
\begin{eqnarray}\label{gain loss part}
	\begin{split}
		\Gamma^{\alpha \beta }(g_1,g_2)\triangleq & \frac{1}{\sqrt{\mu^{\alpha }_{\delta}}}Q^{ \alpha \beta}(\sqrt{\mu^{\alpha }_{\delta}}g_1,  \sqrt{\mu^{\beta }_{\delta}}g_2)\\
		= &  \frac{1}{\sqrt{\mu^{\alpha }_{\delta}}}\int_{\mathbb{R}^3\times \mathbb{S}^2}|v-v_*|^{\gamma}b^{\alpha \beta}(\theta)\\
		&  \times \left(\sqrt{\mu^{\alpha }_{\delta}(v') \mu^{\beta }_{\delta}(v_*')}g_1(v')g_2(v_*')
		-\sqrt{\mu^{\alpha }_{\delta}(v)\mu^{\beta }_{\delta}(v_*)}g_1(v)g_2(v_*)\right)d\sigma dv_* \\
		=:& J_{\rm gain}^{\alpha\beta} - J_{\rm loss}^{\alpha\beta}.
	\end{split}
\end{eqnarray}

For the $L^2$ estimates of $\mathbf{f}_R$, we introduce the following  global bi-Maxwellian
\begin{gather}\label{def1.35}
	\mu_{M}=
	\left(
	\begin{array}{cccc}
		\mu_{M}^A\\
		\mu_{M}^B
	\end{array}
	\right)
	=
	\left(
	\begin{array}{cccc}
		\frac{({m^A})^{3/2}}{(2\pi \theta_M)^{3/2}}	e^{- \frac{m_A|v|^2}{2\theta_M}}\\
		\frac{({m^B})^{3/2}}{(2\pi \theta_M)^{3/2}}	e^{- \frac{m_B|v|^2}{2\theta_M}}
	\end{array}
	\right),
\end{gather}

From \eqref{2boundzeroorder}, there exists some sufficiently small $\delta_2>0$ such that, for all $0<\delta\leq \delta_2$,  one can choose $\theta_M$  satisfies
\begin{equation}\label{globalMaxwellian}
	\theta_M\leq\mathop{\rm min}_{(t,x)\in [0,\tau^\delta]\times \mathbb{R}^3} \theta_{\delta}(t,x)\leq\mathop{\rm max}_{(t,x)\in [0,\tau^\delta]\times \mathbb{R}^3} \theta_{\delta}(t,x)\leq \frac{m^A+m^B}{\max\{m^A,m^B\}} \theta_M.
\end{equation}
Thus, from \eqref{2boundzeroorder}, there exist constants $C $ and
$\frac{m^A}{m^A+m^B}<\tilde{q}<1$ such that
\begin{equation}\label{1.36}
	C^{-1} \mu^{\alpha}_M \leq \mu^{\alpha}_{\delta}  \leq C  (\mu^{\alpha}_M)^{\tilde{q}}, \quad {\rm for}\quad \alpha=A,B.
\end{equation}

Furthermore,  we define
\begin{equation}\label{littleh3.2}
	h^\alpha_R=  \frac{(1+|v|^2)^{l} }{\sqrt{\mu^\alpha_M}} F_R^\alpha =   \frac{w }{\sqrt{\mu^\alpha_M}} F_R^\alpha\quad {\rm for} ~~\alpha=A,B
\end{equation}
%and denote that
%\begin{equation}
%	\mathbf{h}_R = \left(
%	\begin{array}{ccc}
	%		h^A_{R}\\
	%		h^B_R
	%	\end{array}
%	\right).
%\end{equation}
%for any fixed
%\begin{center}
%	$\beta \geq 10- |\gamma|$
%\end{center}
where $w(v)=\left \langle v \right\rangle^{l}$ with $  l >\frac{21}{4}- \gamma$. The $L^2$  estimate for the remainder is list in the following proposition.

\begin{proposition}\label{L2estimate}
	Let $(n^{A}_{\delta},n^{B}_{\delta},\mathbf{u}_{\delta},\theta_{\delta})\in L^{\infty}([0,\tau^{\delta}),H^{s_0})$ be the solution of the hyperbolic system \eqref{EQF0EPSION} obtained in Lemma \ref{cpesystemlem}. Let $\mathbf{f}_R=(f^A_R,f^B_R)^\top $%$
	satisfies \eqref{equationoff} and $\mathbf{h}_R=(h^A_R,h^B_R)^\top $ is defined as \eqref{littleh3.2}. Then, there exist constants $\varepsilon_0 >0$ and $C$ such that for all $0<\varepsilon < \varepsilon_0$, it holds
	\begin{equation}\label{MainresultL2}
		\begin{split}
			\frac{d}{dt}\Vert \mathbf{f}_R\Vert_{L^{2}_{x,v}}^{2}+\frac{c_0}{2\varepsilon}\Vert(\mathbf{I}-\mathcal{P})\mathbf{f}_R\Vert_{\nu}^2
			\leq C \left(\sqrt{\varepsilon}\Vert\varepsilon^{\frac{3}{2}}\mathbf{h}_R\Vert_{L^{\infty}_{x,v}}+1\right)
			(\Vert \mathbf{f}_R\Vert_{L^{2}_{x,v}}^{2}+\Vert \mathbf{f}_R\Vert_{L^{2}_{x,v}}).
		\end{split}	
	\end{equation}
\end{proposition}
\noindent The proof of \eqref{MainresultL2} is presented in \cite{[50]Wu2023JDE} page 438-441.\\

We now consider the $L^{\infty}$ estimates of the remainder. Recall the global Maxwellian defined in \eqref{def1.35}, for the vector function $\mathbf{g}=(g^{A},g^{B})^{T}$, we introduce the  linear operator $\mathbf{L}_{M}=(\mathbf{L}_M^{A},\mathbf{L}_M^{B})^{T}$
as following:
\begin{gather}\label{4.2}
	\begin{split}
		\mathbf{L}^\alpha_M \mathbf{g}&=- \frac{1}{\sqrt{\mu^{\alpha}_M}} \sum_{\beta=A,B} \left[Q^{ \alpha\beta}(\mu^\alpha_{\delta},\sqrt{\mu^{\beta}_M} g^{\beta})+Q^{\alpha \beta}(\sqrt{\mu^\alpha_M}g^{\alpha},\mu^\beta_{\delta})\right].
		\vspace{3pt}\\
		&=: \upsilon^{\alpha} g^\alpha + K_M^\alpha \mathbf{g}.
	\end{split}	
\end{gather}
Here, the collision frequency is defined as 
\begin{equation}\label{nuKsuanzidef}
	\upsilon^{\alpha}=   \sum_{\beta=A,B} \int_{\mathbb{R}^3\times \mathbb{S}^2}
	B^{ \alpha \beta}(|v-v_*|, \cos \theta) \mu^{\beta}(v_*) d\omega dv_*, 
\end{equation}
and the operator is decoupled into  
\begin{equation}\label{2GlobalLinearoperator}
	K_{M}^{\alpha} \mathbf{g}= K_{M,1}^{\alpha} \mathbf{g}+K_{M,2}^{\alpha} \mathbf{g},
\end{equation}
with
\begin{equation*}
	\begin{split}
		&	 K_{M,1}^{\alpha} \mathbf{g}=  \frac{1}{\sqrt{\mu_M^\alpha}} \sum_{\beta=A,B} \int_{\mathbb{R}^3 \times \mathbb{S}^2}
		B^{\alpha\beta}(|v-v_*|, \cos \theta)  \mu^\alpha_{\delta}(v) \sqrt{\mu_M^\beta(v_*)}g^\beta(v_*)   d\omega dv_*,\\
		&	 K_{M,2}^{\alpha} \mathbf{g}= - \frac{1}{\sqrt{\mu_M^\alpha}} \sum_{\beta=A,B} \int_{\mathbb{R}^3 \times \mathbb{S}^2}
		B^{\alpha\beta}(|v-v_*|, \cos \theta)  \Big[\mu^\alpha_{\delta}(v')\sqrt{\mu_M^\beta(v_*')}g^\beta(v_*')\\
		& \hspace{6.8cm} +\mu^{\beta}_{\delta}(v_*') \sqrt{\mu_{M}^{\alpha}(v')}\, g^\alpha(v')\Big] d\omega dv_*.
	\end{split}	
\end{equation*}	

Firstly, there exists some constant $C$ depends on $\delta$ such that
\begin{eqnarray*}
	C^{-1} \left \langle v \right\rangle^{\gamma} \leq \upsilon^{\alpha} \leq C \left \langle v \right\rangle^{\gamma}, \quad  {\rm for}\quad \alpha=A,B.
\end{eqnarray*}	

Secondly, we introduce a smooth cut off function $0 \leq \chi_m \leq 1 $ to cancel the singular part in the operators of $K^{\alpha \beta}_M$,  for any small $\delta> 0$, such that%to cancel the singular part in the operators of $K^{\alpha \alpha}$ and $K^{\alpha \beta}$,
\begin{equation}\label{Cutoffunctiondef}
	\chi_m(s) \equiv 1 \quad \textnormal{for}  \, s \leq m, \quad {\rm and} \quad   \chi_m(s) \equiv 0,  \quad \textnormal{for}\, s \geq 2m.
\end{equation}
Let $B^{\alpha \beta}_{\chi}(|v-v_*|,\cos \theta) =: B^{\alpha \beta}(|v-v_*|,\cos \theta) \chi_m (|v-v_*|) $, define the singular parts of the  operators as
\begin{eqnarray*}
	K_{M}^{\alpha,\chi} \mathbf{g}= K_{M,1}^{\alpha,\chi} \mathbf{g}+K_{M,2}^{\alpha,\chi} \mathbf{g},
\end{eqnarray*}
with
\begin{gather}\label{4.51}
	\begin{split}
		& K_{M,1}^{\alpha,\chi} \mathbf{g}=  \frac{1}{\sqrt{\mu_M^\alpha}} \sum_{\beta=A,B} \int_{\mathbb{R}^3 \times \mathbb{S}^2}
		B_{\chi}^{\alpha\beta}(|v-v_*|, \cos \theta)  \mu^\alpha_{\delta}(v) \sqrt{\mu_M^\beta(v_*)}g^\beta(v_*)   d\omega dv_*,\\
		&	 K_{M,2}^{\alpha,\chi} \mathbf{g}= - \frac{1}{\sqrt{\mu_M^\alpha}} \sum_{\beta=A,B} \int_{\mathbb{R}^3 \times \mathbb{S}^2}
		B_{\chi}^{\alpha\beta}(|v-v_*|, \cos \theta)  \Big[\mu^\alpha_{\delta}(v')\sqrt{\mu_M^\beta(v_*')}g^\beta(v_*')\\
		& \hspace{6.8cm} +\mu^{\beta}_{\delta}(v_*') \sqrt{\mu_{M}^{\alpha}(v')}\, g^\alpha(v')\Big] d\omega dv_*.
	\end{split}
\end{gather}
Further, the remainder operator is defined as
\begin{eqnarray*}
	K_{M}^{\alpha,1-\chi} \mathbf{g}= K_{M,1}^{\alpha,1-\chi} \mathbf{g}+K_{M,2}^{\alpha,1-\chi} \mathbf{g},
\end{eqnarray*}
with
\begin{eqnarray}\label{ccccKCdef}
	K_{M,1}^{\alpha,1-\chi} \mathbf{g} = ( K_{M,1}^{\alpha}-  K_{M,1}^{\alpha,\chi}) \mathbf{g},\qquad K_{M,2}^{\alpha,1-\chi} \mathbf{g} = ( K_{M,2}^{\alpha}-  K_{M,2}^{\alpha,\chi}) \mathbf{g}.
\end{eqnarray}
Thus, one has 
\begin{eqnarray*}
	K_{M}^{\alpha} \mathbf{g}=K_{M}^{\alpha,\chi} \mathbf{g}+K_{M}^{\alpha,1-\chi} \mathbf{g}.
\end{eqnarray*}
 We denote
\begin{eqnarray*}
	\begin{split}
		K_{M,w}^{\alpha} (\mathbf{g})%=wK_{M}^{\alpha} (\frac{\mathbf{g}}{w})
		= K_{M,w}^{\alpha,\chi} (\mathbf{g})+K_{M,w}^{\alpha,1-\chi} (\mathbf{g}), 
	\end{split}
\end{eqnarray*}
with
\begin{eqnarray}\label{MWMWMWimpmain}
	\begin{split}
		&K_{M,w}^{\alpha,\chi} (\mathbf{g})=K_{M,1,w}^{\alpha,\chi} (\mathbf{g})+ K_{M,2,w}^{\alpha,\chi} (\mathbf{g})  ~ {\rm where}~  K_{M,i,w}^{\alpha,\chi} (\mathbf{g})=wK_{M,i}^{\alpha,\chi} (\frac{\mathbf{g}}{w}) \quad (i=1,2),\\
		& K_{M,w}^{\alpha,1-\chi} (\mathbf{g})=K_{M,1,w}^{\alpha,1-\chi} (\mathbf{g})+ K_{M,2,w}^{\alpha,1-\chi} (\mathbf{g})  ~ {\rm where}~  K_{M,i,w}^{\alpha,1-\chi} (\mathbf{g})=wK_{M,i}^{\alpha,1-\chi} (\frac{\mathbf{g}}{w}) \quad (i=1,2).
	\end{split}
\end{eqnarray}
 The following lemma gives the $L^{\infty}_v$ estimates of the singular parts $K_{M,1,w}^{\alpha,\chi}, K_{M,2,w}^{\alpha,\chi}$.

\begin{lemma}\label{lemma4.1} Let $-3< \gamma \leq 1$ and $0<m<1$ are given. Then, there exists some constant $C$  independent of $m$ such that
	\begin{gather}\label{4.6}
		|K_{M,w}^{\alpha,1-\chi} \mathbf{g} | \leq C m^{3+\gamma} \left \langle v \right\rangle^{\gamma} | \mathbf{g} |_{L^\infty_v}.
	\end{gather}
\end{lemma}
\begin{proof} By \eqref{1.36}, c.f \cite{[22]Guo2010CPAM}, we have
	$$	c_1 \mu_M^{\alpha} \leq   \mu^{\alpha}_{\delta} \leq c_2  (\mu_M^{\alpha})^{\tilde{q}} , \quad {\rm for \,\, some }\quad  \tilde{q} > \frac{\max\{m^A,m^B\}}{m^A+m^B}.$$
	From the conservation law $m^{\alpha}|v'|^{2}+m^{\beta}|v_*'|^2 =m^{\alpha}|v|^2+m^{\beta}|v_*|^2$  for $\alpha , \beta \in \{A,B\}$, we obtain
	\begin{gather}\label{4.12}
		\begin{split}
			\mu^{\alpha}_{\delta}(v')\frac{\sqrt{\mu_M^{\beta}(v_*')}}{\sqrt{\mu_M^{\alpha}(v)}}\leq& c_2  [\mu_M^{\alpha}(v') ]^{\tilde{q}} \frac{\sqrt{\mu_M^{\beta}(v_*')}}{\sqrt{\mu_M^{\alpha}(v)}} \leq C [\mu_M^{\alpha}(v') ]^{\tilde{q}-\frac{1}{2}} \sqrt{\mu_M^{\beta}(v_*)}, \\		
			\mu^{\beta}_{\delta}(v_*')\frac{\sqrt{\mu_M^{\alpha}(v')}}{\sqrt{\mu_M^{\alpha}(v)}}\leq& c_2  [\mu_M^{\beta}(v_*') ]^{\tilde{q}} \frac{\sqrt{\mu_M^{\alpha}(v')}}{\sqrt{\mu_M^{\alpha}(v)}} \leq C [\mu_M^{\beta}(v_*') ]^{\tilde{q}-\frac{1}{2}} \sqrt{\mu_M^{\beta}(v_*)},\\
			\mu^\alpha_{\delta} (v)\frac{\sqrt{\mu_M^{\beta}(v_*)}}{\sqrt{\mu_M^{\alpha}(v)}} \leq& c_2 [\mu_M^{\alpha}(v)]^{\tilde{q}-\frac{1}{2}}\sqrt{\mu_M^{\beta}(v_*)} .
		\end{split}
	\end{gather}
	Since $|v-v_*| \leq 2m < 2$, then one has
	\begin{align*}
		&\mu_M^{\beta}(v_*) =(2\pi \theta_M)^{-\frac{3}{2}} e^{-\frac{m^{\beta}|v_*|^2}{2\theta_M}} \leq (2\pi \theta_M)^{-\frac{3}{2}} e^{-\frac{m^{\beta}\{|v|^2-|v-v_*|^2\}}{2\theta_M}}  \leq C \mu_M^{\beta}(v),\\
		&\frac{w(v)}{w(v_*)}\big[\mu_M^{\beta}(v)\big]^{1/4} \leq C \left \langle v \right\rangle^{\gamma}.
	\end{align*}
	Thus, for all $-3<\gamma\leq1$, one has
	\begin{align*}
		|K_{M,w}^{\alpha,1-\chi} \mathbf{g}| & \leq C \sum_{\beta=A,B}\int_{\mathbb{R}^3 \times \mathbb{S}^2} |v-v_*|^{\gamma} \frac{w(v)}{w(v_*)}\chi_{m}(|v-v_*|)b^{\alpha \beta}(\theta)d\omega dv_* \sqrt{\mu_M^{\beta}(v)}\, | \mathbf{g}  |_{L^\infty_v}\\
		& \leq C m^{3+\gamma} \left \langle v \right\rangle^{\gamma} | \mathbf{g}  |_{L^\infty_v}.
	\end{align*}
	This  completes the proof of Lemma \ref{lemma4.1}.
\end{proof}

The following  lemma  provide detailed estimates for $K_{M}^{\alpha,c}$ operators, which  plays an important role in $L^{\infty}$ estimates.
\begin{lemma}\label{LeMK2ker4444} For $-3<\gamma\leq1$, by the notations in \eqref{2GlobalLinearoperator}-\eqref{MWMWMWimpmain}, the regular part is expressed as $K_{M,w}^{\alpha,1-\chi} = K_{M,1,w}^{\alpha,1-\chi}+K_{M,2,w}^{\alpha,1-\chi}$. Further, $K_{M,1,w}^{\alpha,\chi}, K_{M,2,w}^{\alpha,\chi}$ can be expressed in the following form of the integral
	\begin{eqnarray}\label{lem4311232}
		\begin{split}
			K_{M,1,w}^{\alpha,1-\chi} \mathbf{g}(v)=\sum_{\beta=A,B}\int_{\mathbb{R}^3} k_{M,1}^{\alpha \beta}(v,v_*)\frac{w(v)}{w(v_*)}g^{\beta}(v_*) d v_{*},
		\end{split}
	\end{eqnarray}
	and
	\begin{equation}
		K_{M,2,w}^{\alpha,1-\chi}\mathbf{g}(v) = \underbrace{K_{M,2,w}^{\alpha,1-\chi,(1)}\mathbf{g}(v)}_{\text{Typical part}} + \underbrace{K_{M,2,w}^{\alpha,1-\chi,(2)}\mathbf{g}(v)}_{\text{Hybrid part}}.
	\end{equation}
	where the ``Typical part'' is expressed as
	\begin{eqnarray}\label{lemerew4311234}
		\begin{split}
			K_{M,2,w}^{\alpha,1-\chi,(1)}\mathbf{g}(v)
			&=\sum_{\beta=A,B} \int_{\mathbb{R}^3}\frac{w(v)}{w(v+\frac{2m^{\beta}}{m^{\alpha}+m^{\beta}}u_{\parallel})}k_{M,2}^{\alpha \beta(1)}(v,u_{\parallel})g^{\alpha}(v+\frac{2m^{\beta}}{m^{\alpha}+m^{\beta}}u_{\parallel}) d u_{\parallel},
		\end{split}
	\end{eqnarray}
	and the ``Hybrid part'' is expressed as
	\begin{eqnarray}\label{lem431123512}
		\begin{split}
			K_{M,2,w}^{\alpha,1-\chi,(2)}\mathbf{g}(v)
			&=\prescript{(\alpha \neq \beta)}{}{K}_{M,2,w}^{\alpha,1-\chi,(2)}\mathbf{g}(v)+\prescript{(\alpha = \beta)}{}{K}_{M,2,w}^{\alpha,1-\chi,(2)}\mathbf{g}(v)\\
			&=:\int_{\mathbb{R}^3}\int_{\mathbb{R}^2}\frac{w(v)}{w(v+u_{\perp}+\frac{m^{\beta}-m^{\alpha}}{m^{\alpha}+m^{\beta}}u_{\parallel})}k_{M,2}^{\alpha \beta(2)}(v,u_{\perp},u_{\parallel})\\ &\hspace{5cm}\times g^{\beta}(v+u_{\perp}+\frac{m^{\beta}-m^{\alpha}}{m^{\alpha}+m^{\beta}}u_{\parallel}) d u_{\perp} d u_{\parallel} \qquad \alpha \neq \beta\\
			&\hspace{0.6cm}+\int_{\mathbb{R}^3}	\frac{w(v)}{w(v+u_{\perp})}k_{M,2}^{\alpha \alpha(2)}(v,u_{\perp})g^{\alpha}(v+u_{\perp}) d u_{\perp}.
		\end{split}
	\end{eqnarray}
	There exists a small constant $c>0$ such that
	\begin{eqnarray}\label{a1a4main236}
		\begin{split}
			\frac{w(v)}{w(v_*)}|k_{M,1}^{\alpha \beta}(v,v_*)|&\leq C \,e^{-c(|v|^2+|v_*|^2)},\\
			\frac{w(v)}{w(v+\frac{2m^{\beta}}{m^{\alpha}+m^{\beta}}u_{\parallel})}|k_{M,2}^{\alpha \beta(1)}(v,u_{\parallel})|&\leq C_m \frac{(1+|v|+|u_{\parallel}|)^{\gamma-1}}{|u_{\parallel}|} \, e^{-c(|u_{\parallel}|^2+|v_{\parallel}|^2)}.
		\end{split}
	\end{eqnarray}
	If $\alpha \neq \beta$, 
	\begin{eqnarray}\label{sofd237e}
		\begin{split}
			\frac{w(v)}{w(v+u_{\perp}+\frac{m^{\beta}-m^{\alpha}}{m^{\alpha}+m^{\beta}}u_{\parallel})}|k_{M,2}^{\alpha\beta(2)}(v,u_{\perp},u_{\parallel})|\leq C_m\, e^{-c(|v|^2+|u_{\parallel}|^2+|u_{\perp}|^2)}.
		\end{split}
	\end{eqnarray}
	If $\alpha = \beta$, 
	\begin{eqnarray}\label{ZZZzequagezV}
		\begin{split}
			\frac{w(v)}{w(v+u_{\perp})}|k_{M,2}^{\alpha\alpha(2)}(v,u_{\perp})|&\leq C_m\,  \frac{(1+|v|+|u_{\perp}|)^{\gamma-1}}{|u_{\perp}|} \, e^{-c(|u_{\perp}|^2+|v_{\parallel}|^2)}.
		\end{split}
	\end{eqnarray}
	Moreover, they satisfy
	\begin{eqnarray}\label{sec4inL1decayK12}
		\begin{split}
			&\int_{\mathbb{R}^3}\big|k_{M,2}^{\alpha \beta(1)}(v,u_{\parallel})\big|\frac{w(v)}{w(v+\frac{2m^{\beta}}{m^{\alpha}+m^{\beta}}u_{\parallel})} d u_{\parallel} +\int_{\mathbb{R}^3}\big|k_{M,2}^{\alpha \alpha(2)}(v,u_{\perp})\big|\frac{w(v)}{w(v+u_{\perp})} d u_{\perp} \leq C_m \left \langle v \right\rangle^{\gamma-2}.
		\end{split}
	\end{eqnarray}
	%Hence	the following estimates yield
	%\begin{eqnarray}\label{K1111XVE}
		%| K_{M,1,w}^{\alpha,c}\mathbf{g}(v) |+| (D_x  K_{M,1,w}^{\alpha,c}) \mathbf{g}(v) |+| (D_v K_{M,1,w}^{\alpha,c}) \mathbf{g}(v) | \leq C_{\delta} \left \langle v \right\rangle^{\gamma-2}e^{-\frac{c}{2}|v|^2}  |\mathbf{g} |_{L^\infty_v},
	%\end{eqnarray}
%	and
	%\begin{eqnarray}\label{K222XVE}
		%\begin{split}
		%	| K_{M,2,w}^{\alpha,c}\mathbf{g}(v) |+| (D_x  K_{M,2,w}^{\alpha,c}) \mathbf{g}(v) | &\leq C_{\delta}  \left \langle v \right\rangle^{\gamma-2}  |\mathbf{g} |_{L^\infty_v},\\
		%	| (D_v  K_{M,2,w}^{\alpha,c} )\mathbf{g}(v) | &\leq C_{\delta} \left \langle v \right\rangle^{\gamma-1}  | \left \langle v \right\rangle\mathbf{g} |_{L^\infty_v}.
		%\end{split}
	%\end{eqnarray}
\end{lemma}

The above lemma provides the decay estimates for the linear $\mathbf{K}$-operator. Although applying it to prove the $L^{\infty}$-estimate is not difficult, it would consume approximately 8–10 pages. For the sake of conciseness in the article, we demonstrate its relationship with a single particle through the following two remarks. This connection allows its integral kernel form to be transformed into the form of the integral kernel satisfied by \cite{[22]Guo2010CPAM}, thus enabling us to apply its proof.

\begin{remark}\label{kernalcompwithsinglek}
	‌The kernel $k_{M,2}^{\alpha \beta(1)}(v,u_{\parallel})$ has similar properties to the Boltzmann equation kernel for a single particle species. Although it contains the more parameters $m^A,m^B$, in fact, through a change of variable:
	\begin{equation}\label{pybh}
		v+\frac{2m^{\beta}}{m^{\alpha}+m^{\beta}}u_{\parallel} \rightarrow v^*,
	\end{equation}
	 yields
	\begin{equation*}
		\int_{\mathbb{R}^3}k_{M,2}^{\alpha \beta(1)}(v,u_{\parallel})\frac{w(v)}{w(v+\frac{2m^{\beta}}{m^{\alpha}+m^{\beta}}u_{\parallel})}g^{\alpha}(v+\frac{2m^{\beta}}{m^{\alpha}+m^{\beta}}u_{\parallel}) d u_{\parallel}=	\int_{\mathbb{R}^3}\tilde{k}_{M,2}^{\alpha \beta(1)}(v,v^*)\frac{w(v)}{w(v^*)}g^{\alpha}(v^*) d v^* ,
	\end{equation*}
	which satisfies
	\begin{eqnarray}
		\begin{split}
			\frac{w(v)}{w(v^*)}|\tilde{k}_{M,2}^{\alpha \beta(1)}(v,v^*)|\leq C_{m} \frac{(1+|v|+|v^*|)^{\gamma-1}}{|v-v^*|} \, e^{-c\{|v-v^*|^2+\frac{||v|^2-|v^*|^2|^2}{|v-v^*|^2}\}}.
		\end{split}
	\end{eqnarray}
	The above estimate is identical to that for the single-particle kernel.  Hence, $K_{M,2,w}^{\alpha,1-\chi,(1)}\mathbf{g}(v)$ is called ``Typical part". While, ``Hybrid part"   consists of two terms exhibiting different decay characteristics. One term has exponentially decaying integral kernels $k_{M,2}^{\alpha \beta(2)}(v,u_{\perp},u_{\parallel})$\, $(\alpha\neq\beta)$. The other term's integral kernel $k_{M,2}^{\alpha \alpha(2)}(v,u_{\perp})$ shares similar properties with the $k_{M,2}^{\alpha \alpha(1)}(v,u_{\parallel})$, since they are almost symetric $(u_{\parallel}\leftrightarrow u_{\perp})$.
\end{remark}

\begin{remark}\label{rek26262}
	‌By \eqref{lem431123512}, the infinitesimal elements of integral kernel are on $u_{\perp}$ and $u_{\parallel}$, and the independent variable in $g^{\beta}(\cdot)$ is $v+u_{\perp}+\frac{m^{\beta}-m^{\alpha}}{m^{\alpha}+m^{\beta}}u_{\parallel}$. This makes it impossible to perform a translation transformation as \eqref{pybh} to satisfy the integral kernel form in \cite{[22]Guo2010CPAM}. However, noting that the integral kernel $k_{M,2}^{\alpha \beta(2)}(v,u_{\perp},u_{\parallel})$ satisfies the estimate \eqref{sofd237e}, it can be indirectly deformed formally through control functions. So, we perform the variable substitution 
	\begin{equation*}
		u_{\perp}+\frac{m^{\beta}-m^{\alpha}}{m^{\alpha}+m^{\beta}}u_{\parallel} \rightarrow u^*.
	\end{equation*}
	 There exists an integral kernel $\hat{k}_{M,2}^{\alpha \beta(2)}(v,u^*)$ such that:
	\begin{eqnarray}\label{ewfrwjkf}
		\begin{split}
			&\big|\prescript{(\alpha \neq \beta)}{}{K}_{M,2,w}^{\alpha,1-\chi,(2)}\mathbf{g}(v)\big|\leq\int_{\mathbb{R}^3}\hat{k}_{M,2}^{\alpha \beta(2)}(v,u^*) g^{\beta}(v+u^*) \,d u^* \qquad \alpha \neq \beta,
		\end{split}
	\end{eqnarray}
	which satisfies
	\begin{eqnarray}
		\begin{split}
			|\hat{k}_{M,2}^{\alpha \beta(2)}(v,u^*)|\leq C e^{-c(|v|^2+|u^*|^2)} .
		\end{split}
	\end{eqnarray}
\end{remark}

The proofs of Remark \ref{kernalcompwithsinglek} and Remark \ref{rek26262} will be given after the proof of Lemma \eqref{LeMK2ker4444}.

\noindent \textbf{Proof of Lemma \eqref{LeMK2ker4444}}:
\begin{proof} We take the  weight function $w(v)=(1+|v|)^l,\,l\geq0$ to give the proof, by \eqref{nuKsuanzidef} and \eqref{ccccKCdef}, we have
	\begin{equation*}
		K_{M,1}^{\alpha,1-\chi} \mathbf{g}=  \frac{\mu^\alpha_{\delta}(v)}{\sqrt{\mu_M^\alpha}} \sum_{\beta=A,B} \int_{\mathbb{R}^3 \times \mathbb{S}^2}
		\{1- \chi_{m}(|v-v_*|)\}B^{\alpha\beta}(|v-v_*|, \cos \theta)   \sqrt{\mu_M^\beta(v_*)}g^\beta(v_*)   d\omega dv_*.
	\end{equation*}
	So we take
	\begin{equation*}
		k_{M,1}^{\alpha \beta}(v,v_*)=
		\{1- \chi_{m}(|v-v_*|)\}B^{\alpha\beta}(|v-v_*|, \cos \theta)  \frac{\mu^\alpha_{\delta}(v)}{\sqrt{\mu_M^\alpha}} \sqrt{\mu_M^\beta(v_*)}.
	\end{equation*}
	We can choose $c$ small enough, such that
	\begin{equation*}
		\frac{\mu^\alpha_{\delta}(v)}{\sqrt{\mu_M^\alpha}} \sqrt{\mu_M^\beta(v_*)} \leq C e^{-2c\{|v|^2+|v_*|^2\}}.
	\end{equation*}
	Since $\kappa_0$ is small enough, and $B^{\alpha\beta}(|v-v_*|, \cos \theta) \leq C|v-v_*|^{\gamma}$, it yields that
	\begin{equation*}
		\frac{w(v)}{w(v_*)}|k_{M,1}^{\alpha \beta}(v,v_*)|\leq C \,e^{-c\{|v|^2+|v_*|^2\}}.
	\end{equation*}
	This complete the proof of $\eqref{a1a4main236}_1$.

	\noindent $\mathbf{Step 1}$.  Analysis of the expression for operator $K_{M,2}^{\alpha,1-\chi}$. 
	
	We perform a variable substitution:
	\begin{align}
		&u=v_*-v,\quad u_{\parallel} = (u \cdot \omega)\omega, \quad
		u_{\perp} = u - (u \cdot \omega)\omega ,
		\label{CE2ESYCGV}
	\end{align}
	where $u \in \mathbb{R}^3$, $\omega \in \mathbb{S}^2$, $u_{\parallel} \in \mathbb{R}^3$, $u_{\perp} \in \mathbb{R}^2$, and
	\begin{equation*}
		|u_{\parallel}|^{2}(|u_{\parallel}|^{-1}d|u_{\parallel}|)d\omega = du_{\parallel} ,\qquad du ={du}^{+}+{du}^{-}=2 (|u_{\parallel}|^{-1} d|u_{\parallel}|) du_{\perp} .
	\end{equation*}
	The expression $(|u_{\parallel}|^{-1}d|u_{\parallel}|)$ can be identified with the radial differential $``dr''$ in spherical coordinates, as its directional variation is confined to the axis defined by $\omega$ or $-\omega$. 
	It yields
	\begin{equation*}
		{du}^{+}=  \mathbf{1}_{\{u\cdot\omega\geq0\}}du=(|u_{\parallel}|^{-1} d|u_{\parallel}|) du_{\perp},\hspace{0.5cm} {du}^{-}= \mathbf{1}_{\{u\cdot\omega<0\}} du=(|u_{\parallel}|^{-1} d|u_{\parallel}|) du_{\perp}.
	\end{equation*}
	So, we get
	\begin{equation}\label{impykba}
		du d\omega = 2|u_{\parallel}|^{-2} du_{\perp} du_{\parallel}.
	\end{equation}
	If $m^{\alpha}>m^{\beta}$, we perform variable substitution
	$-u\rightarrow u$.
	Therefore, we only needs to consider  $m^\beta>m^\alpha$,\,for $\beta \neq \alpha$.
	
	Then $K_{M,2}^{\alpha,1-\chi} \mathbf{g}$ in \eqref{ccccKCdef} can be expressed by
	\begin{align}
		\sum_{\beta=A,B}&\Bigg(\int_{\mathbb{R}^3 \times \mathbb{S}^2} |u|^{\gamma} [\mu_M^{\beta}(u+v)]^{\frac{1}{2}}\{1-\chi_m\}(|u|)[\mu_{M}^{\beta}(v + u_{\perp}+\frac{m^{\beta}-m^{\alpha}}{m^{\alpha}+m^{\beta}}u_{\parallel})]^{\tilde{q}-\frac{1}{2}}\notag\\
		&\hspace{5cm}\times g^\alpha(v + \frac{2m^{\beta}}{m^{\alpha}+m^{\beta}}u_{\parallel})b^{ \alpha \beta}(\theta) \, du \, d\omega \Bigg)\label{CE2easyPT} \\
		+& \int_{\mathbb{R}^3 \times \mathbb{S}^2} |u|^{\gamma} [\mu_{M}^{\beta}(u+v)]^{\frac{1}{2}}\{1-\chi_m\}(|u|)[\mu_{M}^{\alpha}(v + \frac{2m^{\beta}}{m^{\alpha}+m^{\beta}}u_{\parallel})]^{\tilde{q}-\frac{1}{2}}\hspace{1cm}(\alpha \neq \beta)\notag\\
		&\hspace{5.6cm}\times g^\beta(v + u_{\perp}+\frac{m^{\beta}-m^{\alpha}}{m^{\alpha}+m^{\beta}}u_{\parallel})b^{ \alpha \beta}(\theta) \, du \, d\omega \label{CE2hadPT}\\
		+& \int_{\mathbb{R}^3 \times \mathbb{S}^2} |u|^{\gamma} [\mu_{M}^{\alpha}(u+v)]^{\frac{1}{2}}\{1-\chi_m\}(|u|)[\mu_{M}^{\alpha}(v + u_{\parallel})]^{\tilde{q}-\frac{1}{2}} g^\alpha(v + u_{\perp})b^{ \alpha \alpha}(\theta) \, du \, d\omega .\label{CE2hadPTeq}
	\end{align}
	where $\alpha \neq \beta$ in \eqref{CE2hadPT}. Let $\bar{c}=(2\tilde{q}-1)$, $0<\bar{c}<1$, and
	\begin{eqnarray*}
		\begin{split}
			\zeta_{\parallel} + \zeta_{\perp} &= \frac{1+\sqrt{\bar{c}}}{2}v + (\frac{1}{2}+\frac{\sqrt{\bar{c}}(m^{\beta}-m^{\alpha})}{2(m^{\alpha}+m^{\beta})})u_{\parallel}+\frac{1+\sqrt{\bar{c}}}{2}u_{\perp},\\
			\psi_{\parallel} + \psi_{\perp} &=\frac{1-\sqrt{\bar{c}}}{2}v + (\frac{1}{2}-\frac{\sqrt{\bar{c}}(m^{\beta}-m^{\alpha})}{2(m^{\alpha}+m^{\beta})})u_{\parallel} +\frac{1-\sqrt{\bar{c}}}{2}u_{\perp}.
		\end{split}
	\end{eqnarray*}
	where $\zeta_{\parallel} \parallel u_{\parallel},\quad   \zeta_{\perp} \parallel u_{\perp}$. Then, we change variables
	\begin{eqnarray}\label{cgvab4343111x}
		\begin{split}
			\zeta_{\parallel} &= \frac{1+\sqrt{\bar{c}}}{2}v_{\parallel} + (\frac{1}{2}+\frac{\sqrt{\bar{c}}(m^{\beta}-m^{\alpha})}{2(m^{\alpha}+m^{\beta})})u_{\parallel} \\
			\zeta_{\perp} &= \frac{1+\sqrt{\bar{c}}}{2}v_{\perp} +\frac{1+\sqrt{\bar{c}}}{2}u_{\perp},
		\end{split}
	\end{eqnarray}
	and
	\begin{eqnarray}\label{cgvab4343111p}
		\begin{split}
			\psi_{\parallel} &= \frac{1-\sqrt{\bar{c}}}{2}v_{\parallel} + (\frac{1}{2}-\frac{\sqrt{\bar{c}}(m^{\beta}-m^{\alpha})}{2(m^{\alpha}+m^{\beta})})u_{\parallel} \\
			\psi_{\perp} &= \frac{1-\sqrt{\bar{c}}}{2}v_{\perp} +\frac{1-\sqrt{\bar{c}}}{2}u_{\perp},
		\end{split}
	\end{eqnarray}
	where $v=v_{\parallel}+v_{\perp}, u=u_{\parallel}+u_{\perp}$, and $|v|^2=|v_{\parallel}|^2+|v_{\perp}|^2, |u|^2=|u_{\parallel}|^2+|u_{\perp}|^2$.
	
	\noindent 	Therefore, it yields that
	\begin{align*}
		u + v &= [\zeta_{\parallel}+\zeta_{\perp}]+[\psi_{\parallel}+\psi_{\perp}], \label{CEPT2ACA1vb}\\
		\sqrt{\bar{c}}(v + u_{\perp}&+\frac{m^{\beta}-m^{\alpha}}{m^{\alpha}+m^{\beta}}u_{\parallel}) = [\zeta_{\parallel}+\zeta_{\perp}]-[\psi_{\parallel}+\psi_{\perp}].\notag
	\end{align*}
	Obviously, the exponent in \eqref{CE2easyPT} is $-\frac{m^{\beta}}{4\theta_M}\Big[| \zeta_{\perp}|^2+|\zeta_{\parallel}|^2+|\psi_{\parallel}|^2+|\psi_{\perp}|^2\Big]$.	\\
	Notice that
	\begin{equation*}
		|u_{\parallel}| = |u \cos \theta| = \sqrt{|u_{\parallel}|^2 + |u_{\perp}|^2} \, |\cos \theta|.
	\end{equation*}
	Hence, \eqref{CE2easyPT} can be written as
	\begin{align}
		& \int_{\mathbb{R}^3} \frac{1}{|u_{\parallel}|^2} e^{-\frac{m^\beta}{4\theta_M}\big[|\zeta_{\parallel}|^2+|\psi_{\parallel}|^2\big]} g^\alpha(v + \frac{2m^{\beta}}{m^{\alpha}+m^{\beta}}u_{\parallel}) \notag\\
		&\quad \times \int_{\mathbb{R}^2} e^{-\frac{m^\beta}{4\theta_M}\big[| \zeta_{\perp}|^2+|\psi_{\perp}|^2\big]} \left[|u_{\parallel}|^2 + |u_{\perp}|^2\right]^{\frac{\gamma}{2}} \{1-\chi_m\}\big(\sqrt{|u_{\parallel}|^2 + |u_{\perp}|^2}\big) b^{ \alpha \beta}(\theta) \, du_{\perp} du_{\parallel} \notag\\
		=& \int_{\mathbb{R}^3} \frac{1}{|u_{\parallel}|} e^{-\frac{m^\beta}{4\theta_M}\big[|\zeta_{\parallel}|^2+|\psi_{\parallel}|^2\big]} g^\alpha(v + \frac{2m^{\beta}}{m^{\alpha}+m^{\beta}}u_{\parallel})\notag \\
		&\quad \times \int_{\mathbb{R}^2} e^{-\frac{m^\beta}{4\theta_M}\big[| \zeta_{\perp}|^2+|\psi_{\perp}|^2\big]} \left[|u_{\parallel}|^2 + |u_{\perp}|^2\right]^{\frac{\gamma - 1}{2}} \{1-\chi_m\}\big(\sqrt{|u_{\parallel}|^2 + |u_{\perp}|^2}\big) \frac{b^{ \alpha \beta}(\theta)}{|\cos \theta|} \, du_{\perp} du_{\parallel}.\label{CE2BDS249esD1}
	\end{align}
		We introduce the integral kernel $k_{M,2}^{\alpha \beta(1)}(v,u_{\parallel})$:
	\begin{align}
		k_{M,2}^{\alpha \beta(1)}(v,u_{\parallel}) &=:  \frac{1}{|u_{\parallel}|} e^{-\frac{m^\beta}{4\theta_M}\big[|\zeta_{\parallel}|^2+|\psi_{\parallel}|^2\big]} \int_{\mathbb{R}^2} e^{-\frac{m^\beta}{4\theta_M}[| \zeta_{\perp}|^2+|\psi_{\perp}|^2]}\notag\\
		&\hspace{2.8cm}\times  \left[|u_{\parallel}|^2 + |u_{\perp}|^2\right]^{\frac{\gamma - 1}{2}} \chi\big(\sqrt{|u_{\parallel}|^2 + |u_{\perp}|^2}\big) \frac{b^{ \alpha \beta}(\theta)}{|\cos \theta|} \, du_{\perp}.\label{ker1rrr}
	\end{align}
	Then \eqref{CE2easyPT} can be further expressed as
	\begin{equation*}
		\sum_{\beta=A,B}	\int_{\mathbb{R}^3}k_{M,2}^{\alpha \beta(1)}(v,u_{\parallel})\,g^\alpha(v + \frac{2m^{\beta}}{m^{\alpha}+m^{\beta}}u_{\parallel}) \,d u_{\parallel},
	\end{equation*}
	and \eqref{lemerew4311234} is expressed as
	\begin{eqnarray*}
		\begin{split}
	\sum_{\beta=A,B} \int_{\mathbb{R}^3}\frac{w(v)}{w(v+\frac{2m^{\beta}}{m^{\alpha}+m^{\beta}}u_{\parallel})}k_{M,2}^{\alpha \beta(1)}(v,u_{\parallel})g^{\alpha}(v+\frac{2m^{\beta}}{m^{\alpha}+m^{\beta}}u_{\parallel}) d u_{\parallel}.
		\end{split}
	\end{eqnarray*}
	We denote $\hat{c}=\frac{m^\beta}{4\theta_M}[(\frac{1+\sqrt{\bar{c}}}{2})^2+(\frac{1-\sqrt{\bar{c}}}{2})^2]$,
	since $(\frac{1+\sqrt{\bar{c}}}{2}):(\frac{1+\sqrt{\bar{c}}}{2})=(\frac{1-\sqrt{\bar{c}}}{2}):(\frac{1-\sqrt{\bar{c}}}{2})$, we can write the exponents as
	\begin{equation*}
		e^{-\frac{m^\beta}{4\theta_M}[| \zeta_{\perp}|^2+|\psi_{\perp}|^2]}=e^{-\hat{c}|v_{\perp}+u_{\perp}|^2}.
	\end{equation*}
	
	\noindent By a direct calculation, it shows that
	\begin{equation}\label{0CET2PARALCOMPAR}
		(\frac{1+\sqrt{\bar{c}}}{2}):(\frac{1}{2}+\frac{\sqrt{\bar{c}}(m^{\beta}-m^{\alpha})}{2(m^{\alpha}+m^{\beta})})>(\frac{1-\sqrt{\bar{c}}}{2}):(\frac{1}{2}-\frac{\sqrt{\bar{c}}(m^{\beta}-m^{\alpha})}{2(m^{\alpha}+m^{\beta})}).
	\end{equation}
	This indicates that if we regard $v_{\parallel}$ and $u_{\parallel}$ as a set of basis, the two vectors $\zeta_{\parallel}$ and $	\psi_{\parallel}$ are linearly independent. But,  vectors $\zeta_{\perp}$ and $\psi_{\perp}$ are not linearly independent which means there is no decay about $u_{\perp}$ and $v_{\perp}$.
	
	By comparing the expressions for $\zeta_{\parallel}$ and $\psi_{\parallel}$ in \eqref{cgvab4343111x} and \eqref{cgvab4343111p}, one can see that there exist constants $0<c_{v_\parallel},c_{u_\parallel}<1$ such that
	\begin{equation}\label{cccu+u-}
		\begin{split}
			c_{v_\parallel}(\frac{1+\sqrt{\bar{c}}}{2})=&(\frac{1-\sqrt{\bar{c}}}{2}),\\
			c_{u_\parallel}(\frac{1}{2}+\frac{\sqrt{\bar{c}}(m^{\beta}-m^{\alpha})}{2(m^{\alpha}+m^{\beta})})=&(\frac{1}{2}-\frac{\sqrt{\bar{c}}(m^{\beta}-m^{\alpha})}{2(m^{\alpha}+m^{\beta})}).
		\end{split}
	\end{equation}
	From \eqref{cccu+u-}, it can be obtained that
	\begin{equation*}
		\begin{split}
			c_{u,-}:=&c_{v_\parallel}(\frac{1}{2}+\frac{\sqrt{\bar{c}}(m^{\beta}-m^{\alpha})}{2(m^{\alpha}+m^{\beta})})-(\frac{1}{2}-\frac{\sqrt{\bar{c}}(m^{\beta}-m^{\alpha})}{2(m^{\alpha}+m^{\beta})})<0,\\
			c_{v,-}:=&(\frac{1-\sqrt{\bar{c}}}{2})-c_{u_\parallel}(\frac{1+\sqrt{\bar{c}}}{2})<0.
		\end{split}
	\end{equation*}
	We denote
	\begin{equation*}
		\begin{split}
			\mathbf{c}_{u,+}=\frac{1}{\sqrt{2}}(c_{v_\parallel}\zeta_{\parallel}+\psi_{\parallel}), \qquad \mathbf{c}_{u,-}=\frac{1}{\sqrt{2}}(c_{v_\parallel}\zeta_{\parallel}-\psi_{\parallel})=\frac{c_{u,-}}{\sqrt{2}}u_{\parallel},\\
			\mathbf{c}_{v,+}=\frac{1}{\sqrt{2}}(c_{u_\parallel}\zeta_{\parallel}+\psi_{\parallel}), \qquad \mathbf{c}_{v,-}=\frac{1}{\sqrt{2}}(\psi_{\parallel}-c_{u_\parallel}\zeta_{\parallel})=\frac{c_{v,-}}{\sqrt{2}}v_{\parallel},
		\end{split}
	\end{equation*}
	which gives the equality that
	\begin{equation*}
		\begin{split}
			(c_{v_\parallel})^2|\zeta_{\parallel}|^2+|\psi_{\parallel}|^2= |\mathbf{c}_{u,+}|^2+|\mathbf{c}_{u,-}|^2=|\mathbf{c}_{u,+}|^2+\frac{(c_{u,-})^2}{2}|u_{\parallel}|^2,\\
			(c_{u_\parallel})^2|\zeta_{\parallel}|^2+|\psi_{\parallel}|^2= |\mathbf{c}_{v,+}|^2+|\mathbf{c}_{v,-}|^2=|\mathbf{c}_{v,+}|^2+\frac{(c_{v,-})^2}{2}|v_{\parallel}|^2.
		\end{split}
	\end{equation*}
	The exponential term about $|\zeta_{\parallel}|,|\zeta_{\parallel}|,|\psi_{\perp}|,|\psi_{\perp}|$, in \eqref{CE2BDS249esD1} can be bounded as follows:
	\begin{equation*}
		e^{-\frac{1}{4}\Big[|\psi_{\parallel}|^2+|\zeta_{\parallel}|^2+|\psi_{\perp}|^2+|\zeta_{\perp}|^2\Big]} \leq e^{-\bar{c}\Big[|v_{\parallel}|^2+|u_{\parallel}|^2+|u_{\perp}+v_{\perp}|^2\Big]}.
	\end{equation*}
	Therefore, \eqref{CE2easyPT} can be finally bounded by
	\begin{align}
		& \int_{\mathbb{R}^3} \frac{1}{|u_{\parallel}|} e^{-\bar{c}_1\big[|v_{\parallel}|^2+|u_{\parallel}|^2\big]} g^\alpha(v + \frac{2m^{\beta}}{m^{\alpha}+m^{\beta}}u_{\parallel})\notag \\
		&\quad \times \int_{\mathbb{R}^2} e^{-\hat{c}| u_{\perp}+v_{\perp}|^2} \left[|u_{\parallel}|^2 + |u_{\perp}|^2\right]^{\frac{\gamma - 1}{2}} \{1-\chi_m\}\big(\sqrt{|u_{\parallel}|^2 + |u_{\perp}|^2}\big) \frac{b^{ \alpha \beta}(\theta)}{|\cos \theta|} \, du_{\perp} du_{\parallel},\label{CE2BDSesD1}
	\end{align}
	where
	\begin{equation*}
		\bar{c}_1=\frac{m^{\beta}}{64\theta_M}\min\{(c_{u,-})^2,(c_{v,-})^2,(\tilde{c}_{u,-})^2,(\tilde{c}_{v,-})^2\}>0,\quad \quad \hat{c}=\frac{m^{\beta}(1+\bar{c})}{8\theta_M}>0.
	\end{equation*}
	Therefore, $k_{M,2}^{\alpha \beta(1)}(v,u_{\parallel})$ is bounded by
	\begin{align}
		\big|k_{M,2}^{\alpha \beta(1)}(v,u_{\parallel})\big|&\leq\frac{C}{|u_{\parallel}|} e^{-\bar{c}_1\big[|v_{\parallel}|^2+|u_{\parallel}|^2\big]}
		\int_{\mathbb{R}^2} e^{-\hat{c}| u_{\perp}+v_{\perp}|^2} \left[|u_{\parallel}|^2 + |u_{\perp}|^2\right]^{\frac{\gamma - 1}{2}} \notag\\
		&\hspace{3cm}\times \{1-\chi_m\}\big(\sqrt{|u_{\parallel}|^2 + |u_{\perp}|^2}\big) \frac{b^{ \alpha \beta}(\theta)}{|\cos \theta|} \, du_{\perp}. \label{bdskm21}
	\end{align}
	
	Next, we consider integral for $k_{M,2}^{\alpha \beta(2)}(v,u_{\perp},u_{\parallel})$,$(\alpha \neq \beta)$. Note that in \eqref{CE2hadPT}, the exponents are different. To clearly characterize their exponential decay, we let
	\begin{eqnarray*}
		\begin{split}
			\xi_{\parallel} + \xi_{\perp} &= \frac{\sqrt{\bar{c}m^\alpha}+\sqrt{m^\beta}}{2}v + (\frac{\sqrt{m^\beta}}{2}+\frac{\sqrt{\bar{c}m^{\alpha}}m^{\beta}}{m^{\alpha}+m^{\beta}})u_{\parallel}+\frac{\sqrt{m^\beta}}{2}u_{\perp}, \quad
			\xi_{\parallel} \parallel u_{\parallel},\quad   \xi_{\perp} \parallel u_{\perp},\\
			\eta_{\parallel}+\eta_{\perp}&=\frac{\sqrt{m^\beta}-\sqrt{\bar{c}m^\alpha}}{2}v+(\frac{\sqrt{m^\beta}}{2}-\frac{\sqrt{\bar{c}m^{\alpha}}m^{\beta}}{m^{\alpha}+m^{\beta}})u_{\parallel}+\frac{\sqrt{m^\beta}}{2}u_{\perp}, \quad
			\eta_{\parallel} \parallel u_{\parallel},\quad   \eta_{\perp} \parallel u_{\perp}.
		\end{split}
	\end{eqnarray*}
	It follows that
	\begin{align*}
		\sqrt{\bar{c}m^\alpha}(v + \frac{2m^{\beta}}{m^{\alpha}+m^{\beta}}u_{\parallel})&= \xi_{\parallel} + \xi_{\perp}-\big[\eta_{\parallel}+\eta_{\perp}\big], \\
		\sqrt{m^\beta}(v +u_{\parallel}+ u_{\perp}) &= \xi_{\parallel} + \xi_{\perp}+\big[\eta_{\parallel}+\eta_{\perp}\big].
	\end{align*}
	To keep the notation concise, we make the following variable substitution
	\begin{eqnarray}\label{etetetcb}
		\begin{split}
			\eta_{\parallel} &= \frac{\sqrt{m^\beta}-\sqrt{\bar{c}m^\alpha}}{2}v_{\parallel} + (\frac{\sqrt{m^\beta}}{2}-\frac{\sqrt{\bar{c}m^{\alpha}}m^{\beta}}{m^{\alpha}+m^{\beta}})u_{\parallel}, \\
			\eta_{\perp} &= \frac{\sqrt{m^\beta}-\sqrt{\bar{c}m^\alpha}}{2}v_{\perp} +\frac{\sqrt{m^\beta}}{2}u_{\perp},
		\end{split}
	\end{eqnarray}
	and
	\begin{eqnarray}\label{xixixicb}
		\begin{split}
			\xi_{\parallel} &= \frac{\sqrt{m^\beta}+\sqrt{\bar{c}m^\alpha}}{2}v_{\parallel} + (\frac{\sqrt{m^\beta}}{2}+\frac{\sqrt{\bar{c}m^{\alpha}}m^{\beta}}{m^{\alpha}+m^{\beta}})u_{\parallel}, \\
			\xi_{\perp} &= \frac{\sqrt{m^\beta}+\sqrt{\bar{c}m^\alpha}}{2}v_{\perp} +\frac{\sqrt{m^\beta}}{2}u_{\perp},
		\end{split}
	\end{eqnarray}
	where $v=v_{\parallel}+v_{\perp}, v_{\parallel} \parallel u_{\parallel}, v_{\perp}\parallel u_{\perp}$.
	Then, we deduce that the exponents in \eqref{CE2hadPT} are
	\begin{equation*}
		-\frac{1}{4\theta_M}\Big[|\xi_{\parallel}|^2+|\xi_{\perp}|^2+|\eta_{\parallel}|^2+|\eta_{\perp}|^2\Big] .
	\end{equation*}
	Therefore, \eqref{CE2hadPT} can be written as
	\begin{align}
		& \int_{\mathbb{R}^3} \frac{1}{|u_{\parallel}|^2} e^{-\frac{1}{4\theta_M}\big[|\xi_{\parallel}|^2+|\eta_{\parallel}|^2\big]}\int_{\mathbb{R}^2}b^{ \alpha \beta}(\theta)\chi_m\big(\sqrt{|u_{\parallel}|^2 + |u_{\perp}|^2}\big)  \notag\\
		&\quad \times  e^{-\frac{1}{4\theta_M}\big[|\xi_{\perp}|^2+|\eta_{\perp}|^2\big]} \left[|u_{\parallel}|^2 + |u_{\perp}|^2\right]^{\frac{\gamma}{2}}   g^\beta(v + u_{\perp}+\frac{m^{\beta}-m^{\alpha}}{m^{\alpha}+m^{\beta}}u_{\parallel})\, du_{\perp} du_{\parallel} \notag\\
		=& \int_{\mathbb{R}^3} \frac{1}{|u_{\parallel}|} e^{-\frac{1}{4\theta_M}\big[|\xi_{\parallel}|^2+|\eta_{\parallel}|^2\big]}\int_{\mathbb{R}^2} \frac{b^{ \alpha \beta}(\theta)}{|\cos \theta|}\chi_m\big(\sqrt{|u_{\parallel}|^2 + |u_{\perp}|^2}\big)  \notag\\
		&\quad \times  e^{-\frac{1}{4\theta_M}\big[|\xi_{\perp}|^2+|\eta_{\perp}|^2\big] } \left[|u_{\parallel}|^2 + |u_{\perp}|^2\right]^{\frac{\gamma - 1}{2}}g^\beta(v + u_{\perp}+\frac{m^{\beta}-m^{\alpha}}{m^{\alpha}+m^{\beta}}u_{\parallel})  \, du_{\perp} du_{\parallel}.\label{CE2BDSHHD2}
	\end{align}
	We define the integral kernel $k_{M,2}^{\alpha \beta(2)}(v,u_{\perp},u_{\parallel})$,\,$(\alpha \neq \beta)$:
	\begin{align}
		k_{M,2}^{\alpha \beta(2)}(v,u_{\perp},u_{\parallel})=:&  \frac{1}{|u_{\parallel}|} e^{-\frac{1}{4\theta_M}\big[|\eta_{\parallel}|^2+|\xi_{\parallel}|^2+|\eta_{\perp}|^2+|\xi_{\perp}|^2\big]} \chi_m\big(\sqrt{|u_{\parallel}|^2 + |u_{\perp}|^2}\big)\notag\\
		&\hspace{4cm}\times \left[|u_{\parallel}|^2 + |u_{\perp}|^2\right]^{\frac{\gamma - 1}{2}} \frac{b^{ \alpha \beta}(\theta)}{|\cos \theta|}.\label{}
	\end{align}
	Then \eqref{CE2hadPT} is expressed as
	\begin{equation*}
		\int_{\mathbb{R}^3\times\mathbb{R}^2}k_{M,2}^{\alpha \beta(2)}(v,u_{\perp},u_{\parallel})\,g^\beta(v + u_{\perp}+\frac{m^{\beta}-m^{\alpha}}{m^{\alpha}+m^{\beta}}u_{\parallel}) \,d u_{\perp}\,d u_{\parallel},
	\end{equation*}
	and $\eqref{lem431123512}_1$ is expressed as
	\begin{eqnarray}
		\begin{split}
			 \int_{\mathbb{R}^3\times\mathbb{R}^2}\frac{w(v)}{w(v + u_{\perp}+\frac{m^{\beta}-m^{\alpha}}{m^{\alpha}+m^{\beta}}u_{\parallel})}k_{M,2}^{\alpha \beta(2)}(v,u_{\perp},u_{\parallel})g^\beta(v + u_{\perp}+\frac{m^{\beta}-m^{\alpha}}{m^{\alpha}+m^{\beta}}u_{\parallel})\,d u_{\perp}\, d u_{\parallel}.
		\end{split}
	\end{eqnarray}
	
	By direct calculation, it shows that  $(\frac{\sqrt{m^\beta}}{2}-\frac{\sqrt{\bar{c}m^{\alpha}}m^{\beta}}{m^{\alpha}+m^{\beta}})>0$, and
	\begin{equation}\label{CET2PARALCOMPAR}
		(\frac{\sqrt{m^\beta}-\sqrt{\bar{c}m^\alpha}}{2}):(\frac{\sqrt{m^\beta}}{2}-\frac{\sqrt{\bar{c}m^{\alpha}}m^{\beta}}{m^{\alpha}+m^{\beta}})>(\frac{\sqrt{m^\beta}+\sqrt{\bar{c}m^\alpha}}{2}):(\frac{\sqrt{m^\beta}}{2}+\frac{\sqrt{\bar{c}m^{\alpha}}m^{\beta}}{m^{\alpha}+m^{\beta}}).
	\end{equation}
	This indicates that if we regard $u_{\parallel}$ and $v_{\parallel}$ as a set of basis, the two vectors $\xi_{\parallel}$ and $\eta_{\parallel}$ are linearly independent. Similarly, we can infer that under the basis $u_{\perp}$ and $v_{\perp}$, vectors $\xi_{\perp}$ and $\eta_{\perp}$ are also linearly independent.
	
	By comparing the expressions for $\xi_{\parallel}$ and $\eta_{\parallel}$ in \eqref{etetetcb} and \eqref{xixixicb}, one can see that there exist constants $0<d_{v_\parallel},d_{u_\parallel}<1$ such that
	\begin{equation}
		\begin{split}
			d_{v_\parallel}(\frac{\sqrt{m^\beta}+\sqrt{\bar{c}m^\alpha}}{2})=&\frac{\sqrt{m^\beta}-\sqrt{\bar{c}m^\alpha}}{2},\\
			d_{u_\parallel}(\frac{\sqrt{m^\beta}}{2}+\frac{\sqrt{\bar{c}m^{\alpha}}m^{\beta}}{m^{\alpha}+m^{\beta}})=&\frac{\sqrt{m^\beta}}{2}-\frac{\sqrt{\bar{c}m^{\alpha}}m^{\beta}}{m^{\alpha}+m^{\beta}}.
		\end{split}
	\end{equation}
	From \eqref{CET2PARALCOMPAR}, it can be obtained that
	\begin{equation*}
		\begin{split}
			d_{u,-}:=&d_{v_\parallel}(\frac{\sqrt{m^\beta}}{2}+\frac{\sqrt{\bar{c}m^{\alpha}}m^{\beta}}{m^{\alpha}+m^{\beta}})-(\frac{\sqrt{m^\beta}}{2}-\frac{\sqrt{\bar{c}m^{\alpha}}m^{\beta}}{m^{\alpha}+m^{\beta}})>0,\\
			d_{v,-}:=&(\frac{\sqrt{m^\beta}-\sqrt{\bar{c}m^\alpha}}{2})-d_{u_\parallel}(\frac{\sqrt{m^\beta}+\sqrt{\bar{c}m^\alpha}}{2})>0.
		\end{split}
	\end{equation*}
	We denote
	\begin{equation*}
		\begin{split}
			\mathbf{d}_{u,+}=\frac{1}{\sqrt{2}}(d_{v_\parallel}\xi_{\parallel}+\eta_{\parallel}), \qquad \mathbf{d}_{u,-}=\frac{1}{\sqrt{2}}(d_{v_\parallel}\xi_{\parallel}-\eta_{\parallel})=\frac{d_{u,-}}{\sqrt{2}}u_{\parallel},\\
			\mathbf{d}_{v,+}=\frac{1}{\sqrt{2}}(d_{u_\parallel}\xi_{\parallel}+\eta_{\parallel}), \qquad \mathbf{d}_{v,-}=\frac{1}{\sqrt{2}}(\eta_{\parallel}-d_{u_\parallel}\xi_{\parallel})=\frac{d_{v,-}}{\sqrt{2}}v_{\parallel},
		\end{split}
	\end{equation*}
	which gives the equality that
	\begin{equation*}
		\begin{split}
			(d_{v_\parallel})^2|\xi_{\parallel}|^2+|\eta_{\parallel}|^2= |\mathbf{d}_{u,+}|^2+|\mathbf{d}_{u,-}|^2=|\mathbf{d}_{u,+}|^2+\frac{(d_{u,-})^2}{2}|u_{\parallel}|^2,\\
			(d_{u_\parallel})^2|\xi_{\parallel}|^2+|\eta_{\parallel}|^2= |\mathbf{d}_{v,+}|^2+|\mathbf{d}_{v,-}|^2=|\mathbf{d}_{v,+}|^2+\frac{(d_{v,-})^2}{2}|v_{\parallel}|^2.
		\end{split}
	\end{equation*}
	Similarly, it can be observed that
	\begin{equation*}
		(\frac{\sqrt{m^\beta}-\sqrt{\bar{c}m^\alpha}}{2}) :\frac{\sqrt{m^\beta}}{2}< (\frac{\sqrt{m^\beta}+\sqrt{\bar{c}m^\alpha}}{2}):\frac{\sqrt{m^\beta}}{2}.
	\end{equation*}
	Through a process analogous to the above that there exist constants
	$0<d_{v_{\perp}}<1$ and $\tilde{d}_{u,-}, \tilde{d}_{v,-}>0$, as well as vectors $\tilde{\mathbf{d}}_{u,+},\tilde{\mathbf{d}}_{v,+}$, such that
	\begin{equation*}
		\begin{split}
			(d_{v_\perp})^2|\xi_{\perp}|^2+|\eta_{\perp}|^2&= |\mathbf{d}_{u,+}|^2+|\mathbf{d}_{u,-}|^2=|\mathbf{d}_{u,+}|^2+\frac{(d_{u,-})^2}{2}|u_{\perp}|^2,\\
			|\xi_{\perp}|^2+|\eta_{\perp}|^2&= |\mathbf{d}_{v,+}|^2+|\mathbf{d}_{v,-}|^2=|\mathbf{d}_{v,+}|^2+\frac{(d_{v,-})^2}{2}|v_{\perp}|^2.
		\end{split}
	\end{equation*}
	The exponential term about $\xi_{\parallel},\xi_{\perp},\eta_{\parallel},\eta_{\perp}$, in \eqref{CE2BDSHHD2} can be bounded as follows:
	\begin{equation*}
		e^{-\frac{1}{4}\Big[|\xi_{\parallel}|^2+|\xi_{\perp}|^2+|\eta_{\parallel}|^2+|\eta_{\perp}|^2\Big]} < e^{-\bar{c}_2\Big[|v_{\parallel}|^2+|v_{\perp}|^2+|u_{\parallel}|^2+|u_{\perp}|^2\Big]},
	\end{equation*}
	where
	\begin{equation*}
		\bar{c}_2=\frac{1}{64\theta_M}\min\{(d_{u,-})^2,(d_{v,-})^2,(\tilde{d}_{u,-})^2,(\tilde{d}_{v,-})^2\}>0.
	\end{equation*}
	Therefore, \eqref{CE2hadPT} is bounded by
	\begin{align*}
		& \int_{\mathbb{R}^3} \frac{1}{|u_{\parallel}|} e^{-\bar{c}_2\big[|v_{\parallel}|^2+|u_{\parallel}|^2\big]}\int_{\mathbb{R}^2} \frac{b^{ \alpha \beta}(\theta)}{|\cos \theta|}\chi_m\big(\sqrt{|u_{\parallel}|^2 + |u_{\perp}|^2}\big)  \notag\\
		&\quad \times  e^{-\bar{c}_2\big[|v_{\perp}|^2+|u_{\perp}|^2\big] } \left[|u_{\parallel}|^2 + |u_{\perp}|^2\right]^{\frac{\gamma - 1}{2}}f^\beta(v + u_{\perp}+\frac{m^{\beta}-m^{\alpha}}{m^{\alpha}+m^{\beta}}u_{\parallel})  \, du_{\perp} du_{\parallel}.
	\end{align*}
	The second  integral kernel $k_{M,2}^{\alpha \beta(2)}(v,u_{\perp},u_{\parallel})$ is bounded by
	\begin{align}
		\big|k_{M,2}^{\alpha \beta(2)}(v,u_{\perp},u_{\parallel})\big|&\leq \frac{C}{|u_{\parallel}|} e^{-\bar{c}_2\big[|v|^2+|u_{\parallel}|^2+|u_{\perp}|^2\big]} \left[|u_{\parallel}|^2 + |u_{\perp}|^2\right]^{\frac{\gamma - 1}{2}}\notag  \\
		&\hspace{3.5cm} \times  \frac{b^{ \alpha \beta}(\theta)}{|\cos \theta|}\{1-\chi_m\}\big(\sqrt{|u_{\parallel}|^2 + |u_{\perp}|^2}\big)\notag\\
		&\leq \frac{C_m}{|u_{\parallel}|} e^{-\bar{c}_2\big[|v|^2+|u|^2\big]} . \label{final222kkk} 
	\end{align}
	
		\noindent\textbf{The degenerate part*}.\hspace{0.1cm}	Finally, we turn to the degenerate $(\alpha=\beta)$ term  in the ``Hybrid part''.
	
  The integration kernel‌ $k_{M,2}^{\alpha \alpha(2)}$ no longer adheres to exponential decay \eqref{final222kkk},
	but instead satisfies the decay  analogous to $k_{M,2}^{\alpha \beta(1)}$. 
	\begin{proof}
		Notice by \eqref{etetetcb} and \eqref{xixixicb}, when $\alpha=\beta$, we have
		\begin{eqnarray*}
			\begin{split}
				\xi_{\parallel} &= \sqrt{m^\alpha}\frac{(1+\sqrt{\bar{c}})}{2}(v_{\parallel}+u_{\parallel}), \hspace{1.1cm}  \eta_{\parallel}=\sqrt{m^\alpha}\frac{(1-\sqrt{\bar{c}})}{2}(v_{\parallel}+u_{\parallel})\\
				\xi_{\perp} &= \sqrt{m^\alpha}\frac{(1+\sqrt{\bar{c}})}{2}v_{\perp} +\frac{\sqrt{m^{\alpha}}}{2}u_{\perp},\hspace{0.5cm}\eta_{\perp} = \sqrt{m^\alpha}\frac{(1-\sqrt{\bar{c}})}{2}v_{\perp} +\frac{\sqrt{m^{\alpha}}}{2}u_{\perp}
			\end{split}
		\end{eqnarray*}
		 Notice that $\xi_{\parallel}$ and $\eta_{\parallel}$ are directly proportional. Therefore, they are not linearly independent. There is not exponential decay for $v_{\parallel}$ and $u_{\parallel}$ individually. But only  exhibits exponential decay with respect to $(v_{\parallel}+u_{\parallel})$. This results in $k_{M,2}^{\alpha \alpha(2)}$ no longer having exponential decay with respect to $v$. 
		
		Comparing \eqref{CE2hadPTeq} and \eqref{CE2easyPT}, ‌we note that the independent variables  are $v+\frac{2m^{\beta}}{m^{\alpha}+m^{\beta}}u_{\parallel}$ and $v+u_{\perp}$ in these two cases, respectively. Since ``$du\,d\omega$'' has five degrees of freedom, when the independent variable involves $\frac{2m^{\beta}}{m^{\alpha}+m^{\beta}}u_{\parallel}$, we can assign $u_{\parallel}$ to have three degrees of freedom, while $u_{\perp}$ retains two degrees of freedom. When the independent variable involves $u_{\perp}$, we must assign $u_{\perp}$ to have three degrees of freedom while $u_{\parallel}$ retains two degrees of freedom.
		
		Recall that in \eqref{CE2easyPT}, we assign $u_{\parallel}$ to have three degrees of freedom,  $u \in \mathbb{R}^3$, $\omega \in \mathbb{S}^2$, $u_{\parallel} \in \mathbb{R}^3$, $u_{\perp} \in \mathbb{R}^2$, and
		\begin{equation*}
			|u_{\parallel}|^{2}(|u_{\parallel}|^{-1}d|u_{\parallel}|)d\omega = du_{\parallel} ,\qquad du = 2(|u_{\parallel}|^{-1} d|u_{\parallel}|) du_{\perp}. 
		\end{equation*}
		But in \eqref{CE2hadPTeq}, we assign $u_{\perp}$ to have three degrees of freedom,  $u \in \mathbb{R}^3$, $\omega' \in \mathbb{S}^2$, $u_{\perp} \in \mathbb{R}^3$, $u_{\parallel} \in \mathbb{R}^2$, and
		\begin{equation}\label{cgvbuprepdd}
			|u_{\perp}|^{2}(|u_{\perp}|^{-1}d|u_{\perp}|)d\omega' = du_{\perp} ,\qquad du =2 (|u_{\perp}|^{-1} d|u_\perp|) du_{\parallel}. 
		\end{equation}
		Let the three coordinates of the spherical coordinate transformation be $(\tilde{r},\tilde{\theta},\tilde{\phi})$, where $\tilde{r}\geq0$ is radial distance, $0\leq\tilde{\theta}\leq \pi$ is polar angle and $0\leq\tilde{\phi}\leq 2 \pi$ is azimuthal angle.  
		
		In \eqref{CE2easyPT}, we take $(|u_{\parallel}|,\theta,\tilde{\phi})\rightarrow(\tilde{r},\tilde{\theta},\tilde{\phi})$. 
		
		The differential area element on the unit sphere is $d\omega=\sin\theta d\theta d\tilde{\phi}$.
		
		In \eqref{CE2hadPTeq}, we take $(|u_{\perp}|,\varphi,\tilde{\phi})\rightarrow(\tilde{r},\tilde{\theta},\tilde{\phi})$. 
		
		The differential area element on the unit sphere is $d\omega'=\sin\varphi d\varphi d\tilde{\phi}$.

		\begin{figure}
			\centering
			\includegraphics[width=9cm,height=5.5cm]{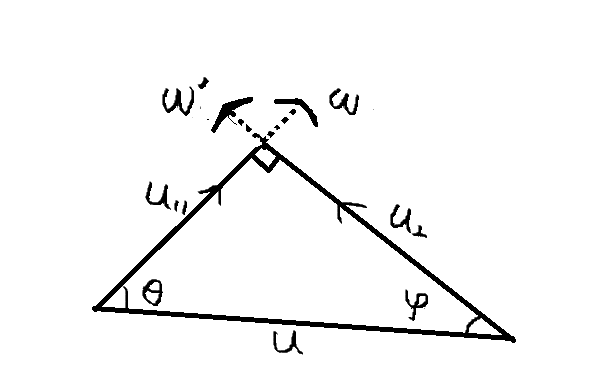}
			\caption{The illustration of variable substitution }
		\end{figure}
		From the picture (Figure 2), we have $|u_{\parallel}|=|u\cos\theta|$, $|u_{\perp}|=|u\cos\varphi|$, and $\varphi+\theta=\frac{\pi}{2}$. Therefore, it yields 
		\begin{equation*}
			du_{\perp}=|u_{\perp}|^{2}(|u_{\perp}|^{-1}d|u_{\perp}|)d\omega' = |u_{\perp}||u_{\parallel}|(|u_{\perp}|^{-1}d|u_{\perp}|)d\omega
		\end{equation*}
		Together with \eqref{cgvbuprepdd}, for , we have
		\begin{equation}\label{casaeqbedwdu}
			du d\omega = 2|u_{\parallel}|^{-1}|u_{\perp}|^{-1} du_{\perp} du_{\parallel}
		\end{equation}
	
	Through a process analogous to \eqref{0CET2PARALCOMPAR}-\eqref{bdskm21}, we show that \eqref{CE2hadPTeq}
	can be expressed as
		\begin{equation*}
			\sum_{\beta=A,B}	\int_{\mathbb{R}^3}k_{M,2}^{\alpha \alpha(2)}(v,u_{\perp})\,g^\alpha(v + u_{\perp}) \,d u_{\perp},
		\end{equation*}
		 and
		 the integral kernel is 
		\begin{align}
			k_{M,2}^{\alpha \alpha(2)}(v,u_{\perp})	&=: \frac{1}{|u_{\perp}|} e^{-\frac{1}{4\theta_M}[|\xi_{\perp}|^2+|\eta_{\perp}|^2] } \int_{\mathbb{R}^2} e^{-\frac{1}{4\theta_M}[|\xi_{\parallel}|^2+|\eta_{\parallel}|^2]}\notag\\
			&\hspace{2.5cm}\times  \left[|u_{\parallel}|^2 + |u_{\perp}|^2\right]^{\frac{\gamma - 1}{2}} \chi_m\big(\sqrt{|u_{\parallel}|^2 + |u_{\perp}|^2}\big) \frac{b^{ \alpha \beta}(\theta)}{|\cos \theta|} \, du_{\parallel}. \label{socd264}
		\end{align}
		where $u_{\perp} \in \mathbb{R}^3, u_{\parallel} \in \mathbb{R}^2$. It is further bounded by
		\begin{align}
			|k_{M,2}^{\alpha \alpha(2)}(v,u_{\perp})|	&\leq \frac{C}{|u_{\perp}|} e^{-c\big[|v_{\perp}|^2+|u_{\perp}|^2\big]} \int_{\mathbb{R}^2} e^{-c|u_{\parallel} + v_{\parallel}|^2}\notag  \\
			&\hspace{2.5cm}\times  \left[|u_{\parallel}|^2 + |u_{\perp}|^2\right]^{\frac{\gamma - 1}{2}} \chi_m\big(\sqrt{|u_{\parallel}|^2 + |u_{\perp}|^2}\big) \frac{b^{ \alpha \beta}(\theta)}{|\cos \theta|} \, du_{\parallel}.\label{} 
		\end{align}
		
		Therefore, when $m^{\alpha}=m^{\beta}$, the integral kernel $k_{\alpha \beta}^{(2),r}(v,u_{\perp},u_{\parallel})$ degenerates into the integral kernel $k_{\alpha \alpha}^{(2),r}(v,u_{\perp})$, which is almost equal to $k_{\alpha \alpha}^{(1),r}(v,u_{\parallel})$ in the symmetric sense.
	\end{proof}
	
	$\mathbf{Step 2}$.  The decay estimate for $K_{M,2,w}^{\alpha,1-\chi}$. \\
	First, we prove the decay estimate for the part of $k_{M,2}^{\alpha \beta(1)}(v,u_{\parallel})$. The integration domain $D$ can be devided
	as follows:
	\begin{align*}
		D=\{(v,u)\in \mathbb{R}^3 \times \mathbb{R}^3\}=&\Biggl\{\Big[(|u_\parallel|\geq \frac{|u|}{\sqrt{2}})\cap(|u_\parallel|\geq \frac{|v|}{8})\Big]\bigcup\Big[(|u_\perp|\geq \frac{|u|}{\sqrt{2}})\cap(|u_\parallel|\geq \frac{|v|}{8})\Big]\\
		&\hspace{0.4cm}\bigcup\Big[(|u_\parallel|\geq \frac{|u|}{\sqrt{2}})\cap(|u_\perp|\geq \frac{|v|}{8})\Big]\bigcup\Big[(|u_\perp|\geq \frac{|u|}{2})\cap(|u_\perp|\geq \frac{|v|}{8})\Big]\\
		&\hspace{0.4cm}\bigcup\Big[(|v_\parallel|\geq \frac{|v|}{\sqrt{2}})\cap(|u_\parallel|\leq \frac{|v|}{8})\cap(|u_\perp|\leq \frac{|v|}{8})\Big]\\
		&\hspace{0.4cm}\bigcup\Big[(|v_\perp|\geq \frac{|v|}{\sqrt{2}})\cap(|u_\parallel|\leq \frac{|v|}{8})\cap(|u_\perp|\leq \frac{|v|}{8})\Big]\Biggr\}.
	\end{align*}
	Note that the above integral domains  satisfy
	\begin{align*}
		&\Big[(|u_\parallel|\geq \frac{|u|}{\sqrt{2}})\cap(|u_\parallel|\geq \frac{|v|}{8})\Big]  \supseteq \Big[(|u_\parallel|\geq \frac{|u|}{\sqrt{2}})\cap(|u_\perp|\geq \frac{|v|}{8})\Big],\\
		&\Big[(|u_\perp|\geq \frac{|u|}{\sqrt{2}})\cap(|u_\perp|\geq \frac{|v|}{8})\Big]  \supseteq \Big[(|u_\perp|\geq \frac{|u|}{\sqrt{2}})\cap(|u_\parallel|\geq \frac{|v|}{8})\Big].
	\end{align*}
	By the cutoff function, we have $|v-v_*|=|u|\geq 2 m$.  Without loss of generality, we assume $|v|\geq 1$. 
	
	\noindent {\bf The proof of $\eqref{a1a4main236}_2$}. It can be divided into the following four cases for discussion.

	\noindent {\bf Case A1. $\Big(|u_\parallel|\geq \frac{|u|}{\sqrt{2}}\Big)\cap\Big(|u_\parallel|\geq \frac{|v|}{8}\Big)$}.

	\noindent Since $\{1-\chi_m\}(\sqrt{|u_{\parallel}|^2 + |u_{\perp}|^2}) = 1-\chi_{m}$ vanishes near the origin and
	$|B(\theta)| \leq C|\cos \theta|$,  for the chosen small $1>m > 0$, we get
	\begin{equation}\label{cuttoffdelta}
		\begin{split}
			&\int_{\mathbb{R}^2} e^{-\hat{c}|u_{\perp} + v_{\perp}|^2} \left[|u_{\parallel}|^2 + |u_{\perp}|^2\right]^{\frac{\gamma - 1}{2}} \chi_m\big(\sqrt{|u_{\parallel}|^2 + |u_{\perp}|^2}\big) \frac{b^{ \alpha \beta}(\theta)}{|\cos \theta|} \, du_{\perp} \\
			&\hspace{0.5cm}\leq C_m \int_{\mathbb{R}^2} e^{-\hat{c}|u_\perp + v_\perp|^2} du_\perp < \infty.
		\end{split}
	\end{equation}
	Since
	\begin{equation*}
		w(v) \leq C w(v+\frac{2m^{\beta}}{m^{\alpha}+m^{\beta}}u_{\parallel}) w(\frac{2m^{\beta}}{m^{\alpha}+m^{\beta}}u_{\parallel}),
	\end{equation*}
	we deduce that
	\begin{equation}\label{contrlofP}
		\frac{w(v)}{w(v+\frac{2m^{\beta}}{m^{\alpha}+m^{\beta}}u_{\parallel})} \leq C  w(\frac{2m^{\beta}}{m^{\alpha}+m^{\beta}}u_{\parallel}) \leq C  e^{\frac{\bar{c}_1}{8}|u_{\parallel}|^2}.
	\end{equation}
	Notice that $|u_\parallel|\geq\frac{|v|}{8}\geq \frac{1}{8}$, it yields that
	\begin{equation*}
		\frac{w(v)}{w(v+\frac{2m^{\beta}}{m^{\alpha}+m^{\beta}}u_{\parallel})}(1+|v|+|u_{\parallel}|)^{1-\gamma}	\leq C e^{\frac{\bar{c}_1}{4}|u_{\parallel}|^2}.
	\end{equation*}
	Then, $\frac{w(v)}{w(v+\frac{2m^{\beta}}{m^{\alpha}+m^{\beta}}u_{\parallel})}|k_{M,2}^{\alpha \beta(1)}(v,u_{\parallel})|$ is bounded by
	\begin{align*}
		&\frac{w(v)}{w(v+\frac{2m^{\beta}}{m^{\alpha}+m^{\beta}}u_{\parallel})}\frac{1}{|u_{\parallel}|} e^{-\bar{c}_1[|v_{\parallel}|^2+|u_{\parallel}|^2]}
		\int_{\mathbb{R}^2} e^{-\hat{c}| u_{\perp}+v_{\perp}|^2} \left[|u_{\parallel}|^2 + |u_{\perp}|^2\right]^{\frac{\gamma - 1}{2}} \\
		&\hspace{5cm}\times \{1-\chi_m\}\big(\sqrt{|u_{\parallel}|^2 + |u_{\perp}|^2}\big) \frac{b^{ \alpha \beta}(\theta)}{|\cos \theta|} \, du_{\perp} \\
		\leq C& \frac{1}{|u_{\parallel}|} e^{-\frac{\bar{c}_1}{2}[|v_{\parallel}|^2+|u_{\parallel}|^2]} (1+|v|+|u_{\parallel}|)^{\gamma-1} .
	\end{align*}
	Therefore, we have completed the proof of $\eqref{a1a4main236}_2$ in case A1.

	\noindent {\bf Case A2. $\Big(|u_\perp|\geq \frac{|u|}{\sqrt{2}}\Big)\cap\Big(|u_\perp|\geq \frac{|v|}{8}\Big)$}.
	
	\noindent In this case $|u_\perp|\geq\frac{|v|}{8}\geq\frac{1}{8}$, it holds that
	\begin{align*}
		&\int_{\mathbb{R}^2} e^{-\hat{c}|u_{\perp} + v_{\perp}|^2} \left[|u_{\parallel}|^2 + |u_{\perp}|^2\right]^{\frac{\gamma - 1}{2}} \chi_m\big(\sqrt{|u_{\parallel}|^2 + |u_{\perp}|^2}\big) \frac{b^{ \alpha \beta}(\theta)}{|\cos \theta|} \, du_{\perp} \\
		&\hspace{0.4cm}\leq  C_{m}\int_{\mathbb{R}^2} e^{-\hat{c}|u_\perp + v_\perp|^2} |u_\perp|^{\gamma-1}du_\perp \\
		&\hspace{0.4cm}\leq C_{m}(1+|v|)^{\gamma-1}\int_{\mathbb{R}^2} e^{-\hat{c}|u_\perp + v_\perp|^2} du_\perp \\
		&\hspace{0.4cm}\leq C_{m}(1+|v|)^{\gamma-1},
	\end{align*}
	and if $|u_{\parallel}|>1$,
	\begin{equation*}
		\frac{w(v)}{w(v+\frac{2m^{\beta}}{m^{\alpha}+m^{\beta}}u_{\parallel})}(1+|u_{\parallel}|)^{1-\gamma} \leq C  w(\frac{2m^{\beta}}{m^{\alpha}+m^{\beta}}u_{\parallel}) (1+|u_{\parallel}|)^{1-\gamma}\leq C  e^{\frac{\bar{c}_1}{4}|u_{\parallel}|^2}.
	\end{equation*}
	if $|u_{\parallel}|\leq1$,
	\begin{equation*}
		\frac{w(v)}{w(v+\frac{2m^{\beta}}{m^{\alpha}+m^{\beta}}u_{\parallel})}(1+|u_{\parallel}|)^{1-\gamma} \leq C  w(\frac{2m^{\beta}}{m^{\alpha}+m^{\beta}}u_{\parallel}) (1+|u_{\parallel}|)^{1-\gamma}\leq C.
	\end{equation*}
	Then, we can bound $\frac{w(v)}{w(v+\frac{2m^{\beta}}{m^{\alpha}+m^{\beta}}u_{\parallel})}|k_{M,2}^{\alpha \beta(1)}(v,u_{\parallel})|$ by
	\begin{align*}
		&\frac{w(v)}{w(v+\frac{2m^{\beta}}{m^{\alpha}+m^{\beta}}u_{\parallel})}\frac{1}{|u_{\parallel}|} e^{-\bar{c}_1[|v_{\parallel}|^2+|u_{\parallel}|^2]}
		\int_{\mathbb{R}^2} e^{-\hat{c}| u_{\perp}+v_{\perp}|^2} \left[|u_{\parallel}|^2 + |u_{\perp}|^2\right]^{\frac{\gamma - 1}{2}} \\
		&\hspace{5cm}\times \{1-\chi_m\}\left(\sqrt{|u_{\parallel}|^2 + |u_{\perp}|^2}\right) \frac{b^{ \alpha \beta}(\theta)}{|\cos \theta|} \, du_{\perp} \\
		&\hspace{0.4cm}\leq C_{m} \frac{1}{|u_{\parallel}|} e^{-\frac{\bar{c}_1}{2}[|v_{\parallel}|^2+|u_{\parallel}|^2]} (1+|v|)^{\gamma-1}(1+|u_{\parallel}|)^{\gamma-1} \\
		&\hspace{0.4cm}\leq C_{m} \frac{1}{|u_{\parallel}|} e^{-\frac{\bar{c}_1}{2}[|v_{\parallel}|^2+|u_{\parallel}|^2]} (1+|v|+|u_{\parallel}|)^{\gamma-1}.
	\end{align*}
	Therefore, we have completed the proof of $\eqref{a1a4main236}_2$ in case A2.

	\noindent {\bf Case A3. $\Big(|v_\parallel|\geq \frac{|v|}{\sqrt{2}}\Big)\cap\Big(|u_\parallel|\leq \frac{|v|}{8}\Big)\cap\Big(|u_\perp|\leq \frac{|v|}{8}\Big)$}.
	
	\noindent 	Since
	\begin{equation*}
		v=v_{\parallel}+v_{\perp}, u=u_{\parallel}+u_{\perp}, \quad and \quad  |v|^2=|v_{\parallel}|^2+|v_{\perp}|^2, |u|^2=|u_{\parallel}|^2+|u_{\perp}|^2,
	\end{equation*}
	in this case we have  $|v_\parallel|\geq \frac{|v|}{\sqrt{2}}\geq |u|$ and
	\begin{equation*}
		\frac{w(v)}{w(v+\frac{2m^{\beta}}{m^{\alpha}+m^{\beta}}u_{\parallel})}(1+|u_{\parallel}|+|v|)^{1-\gamma} \leq C  w(\frac{2m^{\beta}}{m^{\alpha}+m^{\beta}}u_{\parallel}) (1+|u_{\parallel}|+|v|)^{1-\gamma}\leq C  e^{\frac{\bar{c}_1}{4}|v_{\parallel}|^2}.
	\end{equation*}
	By \eqref{cuttoffdelta}, $\frac{w(v)}{w(v+\frac{2m^{\beta}}{m^{\alpha}+m^{\beta}}u_{\parallel})}|k_{M,2}^{\alpha \beta(1)}(v,u_{\parallel})|$ is bounded by
	\begin{align*}
		&\frac{w(v)}{w(v+\frac{2m^{\beta}}{m^{\alpha}+m^{\beta}}u_{\parallel})}\frac{1}{|u_{\parallel}|} e^{-\bar{c}_1[|v_{\parallel}|^2+|u_{\parallel}|^2]}
		\int_{\mathbb{R}^2} e^{-\hat{c}| u_{\perp}+v_{\perp}|^2} \left[|u_{\parallel}|^2 + |u_{\perp}|^2\right]^{\frac{\gamma - 1}{2}} \\
		&\hspace{6cm}\times \{1-\chi_m\}\left(\sqrt{|u_{\parallel}|^2 + |u_{\perp}|^2}\right) \frac{b^{ \alpha \beta}(\theta)}{|\cos \theta|} \, du_{\perp} \\
		&\hspace{0.4cm}\leq C_{m} \frac{1}{|u_{\parallel}|} e^{-\frac{\bar{c}_1}{2}[|v_{\parallel}|^2+|u_{\parallel}|^2]} (1+|v|+|u_{\parallel}|)^{\gamma-1}.
	\end{align*}
	Therefore, we have completed the proof of $\eqref{a1a4main236}_2$ in case A3.

	\noindent {\bf Case A4. $\Big(|v_\perp|\geq \frac{|v|}{\sqrt{2}}\Big)\cap\Big(|u_\parallel|\leq \frac{|v|}{8}\Big)\cap\Big(|u_\perp|\leq \frac{|v|}{8}\Big)$}.
	
	\noindent In this case  it yields that
	\begin{equation*}
		|u_{\perp} + v_{\perp}|\geq \frac{|v|}{4} \geq |u|.
	\end{equation*}
	Then, we get
	\begin{align*}
		&\frac{w(v)}{w(v+\frac{2m^{\beta}}{m^{\alpha}+m^{\beta}}u_{\parallel})}(1+|u_{\parallel}|+|v|)^{1-\gamma} \\
		&\hspace{0.3cm}\leq C  w(\frac{2m^{\beta}}{m^{\alpha}+m^{\beta}}u_{\parallel}) (1+|u_{\parallel}|+|v|)^{1-\gamma}\\
		&\hspace{0.3cm}\leq C  e^{\frac{\bar{c}_1}{4}|u_{\perp} + v_{\perp}|^2}.
	\end{align*}
	Hence, $\frac{w(v)}{w(v+\frac{2m^{\beta}}{m^{\alpha}+m^{\beta}}u_{\parallel})}|k_{M,2}^{\alpha \beta(1)}(v,u_{\parallel})|$ can be bounded by
	\begin{align*}
		&\frac{w(v)}{w(v+\frac{2m^{\beta}}{m^{\alpha}+m^{\beta}}u_{\parallel})}\frac{1}{|u_{\parallel}|} e^{-\bar{c}_1[|v_{\parallel}|^2+|u_{\parallel}|^2]}
		\int_{\mathbb{R}^2} e^{-\hat{c}| u_{\perp}+v_{\perp}|^2} \left[|u_{\parallel}|^2 + |u_{\perp}|^2\right]^{\frac{\gamma - 1}{2}} \\
		&\hspace{6.5cm}\times \{1-\chi_m\}\left(\sqrt{|u_{\parallel}|^2 + |u_{\perp}|^2}\right) \frac{b^{ \alpha \beta}(\theta)}{|\cos \theta|} \, du_{\perp} \\
		&\hspace{0.5cm}\leq C_{m}\frac{(1+|v|+|u_{\parallel}|)^{\gamma-1}}{|u_{\parallel}|} e^{-\bar{c}_1[|v_{\parallel}|^2+|u_{\parallel}|^2]}
		\int_{\mathbb{R}^2} e^{-\frac{\hat{c}}{2}| u_{\perp}+v_{\perp}|^2} \left[|u_{\parallel}|^2 + |u_{\perp}|^2\right]^{\frac{\gamma - 1}{2}} \\
		&\hspace{6.5cm}\times \{1-\chi_m\}\left(\sqrt{|u_{\parallel}|^2 + |u_{\perp}|^2}\right) \frac{b^{ \alpha \beta}(\theta)}{|\cos \theta|} \, du_{\perp} \\
		&\hspace{0.5cm}\leq C_{m} \frac{1}{|u_{\parallel}|} e^{-\bar{c}_1[|v_{\parallel}|^2+|u_{\parallel}|^2]} (1+|v|+|u_{\parallel}|)^{\gamma-1}.
	\end{align*}
	Therefore, we have completed the proof of $\eqref{a1a4main236}_2$ in Case A4.
	
	\noindent {\bf The proof of \eqref{sofd237e}}.
	By the estimate of $k_{M,2}^{\alpha \beta(2)}(v,u_{\perp},u_{\parallel})$ in \eqref{final222kkk}$(\alpha \neq \beta)$, and
	\begin{equation*}
		\frac{w(v)}{w(v+u_{\perp}+\frac{m^{\beta}-m^{\alpha}}{m^{\alpha}+m^{\beta}}u_{\parallel})} \leq C w(u_{\perp})w(\frac{m^{\beta}-m^{\alpha}}{m^{\alpha}+m^{\beta}}u_{\parallel})\leq C e^{\frac{\bar{c}_2}{2}[|u_{\parallel}|^2+|u_{\perp}|^2]},
	\end{equation*}
	we immediately get 
	\begin{eqnarray*}
		\begin{split}
			\frac{w(v)}{w(v+u_{\perp}+\frac{m^{\beta}-m^{\alpha}}{m^{\alpha}+m^{\beta}}u_{\parallel})}|k_{M,2}^{\alpha\beta(2)}(v,u_{\perp},u_{\parallel})|\leq C_m\, e^{-c(|v|^2+|u_{\parallel}|^2+|u_{\perp}|^2)}.
		\end{split}
	\end{eqnarray*}
	
	\noindent {\bf The proof of \eqref{ZZZzequagezV}}.
	Comparing the two expressions \eqref{socd264} and \eqref{ker1rrr} for $\alpha=\beta$, we observe that they become identical if we interchange the symbols ``$\parallel$'' and ``$\perp$'' in one of the expressions. Thus, the analysis of $Case A_1-Case A_4$ for $k_{M,2}^{\alpha \beta(1)}(v,u_{\parallel})$ is  applicable to that of $k_{M,2}^{\alpha\alpha(2)}(v,u_{\perp})$. We therefore obtain
	\begin{eqnarray}\label{uperplifeq}
		\begin{split}
			\frac{w(v)}{w(v+u_{\perp})}|k_{M,2}^{\alpha\alpha(2)}(v,u_{\perp})|&\leq C_m\,  \frac{(1+|v|+|u_{\perp}|)^{\gamma-1}}{|u_{\perp}|} \, e^{-c(|u_{\perp}|^2+|v_{\parallel}|^2)}.
		\end{split}
	\end{eqnarray}

	\noindent {\bf The proof of \eqref{sec4inL1decayK12}}.
	 By $\eqref{a1a4main236}_2$, it follows that
	\begin{align}\label{dddless111}
		&\int_{\mathbb{R}^3}|k_{M,2}^{\alpha \beta(1)}(v,u_{\parallel})|\frac{w(v)}{w(v+\frac{2m^{\beta}}{m^{\alpha}+m^{\beta}}u_{\parallel})} d u_{\parallel}
		\leq C_{m} \int_{\mathbb{R}^3}\frac{(1+|v|+|u_{\parallel}|)^{\gamma-1}}{|u_{\parallel}|} \, e^{-c\{|u_{\parallel}|^2+|v_{\parallel}|^2\}} d u_{\parallel}.
	\end{align}
	Notice that  $|u_{\parallel}|^2+|v_{\parallel}|^2 \geq \frac{1}{2} (|u_{\parallel}|^2+|u_{\parallel}|\cdot|v_{\parallel}|)$, then  \eqref{dddless111} is bounded by
	\begin{equation*}
		C_{m} \int_{\mathbb{R}^3}\frac{(1+|v|+|u_{\parallel}|)^{\gamma-1}}{|u_{\parallel}|} \, e^{-c\{|u_{\parallel}|^2+|u_{\parallel}|\cdot|v_{\parallel}|\}} d u_{\parallel}.
	\end{equation*}
	Since $|u_{\parallel}|\cdot|v_{\parallel}|=|v\cdot u_{\parallel}|$, and $u_{\parallel}\in \mathbb{R}^3$, we decompose $u_{\parallel}$ by
	\begin{equation*}
		u_{\parallel}=u^{0}_{\parallel}+u^{1}_{\parallel},\qquad  u^{0}_{\parallel} \parallel v, \qquad u^{1}_{\parallel} \perp v,
	\end{equation*}
	where $u^{0}_{\parallel}\in\mathbb{R}^1$, $u^{1}_{\parallel}\in\mathbb{R}^2$. So, it yields that
	\begin{align}
		&C_{m} \int_{\mathbb{R}^1}\frac{(1+|v|+|u_{\parallel}|)^{\gamma-1}}{|u_{\parallel}|} \, e^{-c\{|u_{\parallel}|^2+|u_{\parallel}|\cdot|v_{\parallel}|\}} d u_{\parallel} \notag\\
		&\hspace{0.4cm}\leq C_{m}(1+|v|)^{\gamma-1} \int_{\mathbb{R}^1} \, e^{-c|v\cdot u_{\parallel}|} d u^0_{\parallel}\int_{\mathbb{R}^2} \frac{1}{|u^1_{\parallel}|}e^{-c|u^1_{\parallel}|^2}d u^1_{\parallel}\notag\\
		&\hspace{0.4cm}\leq C_{m}(1+|v|)^{\gamma-1} \int_{\mathbb{R}^1} \, e^{-c\{|v| u^0_{\parallel}\}} d u^0_{\parallel}\notag\\
		&\hspace{0.4cm}\leq C_{m}(1+|v|)^{\gamma-2} \int_{\mathbb{R}^1} \, e^{-c\{|v| u^0_{\parallel}\}} d \{|v|u^0_{\parallel}\}\notag\\
		&\hspace{0.4cm} \leq C_{m}(1+|v|)^{\gamma-2}.\label{lstintg270}
	\end{align}
	Therefore, we complete the proof of $\eqref{sec4inL1decayK12}_1$. By \eqref{uperplifeq}, applying‌ an analogous procedure to \eqref{dddless111}-\eqref{lstintg270} ‌with $u_{\parallel}$ and $u_{\perp}$ interchanged‌, we can prove
	\begin{eqnarray*}
		\begin{split}
			\int_{\mathbb{R}^3}\big|k_{M,2}^{\alpha \alpha(2)}(v,u_{\perp})\big|\frac{w(v)}{w(v+u_{\perp})} d u_{\perp} \leq C_m \left \langle v \right\rangle^{\gamma-2}.
		\end{split}
	\end{eqnarray*}

		In summary, this completes the proof of Lemma \ref{LeMK2ker4444}.
	
	\noindent {\bf The proof of remark \ref{kernalcompwithsinglek}}.
	
	\noindent  We change the variable $v+\frac{2m^{\beta}}{m^{\alpha}+m^{\beta}}u_{\parallel}\rightarrow v^*$ to get
	\begin{equation*}
		\int_{\mathbb{R}^3}k_{M,2}^{\alpha \beta(1)}(v,u_{\parallel})\frac{w(v)}{w(v+\frac{2m^{\beta}}{m^{\alpha}+m^{\beta}}u_{\parallel})}g^{\alpha}(v+\frac{2m^{\beta}}{m^{\alpha}+m^{\beta}}u_{\parallel}) d u_{\parallel}=	\int_{\mathbb{R}^3}\tilde{k}_{M,2}^{\alpha \beta(1)}(v,v^*)\frac{w(v)}{w(v^*)}g^{\alpha}(v^*) d v^*.
	\end{equation*}
	Then, by \eqref{bdskm21} and \eqref{contrlofP}, adjusting the constant $c$, it holds that
	\begin{eqnarray}\label{rek4444}
		\begin{split}
			\frac{w(v)}{w(v_*)}|\tilde{k}_{M,2}^{\alpha \beta(1)}(v,v_*)|\leq C_{m} \frac{(1+|v|+|v_*-v|)^{\gamma-1}}{|v-v_*|} \, e^{-c\{|v-v_*|^2+|v_{\parallel}|^2\}}.
		\end{split}
	\end{eqnarray}
	Since
	\begin{equation*}
		v+v_*=2v+\frac{2m^{\beta}}{m^{\alpha}+m^{\beta}}u_{\parallel}, \qquad  (|v_{\parallel}|^2+|u_{\parallel}|^2) \thicksim (|2v_{\parallel}+\frac{2m^{\beta}}{m^{\alpha}+m^{\beta}}u_{\parallel}|^2+|u_{\parallel}|^2),
	\end{equation*}
	Adjusting the constant $c$, the exponential in \eqref{rek4444} can be replaced by
	\begin{equation*}
	 e^{-c\{|v-v_*|^2+|2v_{\parallel}+\frac{2m^{\beta}}{m^{\alpha}+m^{\beta}}u_{\parallel}|^2\}}.
	\end{equation*}
	Furthermore, we write $2v_{\parallel}+\frac{2m^{\beta}}{m^{\alpha}+m^{\beta}}u_{\parallel}$ as
	\begin{equation*}
		\big|2v_{\parallel}+\frac{2m^{\beta}}{m^{\alpha}+m^{\beta}}u_{\parallel}\big|=\frac{\big|\left \langle 2v+\frac{2m^{\beta}}{m^{\alpha}+m^{\beta}}u_{\parallel}, \frac{2m^{\beta}}{m^{\alpha}+m^{\beta}}u_{\parallel} \right \rangle_{v}\big|}{|\frac{2m^{\beta}}{m^{\alpha}+m^{\beta}}u_{\parallel}|}=\frac{|\left \langle v+v_*, v-v_* \right \rangle_{v}|}{|v-v_*|}=\frac{||v|^2-|v_*|^2|}{|v-v_*|}.
	\end{equation*}
	So we get
	\begin{equation*}
		e^{-c\{|v-v_*|^2+|v_{\parallel}|^2\}} \thicksim e^{-c\{|v-v_*|^2+\frac{||v|^2-|v_*|^2|^2}{|v-v_*|^2}\}}.
	\end{equation*}
	Combine with
	\begin{equation*}
		(1+|v|+|v_*-v|) \thicksim (1+|v|+|v_*|),
	\end{equation*}
	thus we complete the proof of remark \ref{kernalcompwithsinglek} .
\end{proof}

\noindent {\bf The proof of remark \ref{rek26262}}.
\begin{proof}
\noindent By \eqref{lem431123512}, 
\begin{eqnarray}\label{v*gg}
	\begin{split}
		\prescript{(\alpha \neq \beta)}{}{K}_{M,2,w}^{\alpha,1-\chi,(2)}\mathbf{g}(v)&=:\int_{\mathbb{R}^3}\int_{\mathbb{R}^2}\frac{w(v)}{w(v+u_{\perp}+\frac{m^{\beta}-m^{\alpha}}{m^{\alpha}+m^{\beta}}u_{\parallel})}k_{M,2}^{\alpha \beta(2)}(v,u_{\perp},u_{\parallel})\\ &\hspace{5cm}\times g^{\beta}(v+u_{\perp}+\frac{m^{\beta}-m^{\alpha}}{m^{\alpha}+m^{\beta}}u_{\parallel}) d u_{\perp} d u_{\parallel}, \qquad \alpha \neq \beta.
	\end{split}
\end{eqnarray}
From \eqref{impykba}, \eqref{CE2ESYCGV}, it yields 
\begin{equation*}
2|u_{\parallel}|^{-2}	d u_{\perp} d u_{\parallel}= du d \omega = dv_* d \omega.
\end{equation*}
Note that 
\begin{equation*}
	v_*'=v+u_{\perp}+\frac{m^{\beta}-m^{\alpha}}{m^{\alpha}+m^{\beta}}u_{\parallel},
\end{equation*}
we need to  calculate the Jacobian determinant for the change of variables $v_* \rightarrow 	v_*'$. The Jacobian matrix is
\begin{equation*}
	\mathbf{J} =\Big(   
		   \frac{\partial v'_*}{\partial v_*} \Big)_{3 \times 3}
=
		 \Big(\mathbf{I}-\frac{2m^{\alpha}}{m^\alpha+m^\beta}\omega \omega^{\rm{T}}\Big)_{3 \times 3},
\end{equation*}
where $\mathbf{I}$ is a $3\times3$ identity matrix, and $\omega=(\omega_1,\omega_2,\omega_3)^{\rm{T}}\in \mathbb{S}^2$ is a column vector. Its Jacobian determinant is
\begin{align*}
	\det\mathbf{J} =& \det \Big(\mathbf{I}-\frac{2m^{\alpha}}{m^\alpha+m^\beta}\omega \omega^{\rm{T}}\Big) \\
	=& \left|\begin{array}{ccc}
		1-\frac{2m^{\alpha}}{m^\alpha+m^\beta}\omega_1^2 & -\frac{2m^{\alpha}}{m^\alpha+m^\beta}\omega_1\omega_2  & -\frac{2m^{\alpha}}{m^\alpha+m^\beta}\omega_1\omega_3 \\
		-\frac{2m^{\alpha}}{m^\alpha+m^\beta}\omega_2\omega_1 & 1-\frac{2m^{\alpha}}{m^\alpha+m^\beta}\omega_2\omega_2 & -\frac{2m^{\alpha}}{m^\alpha+m^\beta}\omega_2\omega_3 \\
		-\frac{2m^{\alpha}}{m^\alpha+m^\beta}\omega_3\omega_1 & -\frac{2m^{\alpha}}{m^\alpha+m^\beta}\omega_3\omega_2  & 1-\frac{2m^{\alpha}}{m^\alpha+m^\beta}\omega_3\omega_3 
	\end{array}\right|\\
	=&(1-\frac{2m^{\alpha}}{m^\alpha+m^\beta}\omega_1^2)\Big[(1-\frac{2m^{\alpha}}{m^\alpha+m^\beta}\omega_2^2)(1-\frac{2m^{\alpha}}{m^\alpha+m^\beta}\omega_3^2)-(\frac{2m^{\alpha}}{m^\alpha+m^\beta})^2\omega_2^2\omega_3^2\Big]\\
	&+\frac{2m^{\alpha}}{m^\alpha+m^\beta}\omega_1\omega_2\Big[-\frac{2m^{\alpha}}{m^\alpha+m^\beta}\omega_1\omega_2(1-\frac{2m^{\alpha}}{m^\alpha+m^\beta}\omega^2_3)-(\frac{2m^{\alpha}}{m^\alpha+m^\beta})^2\omega_1\omega_2\omega^2_3\Big]
	\\&-\frac{2m^{\alpha}}{m^\alpha+m^\beta}\omega_1\omega_3\big[(\frac{2m^{\alpha}}{m^\alpha+m^\beta})^2\omega_1\omega^2_2\omega_3+\frac{2m^{\alpha}}{m^\alpha+m^\beta}\omega_1\omega_3(1-\frac{2m^{\alpha}}{m^\alpha+m^\beta}\omega_2^2)\big]
	\\=& 1-\frac{2m^{\alpha}}{m^\alpha+m^\beta}\omega^2_1-\frac{2m^{\alpha}}{m^\alpha+m^\beta}\omega^2_2-\frac{2m^{\alpha}}{m^\alpha+m^\beta}\omega^2_3,\\
	=& \frac{m^{\beta}-m^{\alpha}}{m^\alpha+m^\beta} \neq 0.
\end{align*}
We denote $u^*=v_*'-v$, then we write \eqref{v*gg} as 
\begin{eqnarray}\label{u*gg}
	\begin{split}
		\prescript{(\alpha \neq \beta)}{}{K}_{M,2,w}^{\alpha,1-\chi,(2)}\mathbf{g}(v)&=:\int_{\mathbb{R}^3}\int_{\mathbb{R}^2}\frac{\frac{1}{2}|u_{\parallel}|^{2}w(v)}{w(v+u^*)}k_{M,2}^{\alpha \beta(2)}(v,u_{\perp},u_{\parallel}) g^{\beta}(v+u^*) d u^* d \omega, \qquad \alpha \neq \beta.
	\end{split}
\end{eqnarray}
According to the energy conservation equality
\begin{equation*}
	|v|^2+|v_*|^2=|v'|^2+|v_*'|^2,
\end{equation*}
and the fact that
\begin{equation*}
	(1+|v|+|u_*-v|) \thicksim (1+|v|+|u_*|), \quad   (1+|v|+|v_*-v|) \thicksim (1+|v|+|v_*|),
\end{equation*}
we have 
\begin{equation}\label{zhiseee}
	(|v|^2+|u|^2) \geq c (|v'|^2+|v_*'|^2+|v|^2)\geq c (|v|^2+|u^*|^2).
\end{equation}
Let 
\begin{equation*}
	\hat{k}_{M,2}^{\alpha \beta(2)}(v,u^*)= C e^{-\frac{c}{2}(|v|^2+|u^*|^2)},
\end{equation*}
we note that
\begin{equation*}
	\frac{|u_{\parallel}|^{2}w(v)}{w(v+u^*)} \leq C e^{\frac{c}{2}(|v|^2+|u^*|^2)},
\end{equation*}
by \eqref{sofd237e} and \eqref{zhiseee}, it yields
\begin{equation*}
	|k_{M,2}^{\alpha \beta(2)}(v,u_{\perp},u_{\parallel})|	\leq C e^{-c(|v|^2+|u^*|^2)}.
\end{equation*}
Therefore
\begin{equation*}
\frac{w(v)}{w(v+u^*)}|k_{M,2}^{\alpha \beta(2)}(v,u_{\perp},u_{\parallel})|	\leq \hat{k}_{M,2}^{\alpha \beta(2)}(v,u^*).
\end{equation*}
Then,
\eqref{u*gg} is bounded by
\begin{eqnarray*}
	\begin{split}
		\big|\prescript{(\alpha \neq \beta)}{}{K}_{M,2,w}^{\alpha,1-\chi,(2)}\mathbf{g}(v)\big| &\leq \int_{\mathbb{R}^3}\int_{\mathbb{R}^2}\hat{k}_{M,2}^{\alpha \beta(2)}(v,u^*) g^{\beta}(v+u^*) d u^* d \omega \\
		&\leq 4 \pi \int_{\mathbb{R}^3}\hat{k}_{M,2}^{\alpha \beta(2)}(v,u^*) g^{\beta}(v+u^*) d u^*  .
	\end{split}
\end{eqnarray*}
This completed the proof of remark \ref{rek26262} .
\end{proof}

We have now completed the proof of the key Lemma \ref{LeMK2ker4444}, and proceed to present the $L^{\infty}_{x,v}$-estimates.

Let $\mathbf{h}_R $ is defined in \eqref{littleh3.2} and $\upsilon^\alpha(t,x,v)$ is defined in \eqref{ffqzdf}. We could derive to
\begin{gather}\label{4.14A}
	\begin{split}
		\partial_t  h^\alpha_R & + v\cdot\nabla_{x} h^\alpha_R  + \frac{1}{\varepsilon} \upsilon^\alpha h^\alpha_R=  \mathcal{R}^\alpha_1+ \mathcal{R}^\alpha_2+ \mathcal{R}^\alpha_3 + \frac{1}{\varepsilon}(K_{\chi}^{\alpha,w}+K_{1-\chi}^{\alpha, w })\mathbf{h}_R,
	\end{split}
\end{gather}
where
\begin{gather*}
	\begin{split}
		\mathcal{R}^\alpha_1=& \varepsilon^2  \frac{ w}{\sqrt{\mu^{\alpha}_M}}\sum_{\beta=A,B} Q^{ \alpha\beta }(\frac{\sqrt{\mu^{\alpha}_M}}{w}h^\alpha_R, \frac{\sqrt{\mu^{\beta}_M}}{w} h^\beta_R),\\
		\mathcal{R}_2^\alpha= & \frac{w}{\sqrt{\mu^{\alpha}_M}} \sum_{\beta=A,B}\sum_{i=1}^{5} \varepsilon^{i-1}\Big[ Q^{ \alpha \beta}(F_i^{\alpha},\frac{\sqrt{\mu^{\beta}_M}}{w}h^\beta_R)+Q^{\alpha \beta}(\frac{\sqrt{\mu^{\alpha}_M}}{w}h^\alpha_R, F_i^{\beta})\Big], \\
		%&\hspace{3cm} +Q^{BA}(F_i^{B},\frac{\sqrt{\mu_M}}{w}h^A_R) + Q^{BA}(\frac{\sqrt{\mu_M}}{w}h^B_R,F_i^{A})\Big],\\
		\mathcal{R}^\alpha_3 = & \varepsilon^2 \frac{w}{\sqrt{\mu^{\alpha}_M}}  \Big\{-( \partial_t +v\cdot \nabla_{x} ) F^\alpha_5 \\
		& \vspace{5pt}+\sum_{\beta=A,B} \sum_{ i,j \leq 5 ; i+j\geq 6 } \varepsilon^{i+j-6}\Big[Q^{\alpha \beta }(F_i^{\alpha},F_j^{\beta})+Q^{\alpha\beta }(F_j^{\alpha},F_i^{\beta})\Big]\Big\},
	\end{split}
\end{gather*}

The characteristic curve of \eqref{4.14A} through $(t,x)$ is defined as the solution of the equation
\begin{gather*}
	\frac{d}{ds} X(s)=v, \quad X|_{s=t}= x, \quad{\rm i.e.} ~~ X(s)= x-v (t-s).
\end{gather*}
Along this characteristics, the equation \eqref{4.14A} gives the expression of $h^\alpha_R$ as
\begin{gather}\label{4.15}
	\begin{split}
		h^\alpha_R(t,x,v) = & \exp \left(-\frac{1}{\varepsilon}\int_{0}^{t}\upsilon^{\alpha}(\tau, X(\tau),v) d\tau\right)h^\alpha_R(0,X(0),v) \\
		& + \int_{0}^{t}\exp \left(-\frac{1}{\varepsilon}\int_{s}^{t}\upsilon^{\alpha}(\tau,  X(\tau),v)   d\tau \right) \sum_{j=1}^3\mathcal{R}^\alpha_j(s, X(s),v)ds\\
		&+ \frac{1}{\varepsilon} \int_{0}^{t}\exp \left(-\frac{1}{\varepsilon}\int_{s}^{t}\upsilon^{\alpha}(\tau,  X(\tau),v)  d\tau \right)  K_{M,w}^{\alpha,\chi} \mathbf{h}_R (s, X(s),v)ds\\
		&+ \frac{1}{\varepsilon} \int_{0}^{t}\exp \left(-\frac{1}{\varepsilon}\int_{s}^{t}\upsilon^{\alpha}(\tau,  X(\tau),v)   d\tau \right)  K_{M,w}^{\alpha,1-\chi} \mathbf{h}_R (s, X(s),v)ds.
	\end{split}
\end{gather}

\begin{proposition}\label{proposition4.3}
	Let $(n^{A}_{\delta},n^{B}_{\delta},\mathbf{u}_{\delta},\theta_{\delta})  \in C([0,\tau] ; H^{s_0}(\mathbb{{R}}^3)) \cap C^{1}([0,\tau] ; H^{s_0-1}(\mathbb{{R}}^3))$  be the  solution of \eqref{EQF0EPSION} which is established in Lemma \ref{cpesystemlem}.
	Let $\mathbf{F}=(F^A_R,F^B_R)^\top$ is the remainder term defined in \eqref{reeqmain}, and $\mathbf{f}_R,\, \mathbf{h}_R$ are defined in \eqref{littlef3.1} and \eqref{littleh3.2}.  %Furthermore, $c_0 >0$ be the constant as in the coercivity
	%		estimate \eqref{Coercivity1.22}.
	Then there exist $\varepsilon_0 >0$ and $C_{\tau}>0$ such that for all $0<\varepsilon < \varepsilon_0$, it holds
	\begin{equation}\label{4.13}
		\begin{split}
			\mathop{sup}_{0\leq t \leq \tau}(\varepsilon^{\frac{3}{2}}\Vert
			\mathbf{h}_R(t)\Vert_{L^{\infty}_{x,v} })
			\leq C_{\tau} \left(\varepsilon^{\frac{3}{2}}\Vert
			\mathbf{h}^{in}_{R}\Vert_{L^{\infty}_{x,v} }+\sup_{0\leq t \leq \tau} \Vert
			\mathbf{f}_R(t)\Vert_{L^{2}_{x,v} } +\varepsilon^{\frac{7}{2}}\right),
		\end{split}	
	\end{equation}
	where the initial data is given by $\mathbf{h}_R(0)=\mathbf{h}^{in}_{R}$.
\end{proposition}
\noindent By Lemma \eqref{LeMK2ker4444}, Remark \ref{kernalcompwithsinglek} and Remark \ref{rek26262}, we can use the proof in \cite{[50]Wu2023JDE} page 446-454.\\

\hspace{14cm}\qedsymbol  \\

\subsection{The hydrodynamic limit}
We now proceed to the proof of Theorem \ref{mainthemCPE}.

\begin{proof} Proposition \ref{L2estimate} and  Proposition \ref{proposition4.3} give us that
	\begin{align*}
		&\frac{d}{dt}\Vert \mathbf{f}_R \Vert_{L^{2}_{x,v} }^{2}+\frac{c_0}{2\varepsilon}\Vert(\mathbf{I}-\mathcal{P}\}\mathbf{f}_R\Vert_{\nu}^2\\
		& \hspace{0.5cm}\leq  C \left(\sqrt{\varepsilon}\Big[\sup_{0 \leq s \leq t} \Vert \mathbf{f}_R(s) \Vert_{L^{2}_{x,v} }+ \Vert\varepsilon^{\frac{3}{2}}\mathbf{h}_{R}^{\rm in}\Vert_{L^{\infty}_{x,v} }+\varepsilon^{\frac{7}{2}}\Big]+1\right) (\Vert \mathbf{f}_R\Vert_{L^{2}_{x,v} }^{2}+\Vert \mathbf{f}_R\Vert_{L^{2}_{x,v} }).		
	\end{align*}
	Gronwall inequality yields
	\begin{equation*}
		\Vert \mathbf{f}_R \Vert_{L^{2}_{x,v} }+1\leq (\Vert \mathbf{f}^{\rm in}_R\Vert_{L^{2}_{x,v} }+1)\exp\Big\{Ct\{2+\sqrt{\varepsilon}(\sup_{0 \leq s \leq t} \Vert \mathbf{f}_R(s) \Vert_{L^{2}_{x,v} }+ \Vert\varepsilon^{\frac{3}{2}}\mathbf{h}_{R}^{\rm in}\Vert_{L^{\infty}_{x,v} })\}\Big\}.
	\end{equation*}
	For bounded $\sup_{0 \leq s \leq t} \Vert \mathbf{f}_R(s) \Vert_{L^{2}_{x,v} }$, and for $\varepsilon$ small, using the Taylor expansion of the
	exponential function in the above inequality, we obtain
	\begin{equation*}
		\Vert \mathbf{f}_R \Vert_{L^{2}_{x,v} } \leq C (\Vert \mathbf{f}^{\rm in}_R\Vert_{L^{2}_{x,v} }+1) \{1+\sqrt{\varepsilon}(\sup_{0 \leq s \leq t} \Vert \mathbf{f}_R(s) \Vert_{L^{2}_{x,v}}+ \Vert\varepsilon^{\frac{3}{2}}\mathbf{h}_{R}^{\rm in}\Vert_{L^{\infty}_{x,v}})\}
	\end{equation*}
	For $0<t\leq\tau$ and small enough $\varepsilon$, it yields that
	\begin{equation*}
		\sup_{0 \leq t \leq \tau} \Vert \mathbf{f}_R(s) \Vert_{L^{2}_{x,v} } \leq C_{\tau}\{1+\Vert \mathbf{f}^{\rm in}_R\Vert_{L^{2}_{x,v} }+\Vert\varepsilon^{\frac{3}{2}}\mathbf{h}_{R}^{\rm in}\Vert_{L^{\infty}_{x,v} }\}
	\end{equation*}
	This completed the proof of Theorem \ref{mainthemCPE}.
\end{proof}

\section{Acoustic Limit}

The acoustic system is the linearization around the uniform equilibrium (1,1,0,1) of this compressible Euler system for mixed fluids. Its fluid fluctuations $(\sigma^A,\sigma^B,\mathbf{u},\theta)$ satisfy the following linear system
\begin{equation}\label{EQF0EPSIONline}
	\left \{
	\begin{array}{lll}
		\partial_t \sigma^{A} +\nabla_x \cdot \mathbf{u}=0, \\[2mm]
		\partial_t \sigma^{B} +\nabla_x \cdot \mathbf{u}=0, \\[2mm]
		\partial_t  \mathbf{u}
		+ \frac{2\nabla \theta}{m^A+m^B}+ \frac{\nabla n}{m^A+m^B}
		=0, \\[2mm]
		\partial_t \theta + \frac{2}{3}  \nabla_x \cdot \mathbf{u} = 0,
	\end{array}
	\right.
\end{equation}
where $n=\sigma^A+\sigma^B$, with initial data
\begin{equation}\label{lindt}
	(\sigma^A,\sigma^B,\mathbf{u},\theta)(0,\mathbf{x}) = (\sigma^{A,\rm in}, \sigma^{B,\rm in}, \mathbf{u}^{\rm in},  \theta^{\rm in}).
\end{equation}

\subsection{The Acoustic System for gas mixture}
We let $w=\nabla \times \mathbf{u}$, and apply operator $\partial_t$ to $\eqref{EQF0EPSIONline}_3$, it holds that
\begin{align}
	\partial_{t} w=0.
\end{align}
We deduce that
\begin{align}
	w=\nabla \times \mathbf{u}^{\rm{in}}=:w^{\rm{in}}.
\end{align}
According to the curl identity
\begin{equation}\label{irrneeded}
	\nabla(\nabla\cdot\mathbf{u})=\Delta \mathbf{u} + \nabla \times(\nabla \times\mathbf{u}).
\end{equation}
We apply operator $\partial_t$ to $\eqref{EQF0EPSIONline}_3$, and substituting $\eqref{EQF0EPSIONline}_1, \eqref{EQF0EPSIONline}_2$, $\eqref{EQF0EPSIONline}_4$ into it, we obtain:
\begin{equation}
	\partial^2_{tt}\mathbf{u}-\frac{10}{3(m^A+m^B)}\Delta\mathbf{u}=\frac{10}{3(m^A+m^B)}(\nabla \times w^{\rm{in}}).
\end{equation}
Since the source term on the right-hand side is known, the equation above can be solved either by the Fourier transform method or by Poisson's formula. The solution is then substituted into $\eqref{EQF0EPSIONline}_1, \eqref{EQF0EPSIONline}_2$ and $\eqref{EQF0EPSIONline}_4$ to obtain $\sigma^A, \sigma^B$ and $\theta$.

Next, we consider its energy estimate by applying operator $\partial^{\tilde{\alpha}}_x$ to \eqref{EQF0EPSIONline}, we have
\begin{equation}\label{partialDXACOU1}
	\left \{
	\begin{array}{lll}
		\partial_t \partial^{\tilde{\alpha}}_x\sigma^{A} +\nabla_x \cdot \partial^{\tilde{\alpha}}_x\mathbf{u}=0, \\[2mm]
		\partial_t \partial^{\tilde{\alpha}}_x\sigma^{B} +\nabla_x \cdot \partial^{\tilde{\alpha}}_x\mathbf{u}=0, \\[2mm]
		\partial_t  \partial^{\tilde{\alpha}}_x\mathbf{u}
		+ \frac{2\nabla \partial^{\tilde{\alpha}}_x\theta}{m^A+m^B}+ \frac{\nabla \partial^{\tilde{\alpha}}_x n}{m^A+m^B}
		=0, \\[2mm]
		\partial_t \partial^{\tilde{\alpha}}_x\theta + \frac{2}{3}  \nabla_x \cdot \partial^{\tilde{\alpha}}_x\mathbf{u} = 0.
	\end{array}
	\right.
\end{equation}
Multiply $\eqref{partialDXACOU1}_1$ by $\partial^{\tilde{\alpha}}_x\mathbf{u}$ and integrate with respect to $x$, yielding:
\begin{equation}\label{3535444}
	\int_{\mathbb{R}^3} (\partial_t  \partial^{\tilde{\alpha}}_x\mathbf{u}
	+ \frac{2\nabla \partial^{\tilde{\alpha}}_x\theta}{m^A+m^B}+ \frac{\nabla \partial^{\tilde{\alpha}}_x n}{m^A+m^B}) \partial^{\tilde{\alpha}}_x\mathbf{u} dx=0.
\end{equation}
We perform integration by parts on the product of $n$ and $\theta$, followed by substituting $\eqref{partialDXACOU1}_1$, $\eqref{partialDXACOU1}_2$, and $\eqref{partialDXACOU1}_4$ into \eqref{3535444} to derive
\begin{equation}\label{fixdt11}
	\partial_t\int_{\mathbb{R}^3}   |\partial^{\tilde{\alpha}}_x\mathbf{u}|^2
	+ \frac{3|\partial^{\tilde{\alpha}}_x\theta|^2}{(m^A+m^B)}+ \frac{|\partial^{\tilde{\alpha}}_xn|^2}{2(m^A+m^B)} dx=0.
\end{equation}
On the other hand, classical solutions to compressible Euler equations exist for
only finite time. For any given $\tau>0$ and  given acoustic initial data $(\sigma^{A,\rm in}, \sigma^{B,\rm in}, \mathbf{u}^{\rm in},  \theta^{\rm in})$, we define
\begin{equation}
	\delta_1=\frac{C_1}{\tau}.
\end{equation}
Then, the lifespan of the solutions for the compressible Euler equations have a uniform lower bound
\begin{equation*}
	\tau^{\delta}\geq\frac{C_1}{\delta}>\frac{C_1}{\delta_1}=\tau.
\end{equation*}
Therefore, we can consider the solutions of the compressible Euler system on an
arbitrary finite time interval $[0,\tau]$, for fixed $\delta_1$  in \eqref{fixdt11}.

Next, in order to establish $\frac{\mu^{\alpha}_{\delta}}{\delta}$ as an effective intermediary, we therefore examine its second-order variation with respect to $\delta$.

\subsection{The second-order perturbation in $\delta$ of Euler solutions}
We introduce the variables $(\sigma^A_{\delta,d}, \sigma^B_{\delta,d}, \mathbf{u}_{\delta,d}, \theta_{\delta,d})$, such that
\begin{equation}\label{secondpeterbudt}
	\delta^2 \sigma^{\alpha}_{\delta,d}=n^{\alpha}_{\delta}-1-\delta \sigma^{\alpha}, \quad  \delta^2 \mathbf{u}_{\delta,d} = \mathbf{u}_{\delta} -\delta \mathbf{u},\quad \delta^2 \theta_{\delta,d} = \theta_{\delta} -1-\delta \theta.
\end{equation}
\begin{lema}\label{lem3232inp}
	For any given $\tau>0$.   	Let $(n^A_{\delta}, n^B_{\delta}, \mathbf{u}_{\delta}, \theta_{\delta})$ be the Euler solutions of \eqref{EQF0EPSION} with the initial data \eqref{nonlindt} and $(\sigma^A,\sigma^B,\mathbf{u},\theta)$ be the acoustic solution of \eqref{EQF0EPSIONline} with initial data \eqref{lindt}. Then for all $0<\delta\leq\delta_0$
	and $s\geq 3$ , there exists a constant $C_2 > 0$ depending only on time $\tau$ and the
	$H^{s+1}$ norm of $(\sigma^{A,\rm in}, \sigma^{B,\rm in}, \mathbf{u}^{\rm in},  \theta^{\rm in})$ such that
	\begin{equation}
		\|(\sigma^A_{\delta,d}, \sigma^B_{\delta,d}, \mathbf{u}_{\delta,d}, \theta_{\delta,d})\|_{H^s} \leq C_2,
	\end{equation}
	which means the acoustic system is the linearization near the constant state $(1,1,0,1)$ of the compressible Euler system:
	\begin{equation}\label{rfes31010}
		\sup_{0\leq t \leq \tau} \|(n^A_{\delta}-1-\delta \sigma^A, n^B_{\delta}-1-\delta \sigma^B, \mathbf{u}_{\delta}-\delta \mathbf{u}, \theta_{\delta}-1-\delta \theta)\|_{H^s} \leq C_s \delta^2.
	\end{equation}
	Moreover, for $s\geq3$, the Sobolev embedding theorem
	implies the uniform pointwise estimates of $(\sigma^A_{\delta,d}, \sigma^B_{\delta,d}, \mathbf{u}_{\delta,d}, \theta_{\delta,d})$.
\end{lema}
\begin{proof}
	We rewrite \eqref{secondpeterbudt} as
	\begin{equation*}
		n^{\alpha}_{\delta}=1+\delta \sigma^{\alpha}+\delta^2 \sigma^{\alpha}_{\delta,d}, \quad   \mathbf{u}_{\delta} =\delta \mathbf{u}+\delta^2 \mathbf{u}_{\delta,d},\quad  \theta_{\delta} =1+\delta \theta+\delta^2 \theta_{\delta,d},
	\end{equation*}
	and plug it into \eqref{EQF0EPSION} to get
	\begin{align*}
		&\partial_t \left( \delta\sigma^{A} + \delta^2\sigma^{A}_{\delta,d} \right) + \left( \delta \mathbf{u}+\delta^2 \mathbf{u}_{\delta,d} \right) \cdot \nabla \left( \delta\sigma^{A} + \delta^2\sigma^{A}_{\delta,d} \right) \\
		&\hspace{5.5cm} + \left( 1+\delta \sigma^{A}+\delta^2 \sigma^{A}_{\delta,d} \right) \nabla \cdot \left( \delta \mathbf{u}+\delta^2 \mathbf{u}_{\delta,d} \right) = 0, \\
		&\partial_t \left( \delta\sigma^{B} + \delta^2\sigma^{B}_{\delta,d} \right) + \left( \delta \mathbf{u}+\delta^2 \mathbf{u}_{\delta,d} \right) \cdot \nabla \left( \delta\sigma^{B} + \delta^2\sigma^{B}_{\delta,d} \right) \\
		&\hspace{5.5cm} + \left( 1+\delta \sigma^{B}+\delta^2 \sigma^{B}_{\delta,d} \right) \nabla \cdot \left( \delta \mathbf{u}+\delta^2 \mathbf{u}_{\delta,d} \right) = 0, \\
		&\sum_{\alpha=A,B}\left[m^{\alpha}( 1 + \delta \sigma^{\alpha}+\delta^2 \sigma^{\alpha}_{\delta,d} )\right] \partial_t \left( \delta \mathbf{u}+\delta^2 \mathbf{u}_{\delta,d} \right) \\
		&\hspace{2cm} + \sum_{\alpha=A,B}\left[m^{\alpha}( 1 + \delta \sigma^{\alpha}+\delta^2 \sigma^{\alpha}_{\delta,d} )\right] \left[( \delta \mathbf{u}+\delta^2 \mathbf{u}_{\delta,d} ) \cdot \nabla\right] \left( \delta \mathbf{u}+\delta^2 \mathbf{u}_{\delta,d} \right) \\
		&\hspace{5cm} + \sum_{\alpha=A,B}\left( 1 + \delta \sigma^{\alpha}+\delta^2 \sigma^{\alpha}_{\delta,d} \right) \nabla \left( \delta \theta+\delta^2 \theta_{\delta,d} \right) \\
		&\hspace{5cm} + \sum_{\alpha=A,B}\left( 1 + \delta \theta+\delta^2 \theta_{\delta,d} \right) \nabla \left( \delta \sigma^{\alpha}+\delta^2 \sigma^{\alpha}_{\delta,d} \right) = 0, \\
		&\partial_t \left( \delta \theta+\delta^2 \theta_{\delta,d} \right) + \left( \delta \mathbf{u}+\delta^2 \mathbf{u}_{\delta,d} \right) \cdot \nabla \left( \delta \theta+\delta^2 \theta_{\delta,d} \right) \\
		&\hspace{5.6cm} + \frac{2}{3} \left( 1 + \delta \theta+\delta^2 \theta_{\delta,d} \right) \nabla \cdot \left( \delta \mathbf{u}+\delta^2 \mathbf{u}_{\delta,d} \right) = 0.
	\end{align*}
	The coefficients of $(\sigma^A,\sigma^B,\mathbf{u},\theta)$ in the above equations satisfy the acoustic system \eqref{EQF0EPSIONline}, which can be used to simplify the expression. Consequently, residual terms are at least of order $O(\delta^2)$. For example, the continuity equation simplifies to
	\begin{align*}
		&\delta^2 \Big[\partial_{t}\sigma^{\alpha}_{\delta,d}+( \mathbf{u}+\delta \mathbf{u}_{\delta,d})\cdot \nabla \sigma^{\alpha}+(\delta \mathbf{u}+\delta^2 \mathbf{u}_{\delta,d})\cdot\nabla\sigma^{\alpha}_{\delta,d}\\
		&\hspace{3cm}+(\sigma^{\alpha}+\delta\sigma^{\alpha}_{\delta,d})\nabla\cdot \mathbf{u}+(1+\delta \sigma^{\alpha}+\delta^2 \sigma^{\alpha}_{\delta,d})\nabla\cdot\mathbf{u}_{\delta,d}\Big]=0.
	\end{align*}
	Using \eqref{secondpeterbudt}, the above equation can be written as
	\begin{equation*}
		\partial_{t}\sigma^{\alpha}_{\delta,d}+(\mathbf{u}_{\delta}\cdot\nabla)\sigma^{\alpha}_{\delta,d}+	n^{\alpha}_{\delta}\nabla\cdot\mathbf{u}_{\delta,d}+\delta[\nabla\sigma^{\alpha}\cdot\mathbf{u}_{\delta,d}+(\nabla\cdot\mathbf{u})\sigma^{\alpha}_{\delta,d}]+\nabla\cdot(\sigma^{\alpha}\mathbf{u})=0.
	\end{equation*}
	We can also deduce that $(\sigma^A_{\delta,d}, \sigma^B_{\delta,d}, \mathbf{u}_{\delta,d}, \theta_{\delta,d})$ satisfies the following linear system:
	\begin{align}
		&\partial_{t}\sigma^{A}_{\delta,d}+(\mathbf{u}_{\delta}\cdot\nabla)\sigma^{A}_{\delta,d}+	n^{A}_{\delta}\nabla\cdot\mathbf{u}_{\delta,d}+\delta[\nabla\sigma^{A}\cdot\mathbf{u}_{\delta,d}+(\nabla\cdot\mathbf{u})\sigma^{A}_{\delta,d}]=-\nabla\cdot(\sigma^{A}\mathbf{u}),\notag\\
		&\partial_{t}\sigma^{B}_{\delta,d}+(\mathbf{u}_{\delta}\cdot\nabla)\sigma^{B}_{\delta,d}+	n^{B}_{\delta}\nabla\cdot\mathbf{u}_{\delta,d}+\delta[\nabla\sigma^{B}\cdot\mathbf{u}_{\delta,d}+(\nabla\cdot\mathbf{u})\sigma^{B}_{\delta,d}]=-\nabla\cdot(\sigma^{B}\mathbf{u}),\notag\\
		&\rho_{\delta}\partial_{t}\mathbf{u}_{\delta,d}+\rho_{\delta}(\mathbf{u}_{\delta}\cdot\nabla)\mathbf{u}_{\delta,d}+n_{\delta}\nabla\theta_{\delta,d}+\theta_{\delta}\nabla(\sigma^{A}_{\delta,d}+\sigma^{B}_{\delta,d})\label{erjiedeltardeq} \\
		&\hspace{1.5cm}+\delta[(\partial_{t}\mathbf{u})(m^A\sigma^{A}_{\delta,d}+m^B\sigma^{B}_{\delta,d})+\rho_{\delta}(\mathbf{u}_{\delta}\cdot\nabla)\mathbf{u}+\theta_{\delta,d}\nabla n+(\sigma^{A}_{\delta,d}+\sigma^{B}_{\delta,d})\nabla\theta]\notag\\
		&\hspace{5cm}=-(m^A\sigma^{A}+m^B\sigma^B)\partial_{t}\mathbf{u}-\rho_{\delta}(\mathbf{u}\cdot\nabla)\mathbf{u}-\nabla(n\theta),\notag\\
		&\partial_t\theta_{\delta,d}+(\mathbf{u}_{\delta}\cdot\nabla)\theta_{\delta,d}+\frac{2}{3}\theta_{\delta}\nabla\cdot\mathbf{u}_{\delta,d}+\delta[\nabla\theta\cdot\mathbf{u}_{\delta,d}+\frac{2}{3}(\nabla\cdot\mathbf{u})\theta_{\delta,d}]\notag\\
		&\hspace{8cm}=-\mathbf{u}\cdot\nabla\theta-\frac{2}{3}\theta\nabla\cdot\mathbf{u}.\notag
	\end{align}
	The above system is a linear system about the unknown $(\sigma^A_{\delta,d}, \sigma^B_{\delta,d}, \mathbf{u}_{\delta,d}, \theta_{\delta,d})$. Its coefficients depend on $(n^A_{\delta}, n^B_{\delta}, \mathbf{u}_{\delta}, \theta_{\delta})$ and $(\sigma^A,\sigma^B,\mathbf{u},\theta)$, which are the solutions of the compressible Euler system \eqref{EQF0EPSION} and the acoustic system \eqref{EQF0EPSIONline}. These coefficients are also smooth functions with $H^s$ bounds. We denote $\mathbf{W}_d=(\sigma^A_{\delta,d}, \sigma^B_{\delta,d}, \mathbf{u}_{\delta,d}, \theta_{\delta,d})$, the system can be written into the following vector form
	\begin{equation}
		\partial_t \mathbf{W}_d + \sum_{i=1}^3 \mathbf{A}_i \partial_i \mathbf{W}_d + \mathbf{B} \mathbf{W}_d = \mathbf{G}.
	\end{equation}
	According to classical hyperbolic equation theory, it is required to find a symmetric positive-definite matrix $\mathbf{A}_0$ that symmetrizes the coefficient matrices preceding the spatial derivatives.
	\begin{equation}
		\mathbf{A}_0(\mathbf{U}_0) =	\begin{pmatrix}
			\frac{\theta_{\delta}}{n^A_{\delta}}  & 0 & 0 & 0  \\
			0 & \frac{\theta_{\delta}}{n^B_{\delta}} & 0  & 0 \\
			0  & 0 & \rho_{\delta} \mathbf{I}_3 & 0 \\
			0  & 0 & 0  & \frac{3}{2}\frac{n_{\delta}}{\theta}
		\end{pmatrix},
	\end{equation}
	such that
	\begin{equation}
		\begin{gathered}
			\mathbf{A}_0(\mathbf{U}_0)\mathbf{A}_j(\mathbf{U}_0) =	\begin{pmatrix}
				\frac{u^j_{\delta}\theta_{\delta}}{n^{A}_{\delta}}  & 0 & \theta_{\delta}\mathbf{e}_j^{T}  & 0  \\
				0 & \frac{u^j_{\delta}\theta_{\delta}}{n^{B}_{\delta}} & \theta_{\delta}\mathbf{e}_j^{T}  & 0 \\
				\theta_{\delta}\mathbf{e}_j  & \theta_{\delta}\mathbf{e}_j & \rho_{\delta} u^j_{\delta} \mathbf{I}_3 & n_{\delta} \mathbf{e}_j \\
				0  & 0 & n_{\delta} \mathbf{e}_j^{T}  & \frac{3}{2}\frac{n_{\delta}}{\theta_{\delta}}u^j_{\delta} \\
			\end{pmatrix}
		\end{gathered} \quad j=1,2,3.
	\end{equation}
	Here  $\mathbf{B}$ is a $3 \times 3$ matrix and $\mathbf{G}$ is a $3\times1$ column vector, they are
	consist of $(n^A_{\delta}, n^B_{\delta}, \mathbf{u}_{\delta}, \theta_{\delta})$ and $(\sigma^A,\sigma^B,\mathbf{u},\theta)$ and their first derivatives. Since $(n^A_{\delta}, n^B_{\delta}, \mathbf{u}_{\delta}, \theta_{\delta})$ have positive upper and lower bounds
	for $t\in[0,\tau]$, we can apply the standard energy
	method of the linear symmetric hyperbolic system to  obtain the following
	energy estimate
	\begin{equation*}
		\frac{d}{dt} \|U_d\|^2_{H^s} \leq C_3 \|U_d\|^2_{H^s} + C_4 \|U_d\|_{H^s},
	\end{equation*}
	where the constants $C_3, C_4$ depend on $\|(n^A_{\delta}, n^B_{\delta}, \mathbf{u}_{\delta}, \theta_{\delta})\|^2_{H^{s+1}}$ and $\|(\sigma^A,\sigma^B,\mathbf{u},\theta)\|^2_{H^{s+1}}$.
	By the
	Gronwall inequality, we conclude that $(\sigma^A_{\delta,d}, \sigma^B_{\delta,d}, \mathbf{u}_{\delta,d}, \theta_{\delta,d})$ is bounded by a constant
	depending on time $\tau$ and the initial data $(\sigma^{A,\rm in}, \sigma^{B,\rm in}, \mathbf{u}^{\rm in},  \theta^{\rm in})$. Therefore, the proof of Lemma \ref{lem3232inp} is completed.
\end{proof}

\subsection{Second-order expansion approximation of the local Maxwellian $\mu^{\alpha}_{\delta}$.}

By estimate \eqref{rfes31010}, we can choose $\delta_2$ sufficiently small such that
for each $0<\delta\leq\delta_2$, $\theta_{\delta}$ satisfies
\begin{equation}
	\theta_M<\theta_{\delta}< 2\theta_M
\end{equation}
where $\theta_M$ is a constant depending on $\delta$. We take
\begin{equation}
	\delta_0=\min\{\delta_1,\delta_2\}
\end{equation}
For each $\varepsilon>0$, we choose $0<\delta\leq\delta_0$ and take the Hilbert expansion of the Boltzmann equation \eqref{Meqeps}
around the local Maxwellian  of the form
\begin{equation}
	\left(
	\begin{array}{ccc}
		F_\varepsilon^{ A}\\[2mm]
		F_\varepsilon^{ B}
	\end{array}
	\right)=
	\sum_{k=0}^{5}\varepsilon^{k}
	\left(
	\begin{array}{ccc}
		F_{k}^{A}\\[2mm]
		F_{k}^{B}
	\end{array}
	\right)  +\varepsilon^3
	\left(
	\begin{array}{ccc}
		F_{R}^{A}\\[2mm]
		F_{R}^{B}
	\end{array}
	\right) .
\end{equation}	
where  we have set
$F^{\alpha}_0=\mu^{\alpha}_{\delta}$ and $F_0 \cdots F_5$ are the first six terms of the Hilbert expansion.  $(n^A_{\delta}, n^B_{\delta}, \mathbf{u}_{\delta}, \theta_{\delta})$ is the smooth solution to the compressible Euler system satisfying condition \eqref{globalMaxwellian}, By Theorem \ref{mainthemCPE}  for sufficiently small $0<\varepsilon\leq\varepsilon_0$, \eqref{zjkytld} yields.

Next, we need to show that $\mu^{\alpha}_{\delta}$ is close to $\mu^{\alpha}_{0}+\delta G^{\alpha}$, where $\mu^{\alpha}_{0}$ is the a global Maxwellian  defined in \eqref{defuoaf}, $G^{\alpha}$ is a acoustic perturbation defined in \eqref{1dflimit}.

\begin{lema}\label{222maxdel}
	Let $\mu^{\alpha}_{\delta}$ be given in \eqref{dyjhlibsxF0}, and $G^{\alpha}$ be defined as \eqref{1dflimit}. The solutions $(n^A_{\delta}, n^B_{\delta}, \mathbf{u}_{\delta}, \theta_{\delta})$ and $(\sigma^A,\sigma^B,\mathbf{u},\theta)$ guarantee that $(\sigma^A_{\delta,d}, \sigma^B_{\delta,d}, \mathbf{u}_{\delta,d}, \theta_{\delta,d})$ is a smooth solution of \eqref{erjiedeltardeq}. Then there exists a small $\delta_0>0$, and a constant $C_{\tau}$ depending on time $\tau$ and initial data $(\sigma^{A,\rm in}, \sigma^{B,\rm in}, \mathbf{u}^{\rm in},  \theta^{\rm in})$, such that for all $0<\delta\leq\delta_0$, we have
	\begin{equation}
		\sup_{0\leq t \leq \tau}\{\|\mu^{\alpha}_{\delta}-\mu^{\alpha}_0-\delta G^{\alpha} \|_{L^{2}_{x,v} }+\|\mu^{\alpha}_{\delta}-\mu^{\alpha}_0-\delta G^{\alpha} \|_{L^{\infty}_{x,v} }\} \leq C_{\tau} \delta^2.
	\end{equation}
\end{lema}
\begin{proof}
	We define the following functions
	\begin{equation}\label{zzzdbbl}
		n^{\alpha,z}_{\delta}=1+z\sigma^{\alpha}+z^2\sigma^{\alpha}_{\delta,d},\quad  \mathbf{u}^{z}_{\delta}=z\mathbf{u}+z^2\mathbf{u}_{\delta,d}, \quad \theta^{z}_{\delta}=1+z\theta+z^2\theta_{\delta,d},
	\end{equation}
	and introduce auxiliary local bi-Maxwellians
	\begin{equation}
		\mu^{\alpha}(z)	= \mu^{\alpha,z}_{\delta}=\frac{n^{\alpha,z}_{\delta}}{[2\pi\theta^z_{\delta}]^{\frac{3}{2}}}\exp \Big\{-\frac{m^{\alpha}|v-\mathbf{u}_{\delta}^z|^2}{2\theta^z_{\delta}}\Big\}.
	\end{equation}
	By \eqref{secondpeterbudt} and \eqref{zzzdbbl}, we observe that $n^{\alpha,\delta}_{\delta}=n^{\alpha}_{\delta},\,\mathbf{u}_{\delta}^{\delta}=\mathbf{u}_{\delta},\,\theta^{\delta}_{\delta}=\theta_{\delta}$. So, $\mu^{\alpha}(\delta)=\mu^{\alpha}_{\delta}$ and it can be expanded by Taylor’s formula as a function of $z$:
	\begin{equation}\label{maxweldeta}
		\mu^{\alpha}(z)	= \mu^{\alpha}(0)+\{\mu^{\alpha}\}'(0)z+\{\mu^{\alpha}\}''(z_*)\frac{z^2}{2},
	\end{equation}
	here $0\leq z_* \leq z$, depending on $(t,x,v)$ and $\delta$. Direct computation shows that
	\begin{align*}
		\{\mu^{\alpha}\}'(z)=&\Big\{\frac{(n^{\alpha,z}_{\delta})'}{n^{\alpha,z}_{\delta}}-\frac{3(\theta^z_{\delta})'}{2\theta^z_{\delta}}+m^{\alpha}(v-\mathbf{u}_{\delta}^z)\cdot\frac{(\mathbf{u}_{\delta}^z)'}{\theta^z_{\delta}}+\frac{m^{\alpha}|v-\mathbf{u}_{\delta}^z|^2(\theta^z_{\delta})'}{2(\theta^z_{\delta})^2}\Big\}\mu^{\alpha,z}_{\delta}\\
		:=&D^{\alpha,z}_{\delta}\mu^{\alpha,z}_{\delta},
	\end{align*}
	where
	\begin{equation*}
		(n^{\alpha,z}_{\delta})'= \sigma^{\alpha}+ 2z \sigma^{\alpha}_{\delta,d}, \quad  (\mathbf{u}^{z}_{\delta})'=\mathbf{u}+2z\mathbf{u}_{\delta,d}, \quad (\theta^{z}_{\delta})'=\theta+2z\theta_{\delta,d}.
	\end{equation*}
	Let $z=0$, we have $[(n^{\alpha,z}_{\delta})',(\mathbf{u}^{z}_{\delta})',(\theta^{z}_{\delta})'](0)=[\sigma^{\alpha},\mathbf{u},\theta]$, which gives that
	\begin{equation}
		\{\mu^{\alpha}\}'(0)	=\{\sigma^{\alpha}+v\cdot\mathbf{u}+(\frac{m^{\alpha}|v|^2-3}{2})\theta\}\mu^{\alpha}_0 = G^{\alpha}(t,x,v).
	\end{equation}
	Take the second derivative with respect to z, we have
	\begin{align*}
		(D^{\alpha,z}_{\delta})'=&\frac{(n^{\alpha,z}_{\delta})''}{n^{\alpha,z}_{\delta}}-\frac{((n^{\alpha,z}_{\delta})')^2}{(n^{\alpha,z}_{\delta})^2}-\frac{3(\theta^z_{\delta})''}{2\theta^z_{\delta}}+\frac{3((\theta^z_{\delta})')^2}{2(\theta^z_{\delta})^2}-m^{\alpha}\frac{|(\mathbf{u}_{\delta}^z)'|^2}{\theta^z_{\delta}}\\
		&+m^{\alpha}(v-\mathbf{u}_{\delta}^z)\cdot \Big\{\frac{(\mathbf{u}_{\delta}^z)''}{\theta^z_{\delta}}-2\frac{(\mathbf{u}_{\delta}^z)'(\theta^z_{\delta})'}{(\theta^z_{\delta})^2}\Big\}\\
		&+m^{\alpha}|v-\mathbf{u}_{\delta}^z|^2\Big\{\frac{(\theta^z_{\delta})''}{2(\theta^z_{\delta})^2}-\frac{((\theta^z_{\delta})')^2}{(\theta^z_{\delta})^3}\Big\}.
	\end{align*}
	By \eqref{zzzdbbl}, $(n^{\alpha,z}_{\delta})'', (\mathbf{u}_{\delta}^z)'', (\theta^z_{\delta})''$ can be written as
	\begin{equation*}
		(n^{\alpha,z}_{\delta})''=2 \sigma^{\alpha}_{\delta,d}, (\mathbf{u}_{\delta}^z)''=2\mathbf{u}_{\delta,d}, (\theta^z_{\delta})''=2\theta_{\delta,d}
	\end{equation*}
	We take $z=\delta$ in \eqref{maxweldeta} to get
	\begin{equation*}
		\mu^{\alpha}_{\delta}=\mu^{\alpha}_0+\delta G^{\alpha}+\{\mu^{\alpha}\}''(\delta_*)\frac{\delta^2}{2}, \qquad 0\leq \delta_* \leq \delta.
	\end{equation*}
	By Lemma \ref{lem3232inp}, $(\sigma^A_{\delta,d}, \sigma^B_{\delta,d}, \mathbf{u}_{\delta,d}, \theta_{\delta,d})$ are uniformly bounded, which gives the uniform pointwise estimates of $(n^{\alpha,z}_{\delta},n^{\alpha,z}_{\delta},\mathbf{u}^{z}_{\delta},\theta^{z}_{\delta})$ and their first and second order derivatives. Furthermore, for $\delta \leq \delta_0$ small enough, $n^{\alpha,z}_{\delta},n^{\alpha,z}_{\delta},\theta^{z}_{\delta}$ have uniform lower bounds. Thus we complete the proof of Lemma \ref{222maxdel} .
\end{proof}
\noindent At the end of this paper, we prove  Theorem \ref{The result for compressible euler limit}.
\begin{proof}
	From \eqref{defGepa}, we deduce that
	\begin{align*}
		&\sum_{\alpha=A,B}	\sup_{0\leq t \leq \tau} \{ \|	G^{\alpha}_{\varepsilon} - G^{\alpha} \|_{L^{2}_{x,v} }+\|	G^{\alpha}_{\varepsilon} - G^{\alpha} \|_{L^{\infty}_{x,v} }\}\\
		&\hspace{1cm} =\sum_{\alpha=A,B}	\sup_{0\leq t \leq \tau} \{ \|	\frac{F^{\alpha}_{\varepsilon}-\mu^{\alpha}_0}{\delta} - G^{\alpha} \|_{L^{2}_{x,v} }+\|	\frac{F^{\alpha}_{\varepsilon}-\mu^{\alpha}_0}{\delta} - G^{\alpha} \|_{L^{\infty}_{x,v} }\}
	\end{align*}
	where
	\begin{equation*}
		\frac{F^{\alpha}_{\varepsilon}-\mu^{\alpha}_0}{\delta}-G^{\alpha}
		=\frac{F^{\alpha}_{\varepsilon}-\mu^{\alpha}_{\delta}}{\delta}+\frac{\mu^{\alpha}_{\delta}-\mu^{\alpha}_0-\delta G^{\alpha}}{\delta}
	\end{equation*}
	Therefore, by Lemma \ref{222maxdel}, it yields that
	\begin{align*}
		&\sum_{\alpha=A,B}	\sup_{0\leq t \leq \tau} \{ \|	G^{\alpha}_{\varepsilon} - G^{\alpha} \|_{L^{2}_{x,v} }+\|	G^{\alpha}_{\varepsilon} - G^{\alpha} \|_{L^{\infty}_{x,v} }\}\\
		&\hspace{0.5cm} =\sum_{\alpha=A,B}	\sup_{0\leq t \leq \tau} \Big\{ \|	\frac{F^{\alpha}_{\varepsilon}-\mu^{\alpha}_{\delta}}{\delta} \|_{L^{2}_{x,v} }+\|	\frac{F^{\alpha}_{\varepsilon}-\mu^{\alpha}_{\delta}}{\delta} \|_{L^{\infty}_{x,v} }\\
		&\hspace{1.5cm}+\|	\frac{\mu^{\alpha}_{\delta}-\mu^{\alpha}_0-\delta G^{\alpha}}{\delta} \|_{L^{2}_{x,v} }+\|	\frac{\mu^{\alpha}_{\delta}-\mu^{\alpha}_0-\delta G^{\alpha}}{\delta} \|_{L^{\infty}_{x,v} }\Big\}\\
		&\hspace{0.5cm} \leq C \{\frac{\varepsilon}{\delta}+\delta\}.
	\end{align*}
	This completes the proof of the acoustic limit result.
\end{proof}

\subsection*{Acknowledgements}
The authors are grateful for the valuable comments made by the referees. This research was
supported by the National Natural Science Foundation of China (Grant 12271356 and 12526581).


\begin{thebibliography}{10}
	\bibitem{[1]Aoki2003JSP} K. Aoki, C. Bardos and C. Takata, \emph{Knudsen layer for gas mixture}.  J. Stat.
	Phys. 112 (2003), 629-655.
	
	%2
	
	%3
	
	\bibitem{[4]BGLJSP 1991}  C. Bardos, F. Golse and F, C.D. Levermore, \emph{Fluid dynamic limits of kinetic equations. I. Formal derivations}. J. Statist. Phys. 63 (1991) no. 1-2, 323-344.
	
	%\bibitem{[5]BGLCPAM1993}  C. Bardos, F. Golse and C.D. Levermore, \emph{ Fluid dynamic limits of kinetic equations. II. Convergence proofs for the Boltzmann equation.} Comm. Pure Appl. Math. 46 (1993) no. 5, 667-753.
	
	\bibitem{[6]BSU1991}  C. Bardos and S. Ukai, \emph{The classical incompressible Navier-Stokes limit of the Boltzmann
		equation}. Math. Models Methods Appl. Sci, 1 (1991), 235-257.
	
	
	
	\bibitem{[7]Bardos2012CPDE} C. Bardos and X.F. Yang, \emph{The Classification of well-posed kinetic boundary layer for hard sphere gas mixtures}. Comm. Partial Diff. Eqs. 37 (2012), 1286-1314.
	
	
	\bibitem{[9]Boltzmann1872AWW} L. Boltzmann, \emph{Weitere Studien \"uber das W\"armegleichgewicht unter Gasmolek\"ulen. Sitzungs} Akad.Wiss.
	Wien 66 (1872), 275-370; translated as: Further studies on the thermal equilibrium of gas molecules. Kinetic theory, vol. 2, 88-174. Pergamon, London, 1966.
	
	%10
	
	\bibitem{[60]Briant2016ARM} M. Briant, E.S. Daus, \emph{The Boltzmann equation for a multi-species
		mixture close to global equilibrium}. Arch. Rational Mech. Anal. 222 (2016) 1367-1443
	
	%12
	
	%\bibitem{[12]Caflisch1980CPAM} R.E. Caflisch, \emph{The fluid dynamic limit of the nonlinear Boltzmann equation}.
	%Comm. Pure Appl. Math. 33(1980) no.5, 651-666.
	
	
	%14
	
	\bibitem{[DYZHD]VPB}R.J.Duan, T.Yang and H.J.Zhao, \emph{The Vlasov-Poisson-Boltzmann system in the whole space: The hard potential case}, J. Differential Equations, 252 (2012), pp. 6356--6386.
	
	\bibitem{[DL]VPB} R.J.Duan and S.Q. Liu, \emph{The Vlasov-Poisson-Boltzmann system for a disparate mass binary mixture}. J Stat Phys (2017) 169:614–684.
	
	%17
	\bibitem{[FQ]ARC} Z.D. FANG and K.L. QI, \emph{From the Boltzmann equation for gas mixture to the two-fluid incompressible hydrodynamic system}. arxiv:2408.03570vl [math.AP].
	%18
	
	\bibitem{[61]Grad1958TG} H. Grad, \emph{Principles of the kinetic theory of gases.}.  In: Handbuch der Physik
	Bd. 12, Thermodynamik der Gase, pp. 205-249. Springer Verlag, Berlin, 1958.
	
	%20
	
	%\bibitem{[ivR]Guo1998} Y. Guo \emph{Smooth irrotational flows in the large to the Euler-Poisson system in $\mathcal{R}^{3+1}$}. Commun. Math.
	%Phys. 195 (1998), 249–265 .
	
	\bibitem{[ivR]Guo2002} Y. Guo \emph{The Vlasov-Poisson-Boltzmann system near Maxwellians}. Comm. Pure Appl. Math.
	55 (2002), no. 9, 1104-1135.
	
	\bibitem{[17]Guo2003Invention} Y. Guo, \emph{The Vlasov-Maxwell-Boltzmann system near Maxwellians}.  Invent. math. 153 (2003), 593-630.
	
	
	%23
	
	\bibitem{[ininp]Guo2010CMP} Y. Guo and J. Jang, \emph{Global Hilbert expansion for the Vlasov-Poisson-Boltzmann system}.
	Comm. Math. Phys. 299 (2010), 469-501.
	
	\bibitem{[iJMP]Guo} Y. Guo, \emph{Boltzmann diffusive limit beyond the Navier-Stokes approximation}.
	Comm. Pure Appl. Math, 59 (2006), 626-687.
	
	\bibitem{[19]Guo2010ARMA} Y. Guo, \emph{Decay and continuity of the Boltzmann equation in bounded domains}. Arch. Rational Mech. Anal. 197 (2010), 713-809
	
	\bibitem{[20]Guo2021ARMA} Y. Guo,  F.M. Huang and Y. Wang, \emph{Hilbert expansion of the Boltzmann equation with specular boundary
		condition in half-space}. Arch. Rational Mech. Anal. 241 (2021), 231-309.
	
	\bibitem{[21]Guo2009KRM} Y. Guo, J. Jang and N. Jiang, \emph{Local Hilbert expansion for the Boltzmann equation}.
	Kinet. Relat. Models. 2 (2009) no. 1, 205-214.
	
	\bibitem{[22]Guo2010CPAM} Y. Guo, J. Jang and N. Jiang, \emph{Acoustic limit for the Boltzmann equation in optimal scaling}. Comm. Pure Appl. Math. 63 (2010) no. 3, 337-361.
	
	%\bibitem{[AD]GuoXIAO} Y. Guo, Q.H. Xiao, \emph{Global Hilbert expansion for the relativistic Vlasov-Maxwell-Boltzmann system}. Comm. Math. Phys., 384 (2021), pp. 341--401.
	
	\bibitem{[23]Hilbert} D. Hilbert, \emph{ Mathematical problems}. Bull. Amer. Math. Soc. (N.S.) 37 (2000) no. 4, 407-436.
	
	%32-36
	
	\bibitem{[29]Jiang2021} N. Jiang, Y.L. Luo and S.J. Tang, \emph{Compressible Euler limit from Boltzmann equation with complete diffusive boundary condition in half-space}, arxiv:3714535 [math.AP]
	
	\bibitem{[30]Jiang2021} N. Jiang, Y.L. Luo and S.J. Tang, 	\emph{Compressible Euler limit from Boltzmann equation with Maxwell reflection boundary condition in half-space}, arxiv:2101.11199 [math.AP]
	
	%39
	
	%\bibitem{[32]Jiang2018IUM} N. Jiang, C.J. Xu and H. Zhao, \emph{Incompressible Navier-Stokes-Fourier limit from the Boltzmann equation:
		%	classical solutions}, Indiana Univ. Math. J. 67 (2018), no. 5, 1817-1855.
	
	\bibitem{[2025X]Jiang} N. Jiang, Y.J. Lei and H.J. Zhao, \emph{On the Vlasov-Poisson-Boltzmann limit of the Vlasov-Maxwell-Boltzmann system}, J. Func. Anal. 287 (2024) 110529.
	
	%\bibitem{[61]Jiang} N. Jiang, Y.L. Luo and S. Tang, \emph{Grad-Caflish type decay estimate of Pseudo-inverse of linearized Boltzmann operator and appliciation to Hilbert expansion of compressible Euler scaling}, arxiv
	
	\bibitem{[inp]LIWANG2023SIAM} F.C. Li and Y.C. Wang,  \emph{Global Euler-Poisson limit to the
		Vlasov--Poisson--Boltzmann system with soft
		potential}. SIAM J. Math. Anal. Vol 55, no. 4, 2887-2916.
	
	%\bibitem{[33]Lachowicz1987MMAS} M. Lachowicz, 	\emph{On the initial layer and the existence theorem for the nonlinear Boltzmann equation}. Math. Methods Appl. Sci. 9 (1987), no. 3, 342-366.
	
	\bibitem{[PG]Liu2004} T.P. Liu and S.H.Yu, \emph{The Green's function and large-time behavior of solutions for one-dimensional Boltzmann equation}. Comm. Pure Appl. Math 57 (2004), 1543-1608.
	
	\bibitem{[PG]Liu2006} T.P. Liu and S.H.Yu, \emph{The Green's function of  Boltzmann equation, 3D waves}. Bull. Inst. Math. Acad. Sin 1 (2006),  1-78.
	
	
	%\bibitem{[lc]TM1986} T. Makino, \emph{On a Local Existence Theorem for the Evolution
		%	Equation of Gaseous Stars}. Patterns and Waves-Qualitative Analysis of
	%Nonlinear Differential Equations. (1986) 459-479.
	
	\bibitem{[37]Maxwell1867JCM} J.C. Maxwell, 	\emph{On the dynamical theory of gases}. Philos. Trans. Roy. Soc. London Ser. A 157 (1867), 49-88. Reprinted in The scientific letters and papers of James Clerk Maxwell, vol. II, 1862-1873. Dover, New York, 1965.
	
	%\bibitem{[62]A.Majda} A. Majda, \emph{Compressible Fluid Flow and Systems of Conservation Laws in Several Space Variables}. Springer Science, vol 53 New York, 1984.
	
	\bibitem{[YU2010]} A. Sotirov and S.H.Yu, \emph{On the Solution of a Boltzmann System for a Gas Mixtures}. Arch. Ration. Mech. Anal, 195 (2010), 675-700.
	
	%\bibitem{[40]Raymond2009BOOK} L. Saint-Raymond, \emph{Hydrodynamic Limits of the Boltzmann Equation}. Lecture Notes in Math. Springer-Verlag, Berlin, 2009.
	
	%54
	
	%\bibitem{[42]Raymond2009JMPA} L. Saint-Raymond and  F. Golse, \emph{The incompressible Navier-Stokes limit of the Boltzmann equation for hard cutoff potentials}. J. Math. Pures Appl. 91 (2009) no. 5, 508-552.
	
	%56
	
	%\bibitem{[DA]Sobook} Y. Sone, \emph{Molecular Gas Dynamics. Theory, Techniques, and Applications}.  Birkh\"auser Boston,
	%Boston, 2007.
	
	%\bibitem{[44]Strain2008ARMA} M. Strain and Y. Guo, \emph{Exponential decay for soft potentials near Maxwellian}. Arch. Rational Mech. Anal. 187 (2008), 287-339.
	
	%59
	%60
	
	%\bibitem{[64]Takata 2001 TTSP} Takata, S., Aoki, K, \emph{  The ghost effect in the continuum limit for a vapor gas mixture around condensed phases: asymptotic analysis of the Boltzmann equation}. Trans.
	%Theory Stat. Phys. 30, 205-237 (2001)
	
	%\bibitem{[65]Takata 2004 PF} Takata, S., \emph{Kinetic theory analysis of the two-surface problem of a vapor-vapor mixture in the continuum limit}. Phys.Fluids 16(16),2182-2188 (2004)
	
	%63
	%64
	
	\bibitem{YJWangSIMA} Y.J. Wang, \emph{The diffusive limit of the Vlasov-Boltzmann system
		for binary fluids}. SIAM J. Math. Anal. 43, No. 1(2011), pp. 253-301.
	
	\bibitem{[49]WangJDE2013} Y.J. Wang, \emph{ Decay of the two-species Vlasov-Poisson-Boltzmann system}. J. Differential Equations 254 (2013) no. 5, 2304-2340.
	
	\bibitem{[51]WuPERP} G.F.Wang, W.K.Wang and T.F. Wu, \emph{Classical solutions to the Boltzmann equations for  gas mixture with unequal molecular masses}. preprint.
	
	\bibitem{[52]WuPERP} Y.C.Wang, T.F. Wu and X.F. Yang, \emph{Diffusive limit of the Boltzmann equations for gas mixture}. preprint.
	
	\bibitem{[50]Wu2023JDE} T.F. Wu and X.F. Yang, \emph{Hydrodynamic limit of Boltzmann equations for gas
		mixture}. J. Differential Equations 377 (2023) 418-468.
	
\end{thebibliography}
\end{document}